\documentclass[12pt, oneside, reqno]{amsart}
\usepackage{amsmath,amssymb,amsthm,graphicx,mathrsfs,url}
\usepackage[usenames,dvipsnames]{color}
\usepackage[margin=3cm]{geometry}
\usepackage[shortlabels]{enumitem}
\usepackage[colorlinks=true,linkcolor=Violet,citecolor=PineGreen]{hyperref}
\usepackage[capitalize]{cleveref}

\usepackage{comment}

\RequirePackage[numbers,sort&compress]{natbib}
\usepackage{etoolbox} % for \csnumexpr
\usepackage{etoc}
\etocsettocstyle{\vskip0.8\baselineskip \begin{center}{\sc{Contents}}\end{center}\vskip0.5\baselineskip \hrule \vskip0.8\baselineskip}{\noindent\rule{\linewidth}{0.5pt}\vskip0.0\baselineskip}
\usepackage{graphicx}
\usepackage{mathrsfs} 
\numberwithin{equation}{section}

\newtheorem{theorem}{Theorem}[section]
\newtheorem{corollary}[theorem]{Corollary}
\newtheorem{proposition}[theorem]{Proposition}
\newtheorem{lemma}[theorem]{Lemma}
\newtheorem{claim}[theorem]{Claim}
\newtheorem{definition}[theorem]{Definition}

\theoremstyle{definition}
\newtheorem{remark}[theorem]{Remark}

\newcommand{\p}[1]{\left(#1 \right)}

\newcommand{\dotp}[1]{\langle #1 \rangle}

\newcommand{\ones}{\mathbf{1}}
\newcommand{\dd}{\mathrm{d}}

\renewcommand{\div}{\mathrm{div}}
\newcommand{\Var}{\mathrm{Var}}

\newcommand{\vol}{\mathrm{vol}}

\DeclareMathOperator*{\supp}{supp}
\DeclareMathOperator*{\dist}{dist}

\newcommand{\B}{\mathrm{B}}
\newcommand{\C}{\mathbb{C}}
\newcommand{\Cf}{\mathrm{C}}

\newcommand{\E}{\mathbf{E}}

\renewcommand{\H}{\mathrm{H}}

\renewcommand{\L}{\mathrm{L}}

\newcommand{\N}{\mathbb{N}}

\renewcommand{\P}{\mathbf{P}}
\newcommand{\Q}{\mathbb{Q}}
\newcommand{\R}{\mathbb{R}}

\newcommand{\T}{\mathbb{T}}
\newcommand{\U}{\mathbf{U}}

\newcommand{\W}{\mathrm{W}}

\newcommand{\Z}{\mathbb{Z}}

\newcommand{\cC}{\mathscr{C}}
\newcommand{\cD}{\mathcal{D}}

\newcommand{\cF}{\mathcal{F}}

\newcommand{\cH}{\mathcal{H}}

\newcommand{\cL}{\mathcal{L}}

\newcommand{\cO}{\mathcal{O}}
\newcommand{\cP}{\mathcal{P}}

\newcommand{\cX}{\mathcal{X}}

\newcommand{\kbf}{\mathbf{k}}

\newcommand{\jbf}{\mathbf{j}}
\newcommand{\mbf}{\mathbf{m}}

\renewcommand{\Re}{\operatorname{Re}}
\renewcommand{\Im}{\operatorname{Im}}

\newcommand{\eps}{\varepsilon}
\renewcommand{\phi}{\varphi}

\renewcommand{\leq}{\leqslant}
\renewcommand{\geq}{\geqslant}

\usepackage{etoolbox} % for \csnumexpr

\newcommand{\tr}{\operatorname{tr}}

\author{Yann Chaubet}
\address{Laboratoire de Math\'ematiques Jean Leray, Nantes Universit\'e, UMR CNRS 6629, 2 rue de la Houssini\`ere, 44322 Nantes Cedex 03, France}
\email{yann.chaubet@univ-nantes.fr}
\author{Vincent Divol}
\address{Center for research in economics and Statistics, UMR CNRS 9194, 5 Avenue Henry Le Chatelier, 91120 Palaiseau, France}
\email{vincent.divol@ensae.fr}

\begin{document}

\title{Minimax spectral estimation of  weighted Laplace operators}

\vspace{-0.7cm}
\begin{abstract}
Given  $n$ i.i.d.~observations,  
we study the problem of estimating the spectrum of weighted Laplace operators of the form $\Delta_f=\Delta + \alpha \nabla \log f\cdot \nabla$, where $f$ is a positive probability density on a known compact $d$-dimensional  manifold without boundary and $\alpha\in \R$ is a hyperparameter. These operators arise as continuum limits of graph Laplacian matrices and provide valuable geometric information on the underlying data distribution. We establish the exact minimax rates of estimation for this problem, by exhibiting two different rates of convergence for eigenfunctions and eigenvalues. When $f$ belongs to a H\"older-Zygmund class $\cC^s$ of regularity $s\geq 2$, the eigenfunctions  can be estimated with respect to the $\L^q$-norm ($q\geq 1$) via plug-in methods at the minimax rate $n^{-\frac{s+1}{2s+d}}$ for $d\geq 3$ (with different rates for $d\leq 2$). Moreover,  eigenvalues can be estimated  at the minimax rate $n^{-\frac{4s}{4s+d}}+n^{-\frac 12}$. In the regime  $s>\frac d4$, we further show that asymptotically efficient estimators exist. 

 We also present a general framework for estimating nonlinear functionals over H\"older-Zygmund spaces, with potential applications to a broad class of statistical problems.
\end{abstract}
\maketitle

%\vspace{-0.7cm}
\setcounter{tocdepth}{1}

\tableofcontents
\setlength{\footskip}{30pt}

\section{Introduction}

\subsection{Spectral methods in data science}
Over the past twenty years, diffusion-based approaches have achieved numerous successes in statistics and machine learning. For example, in sampling theory, discretizations of diffusion processes are used to design Markov chains with a target stationary distribution, see e.g., \cite{dalalyan2012sparse, dalalyan2017theoretical, durmus2017nonasymptotic}. In generative modeling, diffusion-based models learn a diffusion process between Gaussian noise and the (unknown) target distribution by training a neural network, see \cite{yang2023diffusion} for a survey. Diffusion processes are also employed to construct \emph{spectral representations} of datasets, which are routinely used  for dimension reduction \cite{roweis2000nonlinear,belkin2003laplacian,coifman2006diffusion},
 as an orthonormal basis in a regression task \cite{chun2016eigenvector, hacquard2022topologically}, or for clustering with the spectral clustering method \cite{von2007tutorial}.

Given a set of $n$ i.i.d.~observations $X_1,\dots,X_n$ following some distribution $P$, these spectral representations are obtained from the top eigenvectors of the generator $L_n$ of a discrete random walk over the set of observations, $L_n$ often being referred to as the graph Laplacian matrix of the walk. These eigenvectors are known to be related to the long-term behavior of this random walk, and therefore to non-local aspects of the geometry underlying the data distribution. The random walk is defined locally, for instance, by jumping randomly from one data point $X_i$ to another point $X_j$ found within a small neighborhood of $X_i$ of fixed radius. Under appropriate normalization, this random walk formally converges as $n\to\infty$ to a continuous diffusion process whose generator is a weighted Laplace operator  of the form
\begin{equation}\label{eq:def_Delta}
\Delta_{f} = \Delta + \alpha A_{V}\qquad \text{with}\qquad A_{V}=\nabla V\cdot \nabla,
\end{equation} 
where $\Delta$ is the Laplace operator, $f=e^{V}$ is the density of $P$, and $\alpha\in \mathbb{R}$ is a parameter depending on the exact definition of $L_n$. The spectrum of such an operator provides valuable information about the geometry of the underlying metric measure space. For instance, when $P$ is the uniform distribution on a $d$-dimensional Riemannian manifold $(M,g)$, the connected components of $M$, its volume, and its total scalar curvature can all be inferred from the spectrum of $\Delta_{f}$ which in this case is equivalent to the (unweighted) Laplace operator \cite{berger1971spectre}. When $\alpha=d/2-1$, the operator $\Delta_{f}$ is actually related to the Laplace operator on the manifold $M$ endowed with the conformal metric $f g$, whose spectrum contains information about both the sampling density and the intrinsic geometry of the support $M$ of the data distribution. Let us also mention that $\Delta_f$ is conjugated in an appropriate weighted space to the so-called Witten Laplacian \cite{witten1982supersymmetry}, whose spectrum in the limit $\alpha \to \infty$ provides rich topological informations about critical points of $f$, see e.g. \cite{helffer1985puits, nier2005hypoelliptic, dang2021pollicott, le2021small}.

The effectiveness of diffusion methods based on  graph Laplacian matrices $L_n$ has motivated researchers to investigate their theoretical properties, particularly regarding their convergence to the limit generator $\Delta_{f}$. We give below in \Cref{sec:related} a more detailed account of the research on the topic. We begin by discussing a recent work by  García Trillos, Li and Venkatraman \cite{trillos2025minimax} which is the current state of the art. In this paper, the authors study the rate of convergence of graph Laplacian eigenfunctions with respect to the $\H^1$-norm (instead of the classical $\L^2$-norm).  Under a $\Cf^{2+\alpha}$ smoothness assumption for $\alpha>0$, they show that the minimax rate of estimation of a simple eigenfunction with respect to an empirical $\H^1$-norm is up to logarithmic factors $\cO(n^{-\frac{2}{4+d}})$. Note that this rate is exactly the minimax rate for estimating  a $\Cf^2$ density with respect to the $\L^2$-loss. The same rate of estimation holds for the empirical $\L^2$-loss and for eigenvalue estimation. 

This paper raises several questions that we address in this work:
\begin{enumerate}
    \item Can minimax rates of convergence for eigenfunction estimation be derived for more natural norms such as the $\L^2$-norm or the $\L^q$-norm for $q\geq 2$? Can eigenfunctions still be estimated in the presence of eigenvalue multiplicities?
    \item Is the rate $\cO(n^{-\frac{2}{4+d}})$ minimax for the problem of eigenvalue estimation? If not, what is the minimax rate of convergence?
    \item If more regularity is assumed, i.e., if we assume that $f\in \Cf^s$ for $s\geq 2$, can this regularity be leveraged to obtain faster rates of estimation?
\end{enumerate}
 We answer these questions from a minimax perspective by assuming knowledge of the manifold $M$. 
First, we show that the minimax rate of estimation for the eigenspaces of $\Delta_f$  is of order 
 \[ \cO(n^{-\frac{s+1}{2s+d}})\] 
 (for $d\geq 3$) if the loss function is the $\L^q$-angle between subspaces. The rates become parametric up to logarithmic factors  for $d\leq 2$. These rates of estimation coincide with the minimax rate of estimation of a density with respect to the $\H^{-1}$-norm, or, equivalently, to the Wasserstein distance, see \cite{divol2022measure, niles2022minimax, divol2021short}. This result follows from a stability result for the eigenspaces of $\Delta_f$ under perturbation of $f$.

Second, we show that the minimax rate of estimation for eigenvalue estimation is
 \[ \cO(n^{-\frac{4s}{4s+d}}+n^{-\frac 12}).\]
This rate is reminiscent of minimax rates of estimation in functional estimation problems, e.g., for the estimation of the integral of the square of a density, see \cite{bickel1988estimating,birge1995estimation}. Actually, we show that debiaising methods commonly used in this field can be leveraged to improve naive plug-in estimators.

  Note that even for $s=2$, our rates, respectively of order $\cO(n^{-\frac{3}{4+d}})$ and $\cO(n^{-\frac{8}{8+d}}+n^{-\frac 12})$,  are both faster than the rate of $\cO(n^{-\frac{2}{4+d}})$ appearing in \cite{trillos2025minimax}. For eigenfunction estimation, this discrepancy is simply due to the use of an $\L^q$-norm instead of an $\H^1$-norm, see \Cref{remark:H1_rate}. More importantly, our findings show that the eigenvalue rate of estimation  $\cO(n^{-\frac{2}{4+d}})$ appearing in both \cite{cheng2022convergence} and \cite{trillos2025minimax} is not minimax  (at least if $M$ is known): this answers an open question raised in \cite{trillos2025minimax} about the possibility of obtaining faster rates of estimation for eigenvalue estimation than for eigenfunction estimation (see their Remark 2.9). Another key difference between the two papers is the fact that we are able to estimate not only simple eigenvalues, but also eigenvalues with multiplicities, together with the associated eigenspaces.

\subsection{Notation and geometrical context} If $n$ is an integer, we let $[n]=\{1,\dots,n\}$.  All random variables $X_1,X_2,\dots$ are defined on the same standard Borel probability space with a probability measure $\P$, and $\E$ denotes  expectation with respect to $\P$.
 We will identify measures with their densities, and we write $\E_f$ to indicate expectation under the law with density $f$. If $A:E\to F$ is a linear operator between two Banach spaces $E$ and $F$, we denote by $\|A\|_{E\to F}$ the corresponding operator norm. 

We let $(M,g)$ be a smooth compact connected $d$-dimensional Riemannian manifold without boundary. 
For $s\geq 0$, we let $\Cf^s$ be the H\"older space of regularity $s$ on $M$. Let $\mathcal D'(M)$ be the space of distributions on $M$, which is naturally identified with the dual space of $\mathrm C^\infty(M)$ by using the Riemannian volume form of $M$. The latter induces a sesquilinear duality pairing $\dotp{\cdot,\cdot}$ between distributions, whenever it is well defined. In what follows we will consider the operator
$
\Lambda = (1 - \Delta)^{1/2}
$
which acts on  $\mathcal D'(M)$.  For any $t \in \R$ and $1\leq q \leq \infty$, we define
$
\W^{t, q} = \Lambda^{-t} \L^q
$ 
the usual Sobolev space of order $t$ and integrability exponent $q$ on $M$. The Sobolev spaces $\W^{t,2}$ of weight $2$ are simply denoted by $\H^t$. By construction, for any $t,m \in \R$ and $1\leq q \leq \infty$ the map
$
\Lambda^m : \W^{t, q} \to \W^{t - m, q}
$
is an isomorphism. 
Besov spaces $\B^s_{p,q}$ will also be useful; their main properties are summarized in \Cref{app:besov}. In particular, we will  use the identity $\B^s_{2,2}=\H^s$ for $s\in \R$. H\"older-Zygmund spaces are defined as $\cC^s=\B^s_{\infty,\infty}$ for $s\geq 0$. If $s$ is not an integer, these spaces coincide with the standard H\"older spaces $\Cf^s$, with $\Cf^s \hookrightarrow\cC^s$ if $s$ is an integer. 
Note that all function spaces mentioned above are complex Banach spaces. We denote by $\overline u$ the conjugate of a function $u$.

In what follows we assume that $f \in \cC^s$ is a positive density for some $s >  1$. For $\alpha\in \R$ a nonzero number that will be fixed throughout the paper, the operator $\Delta_f$ defined in  \eqref{eq:def_Delta} 
defines an unbounded operator on $\L^2$ with closed domain $\H^2$, see \Cref{sec:spectral_theory_intro} for more details. Note that, since we assume that the manifold $M$ is known, the case $\alpha=0$ is of limited interest. On $\L^2$ there is a weighted scalar product
$$
\dotp{u,v}_f = \int u\, \bar v\, f^\alpha\, \dd \vol_g, \quad u,v \in \L^2,
$$
and the associated Hilbert space is denoted $\L^2(f^\alpha)$ . The weighted Laplace operator is symmetric on $\L^2(f^\alpha)$, since we have
$$
\dotp{ \Delta_{f}\,u, v }_f =- \int \nabla u \, \overline{\nabla v} \, f^\alpha \dd \vol_g, \quad u, v \in \mathrm C^\infty(M).
$$
In particular $\Delta_f$ is nonpositive, and it is a classical result (see \cite[Chapter 3]{helffer2013spectral}) that its spectrum is discrete and accumulating at $-\infty$. We denote by 
$$
0 = \lambda_{0,f} \leq \lambda_{1,f} \leq \cdots \leq \lambda_{\ell,f} \to \infty
$$
the spectrum $\sigma(-\Delta_f)=\{\lambda_{\ell,f}:\ \ell\geq 0\}$ of $-\Delta_{f}$. Note that with our convention,  it is possible that $\lambda_{\ell,f}=\lambda_{\ell',f}$ for some $\ell < \ell'$.

\subsection{Minimax estimation of eigenfunctions}

We estimate the eigenfunctions of the operator $\Delta_{f}$  through  a  plug-in approach:  we use as estimators the eigenfunctions of the operator $\Delta_{\hat f}$, where $\hat f$ is an appropriate density estimator of $f$. To account for possible eigenvalue multiplicities, we  measure the proximity between spectral projectors associated with $\Delta_{f}$ and $\Delta_{\hat f}$ rather than between eigenfunctions.

Let $\Upsilon$ be a simple loop in $\mathbb C$, with counterclockwise orientation, bounding a compact domain $\Omega \subset \mathbb C$. {We} assume that 
$$
\partial \Omega \cap \sigma(-\Delta_{f}) = \emptyset.
$$
The spectral projector $\Pi_{\Upsilon,f}$ of $\Delta_{f}$ associated with $\Upsilon$ is given by
$$
\Pi_{\Upsilon,f} = \frac{1}{2 \pi i} \int_{\Upsilon} (z + \Delta_{f})^{-1} \dd z.
$$
For any $q\geq 2$, if $\Pi$ and $\Pi'$ are two projection operators that are bounded $\L^2\to \L^q$ and $\L^q \to \L^q$ 
we set
$$
D_{q}(\Pi,\Pi') = \left\| (1 - \Pi') \Pi\right\|_{\L^2 \to \L^q} + \left\| (1 - \Pi) \Pi'\right\|_{\L^2 \to \L^q}.
$$
If $f,h\in \cC^s$ are such that $\Pi_{f, \Upsilon}$ and $\Pi_{h, \Upsilon}$ are well-defined, 
 $D_q(\Pi_{f, \Upsilon}, \Pi_{h, \Upsilon})$ is a measure of the $\L^q$-angle between the finite-dimensional eigenspaces corresponding to the eigenvalues in the domain $\Omega$, respectively for $-\Delta_{f}$ and $-\Delta_{h}$.

Let  $L, \delta > 0$, and $\Upsilon\subset \C$ be a simple loop such that 
\begin{equation}\label{eq:rezl}
\sup_{z \in \Upsilon}|\Re z| < L.
\end{equation}
For $s>2$, we define the statistical model 
$
\mathcal P_s(\Upsilon, \delta, L) \subset \mathscr C^s
$
as follows. We say that a function $f \in \mathscr C^s$ lies in $\mathcal P_s(\Upsilon, \delta, L)$ if $\Upsilon$ encircles at least one eigenvalue of $-\Delta_f$, with
\[\dist(\sigma(-\Delta_f), \Upsilon) > \delta, \quad \inf f > \delta \quad \text{and} \quad \|f\|_{\mathscr C^s} < L.
\]
For  $s=2$, we   replace the Hölder-Zygmund norm $\cC^2$ by the stronger Hölder norm $\Cf^2$ in the definition of $\mathcal P_2(\Upsilon, \delta, L)$.

\begin{theorem}[Estimation of eigenvectors, upper bound]\label{thm:cv_spectral_eigenfunction}
Let $s\geq 2$, $L,\delta>0$, and $q\geq 2$. There exists $C$ such that the following holds. For any simple loop $\Upsilon$ in $\mathbb C$ satisfying \eqref{eq:rezl} with counterclockwise orientation and any $n\geq 1$, there exists an estimator 
$\widehat \Pi_n : M^n \to \mathcal L(\L^2)$ 
such that for each $f \in \mathcal P_s(\Upsilon, \delta, L)$ we have
\begin{equation}\label{cv_spectral_eigenfunction}
\E_{f^{\otimes n}}[D_q(\widehat \Pi_n, \Pi_{\Upsilon,f})] \leq C \begin{cases}
n^{-\frac 12} &\text{if $d=1$,}\\
(n/\log n)^{-\frac 12} &\text{if $d=2$,}\\
n^{-\frac{s+1}{2s+d}} &\text{if $d\geq 3$.}
\end{cases}
\end{equation}
\end{theorem}

\Cref{thm:cv_spectral_eigenfunction} relies on a perturbation bound (see \Cref{prop:eigenfunctionsperturbative}), which relates the distance between two operators $\Delta_{f}$ an $\Delta_{h}$ to the $\W^{-1,p}$-norm between the densities $f$ and $h$.  For densities bounded away from zero and infinity, the  $\W^{-1,p}$-norm  is known to be equivalent to the Wasserstein distance between the corresponding measures. Minimax estimators with respect to the Wasserstein distance have been proposed on a manifold in \cite{divol2022measure} (see also \cite{ niles2022minimax, divol2021short}), and attain the rate  given in \eqref{cv_spectral_eigenfunction}, from which the rate of estimation for the eigenfunctions follows.
 
\begin{remark}\label{rem:identify_loop}
The estimator provided by \Cref{thm:cv_spectral_eigenfunction} reconstructs the spectral projector 
$\Pi_{\Upsilon,f}$ for any contour $\Upsilon$ isolating a cluster of eigenvalues. 
In practice, however, one is often interested in estimating the eigenspace associated with a specific 
eigenvalue $\lambda_{\ell,f}$, which requires constructing from the data a contour $\Upsilon$ enclosing 
$\lambda_{\ell,f}$ and no other eigenvalue. 
 Assume that $\lambda_{\ell,f}$ satisfies the spectral gap condition
\[
\mathrm{dist}\bigl(\lambda_{\ell,f},\, \sigma(-\Delta_f)\setminus\{\lambda_{\ell,f}\}\bigr) > \delta.
\]
Under this assumption such a contour can be identified from the data. 
For instance, \cite{garcia2020error} construct an estimator $\hat\lambda_\ell$ that, with high probability, 
lies within $\delta/4$ of $\lambda_{\ell,f}$. 
Consequently, the circle $\hat\Upsilon$ of radius $\delta/2$ centered at $\hat\lambda_\ell$ encloses 
$\lambda_{\ell,f}$ alone and remains at distance at least $\delta/4$ from the other eigenvalues. 
Applying our estimator with this data-driven contour $\hat\Upsilon$ yields an estimator of the 
$\ell$-th eigenspace with the same convergence rates as in \Cref{thm:cv_spectral_eigenfunction}. 
The same observation applies to the problem of eigenvalue estimation.
\end{remark}

We provide a matching minimax lower bound to \Cref{thm:cv_spectral_eigenfunction}, as follows.

\begin{theorem}[Estimation of eigenvectors, lower bound]\label{thm:lowerbound_eigenfunction}
Let $s\geq 2$ and $q\geq 2$.  Then there exist a constant $c$  and   a  simple loop $\Upsilon$ in $\mathbb C$ satisfying \eqref{eq:rezl} with counterclockwise orientation such that for $L>0$ large enough and $\delta>0$ small enough we have for any $n\geq 1$ 
\begin{equation}
\inf_{\widehat \Pi_n} \sup_{f\in \cP_s(\Upsilon, \delta, L)} \E_{f^{\otimes n}}[D_2(\widehat \Pi_n, \Pi_{\Upsilon,f})]
 \geq c \begin{cases}
n^{-\frac 12} &\text{if $d\leq 2$,}\\
n^{-\frac{s+1}{2s+d}} &\text{if $d\geq 3$,}
\end{cases}
\end{equation}
where the infimum is taken over all estimators $\widehat \Pi_n : M^n \to \mathcal L(\L^2)$.
\end{theorem}

In practical situations where the manifold is unknown to the practitioner, the empirical $\L^2$-loss  between eigenfunctions is used as a surrogate to the $\L^2$-loss. It is therefore more interesting to provide rates of estimation for 
\begin{equation}
   \p{ \frac{1}{n} \sum_{i=1}^n |\phi(X_i)-\widehat \phi(X_i)|^2}^{\frac 12},
\end{equation}
where $\phi$ is an eigenfunction and $\widehat \phi$ an estimator. 
The empirical risk is trivially upper bounded by the $\L^\infty$-norm $\|\phi-\widehat \phi\|_{\L^\infty}$, which we are actually also able to bound, at the price of deteriorating the exponent $\frac{s+1}{2s+d}$ by an arbitrary small number $\eps>0$. 
See \Cref{rem:p=infty} for further details.

\subsection{Minimax estimation of eigenvalues}

We detail in this section the main results we obtained regarding the estimation of the eigenvalues of $-\Delta_f$. Let  $s \geq 2$, $\delta, L > 0$ and $\Upsilon $ be a  simple loop in $\mathbb C$ satisfying \eqref{eq:rezl} with counterclockwise orientation. For $f \in \mathcal P_s(\Upsilon, \delta, L)$ we denote by
\begin{equation}\label{eq:def_mu_upsilon}
\mu_{\Upsilon, f} = \frac 1{\sharp(\Omega \cap \sigma(-\Delta_f))}\sum_{\lambda \in \Omega \cap \sigma(-\Delta_f)} \lambda,
\end{equation}
the average of the eigenvalues of $-\Delta_f$ that are enclosed by $\Upsilon$. Note that this can be written
$$
\mu_{\Upsilon, f} = -\frac 1{\sharp(\Omega \cap \sigma(-\Delta_f))} \tr  (\Delta_f \Pi_{\Upsilon,f}).
$$
Then we have the following result, which tells us that we can estimate $\mu_{\Upsilon, f}$ at a rate of convergence of order $\cO(n^{-\frac{4s}{4s+d}}+n^{-\frac 12})$.

\begin{theorem}[Estimation of eigenvalues, upper bound]\label{thm:cv_eigenvalue}
Let $s\geq 2$ and $L,\delta>0$. There exists $C$ such that the following holds. For any simple loop $\Upsilon$ in $\mathbb C$ satisfying \eqref{eq:rezl} with counterclockwise orientation and any $n\geq 1$, there exists an estimator  $\widehat \mu_n : M^n \to \R$ such that for each $f \in \cP_s(\Upsilon,\delta,L)$ we have
\begin{equation}\label{cv_eigenvalue}
\E_{f^{\otimes n}}\left[|\widehat \mu_n - \mu_{\Upsilon, f}|\right] \leq C \begin{cases}
n^{-\frac{4s}{4s+d}} &\text{if $s\leq \frac d4$,}\\
n^{-\frac 12} &\text{if $s>\frac d4$.}
\end{cases}
\end{equation}
Moreover, if $s>\frac d4$, the estimator $\widehat \mu_n$ is asymptotically efficient.
\end{theorem}
We also obtain a matching minimax lower bound.

\begin{theorem}\label{thm:lowerbound_eigenvalue}
Let $s\geq 2$. Then there exist a constant $c$,  a  simple loop $\Upsilon$ in $\mathbb C$ satisfying \eqref{eq:rezl} with counterclockwise orientation and a $d$-dimensional manifold $M$ (equal to a dilatation of the flat torus) such that for $L>0$ large enough and $\delta>0$ small enough we have for any $n\geq 1$
\begin{equation}
\inf_{\widehat \mu_n} \sup_{f\in \cP_s(\Upsilon,\delta,L)} \E_{f^{\otimes n}}\left[|\widehat \mu_n - \mu_{\Upsilon, f}|\right]
 \geq c \begin{cases}
n^{-\frac{4s}{4s+d}} &\text{if $s\leq \frac d4$,}\\
n^{-\frac 12} &\text{if $s>\frac d4 $.}
\end{cases}
\end{equation}
where the infimum is taken over all estimators $\widehat \mu_n : M^n \to \R$.
\end{theorem}

\begin{remark}
In fact, our methods allow us to estimate the quantity
$$
\tr \chi(-\Delta_f) = \sum_{\lambda \in \sigma(-\Delta_f)} \chi(\lambda)
$$
for any $\chi \in \mathrm C^3_c(\R)$ with the same convergence rates as in \Cref{thm:lowerbound_eigenvalue}, see \Cref{sec:perturbation_function}.
\end{remark}

\subsection{Functional estimation}

The problem of estimating $\mu_{\Upsilon, f}$ 
falls within the scope of  functional estimation, i.e., estimating a functional $\mathrm{T}:\Theta\to \R$ defined over a set $\Theta$ of densities at a fixed density $f\in \Theta$, while observing  $n$ i.i.d.~samples $X_1,\dots,X_n\sim f$. 
This constitutes a classical question in statistics, that has been studied since at least   the 60s with the estimation of the squared norm $\mathrm{T}_f = \int f^2$ of a density \cite{lehmann1963nonparametric, bhattacharyya1969estimation, antille1972linearite}. 
The estimation of other functionals such as the Fisher information \cite{bhattacharya1967estimation} or the entropy \cite{levit1978asymptotically} was then introduced in the 70s and the 80s, see \cite[Section 4.4]{rao1983nonparametric} for a state of the art at the beginning of the 80s. Significant progress was achieved by Bickel and Ritov  \cite{bickel1988estimating} who showed the following  elbow effect. Consider the quadratic functional ${\mathrm{T}}_f = \int f^2$. As long as $f\in \Cf^s$ for $s> \frac 14$, estimators attaining the parametric rate of convergence exist. However, for $s\leq \frac 14$, the rate is deteriorated to  $\cO(n^{-\frac{4s}{4s+1}})$. They further show that this rate is minimax optimal on appropriate H\"older balls. Besides,  efficient estimation is possible in the smooth regime: the estimators are asymptotically normal with minimal asymptotic variance. See also \cite{donoho1990minimax,laurent1996efficient, laurent1997estimation, laurent_adaptive, cailow, nickl} where other estimators of quadratic functionals are developed.

To move beyond the estimation of quadratic functionals, a central idea is to employ debiasing techniques. 
Various approaches have been developed in this direction; we refer the reader to the introduction of \cite{koltchinskii_functional_2023} for a comprehensive overview. Our method follows from the theory of higher order influence functions developed by Robins, Li, Tchetgen and van der Vaart, see \cite{robins2008higher}, that we now present. To fix ideas, suppose first that the functional $\mathrm{T}$ is Lipschitz-continuous with respect to the $\L^2$-norm. Let $f\in \Cf^s$ be a density on $\R^d$ with bounded support. It is well-known that one can construct estimators $\hat f$ satisfying the rate of convergence $\E[\|\hat f-f\|^2_{\L^2}]\lesssim n^{-\frac{2s}{2s+d}}$, see e.g., \cite{tsybakov_introduction_2009}. By Lipschitz continuity, this directly yields the same rate for the plug-in estimator $\hat {\mathrm{T}}_0 = {\mathrm{T}}_{\hat f}$, namely
\begin{equation}\label{eq:estimator_bias0}
    \E[|\hat {\mathrm{T}}_0-{\mathrm{T}}_f|]\lesssim n^{-\frac{s}{2s+d}}.
\end{equation}
The estimator $\hat {\mathrm{T}}_0$  serves as a preliminary, but potentially biased, estimator. Suppose now that the functional $\mathrm T$ is Fréchet twice-differentiable over $\L^2$. Then a first-order Taylor expansion around $\hat f$ yields
\begin{equation}\label{eq:order1_expansion}
     {\mathrm{T}}_f= {\mathrm{T}}_{\hat f} +  {\mathrm{T}}^{(1)}_{\hat f} [f-\hat f] + \cO(\|f-\hat f\|_{\L^2}^2),
\end{equation}
where ${\mathrm{T}}^{(1)}_{\hat f}$ is a linear form over $\L^2$. This expansion  suggests an improved estimator of the form
\begin{equation}
    \hat {\mathrm{T}}_1 = \hat {\mathrm{T}}_0 + \hat B_1,
\end{equation} 
 where $\hat B_1$ is an estimator of   the linear term $B_1={\mathrm{T}}^{(1)}_{\hat f} [f-\hat f]$. 
By the Riesz representation theorem, there exists a function $\psi_{\hat f}\in \L^2$ such that ${\mathrm{T}}^{(1)}_{\hat f}[u] = \int \psi_{\hat f}\,u$, allowing us to estimate $B_1$ by an empirical mean on a second independent sample. Such an estimator $\hat B_1$ satisfies $\E[|\hat B_1-B_1|]\lesssim n^{-1/2}$. Putting everything together, the risk of $\hat {\mathrm{T}}_1$ satisfies
\begin{equation}\label{eq:estimator_bias1}
    \E[|\hat {\mathrm{T}}_1- {\mathrm{T}}_f|] \lesssim n^{-\frac{1}2}+ \mathcal{O}(\E[\|f-\hat f\|^2_{\L^2}]) \lesssim n^{-\frac 12} + n^{-\frac{2s}{2s+d}}.
\end{equation}
Comparing \eqref{eq:estimator_bias1} with \eqref{eq:estimator_bias0}, we observe that debiasing can reduce the bias quadratically, without significantly increasing the variance. 
This debiasing method was first introduced by Ibragimov, Nemirovski, and Khas'minskii in the Gaussian white noise model \cite{ibragimov1986some}, and later extended to settings such as regression and density estimation; see \cite{donoho1990minimax, nemirovskii1991necessary, goldstein1992optimal, birge1995estimation, nemirovski2000topics} for further developments.

This approach extends beyond first-order corrections: higher-order Taylor expansions can be used to further reduce bias, with remainder term at order $J$ typically decaying as $\cO(n^{-\frac{Js}{2s+d}})$. In dimension $d=1$, Birgé and Massart \cite{birge1995estimation} applied this methodology with $J=3$ to estimate  integral functionals of the form
\begin{equation}\label{eq:integral_functional}
    {\mathrm{T}}_f = \int \psi(f(x))\dd x
\end{equation}
where $\psi$ is a smooth function and $f\in \Cf^s$ is a density on $\R$. This framework includes, as special cases, key quantities such as the entropy. Their results, together with following works conducted by Laurent \cite{laurent1997estimation}, and Kerkyacharian and Picard \cite{kerkyacharian1996estimating}, show that such integral functionals in dimension $1$  can be estimated at rate $\cO(n^{-\frac{4s}{4s+1}}+n^{-\frac 12})$. Birgé and Massart also established that this rate is the minimax rate, and that efficient estimation is possible in the regime $s>\frac 14$. 

This idea has since been applied to various problems in nonparametric statistics through the use of higher-order influence functions, see \cite{van2014higher, liu2017semiparametric,liu2021adaptive}. In these approaches, the key step is  to estimate the multilinear forms  $G$ appearing in the Taylor expansion of the functional $\mathrm{T}$. 
In this work, we identify sufficient conditions, formulated solely in terms of operator norms of $G$, independent of any particular representation, that allow for the estimation of $G[f,\dots,f]$ at the rate 
\begin{equation}\label{eq:the_rate}
   \mathcal O(n^{-\frac{4s}{4s+d}} + n^{-\frac 12})
\end{equation}for any density $f\in \Cf^s$ on a $d$-dimensional domain. As  is standard in this literature, our estimation method relies on the use of $U$-statistics and wavelet-based projection estimators. 
If the Taylor expansion of  a functional $\mathrm{T}$ consists of multilinear forms satisfying these conditions, then the debiasing procedure can be applied  to estimate ${\mathrm{T}}_f$
  at the rate \eqref{eq:the_rate}.
We refer to such functionals as  \textit{$s$-regular functionals}, see \Cref{def:regular_functional}.

While our primary motivation is the estimation of eigenvalues of weighted Laplace operators $\Delta_f$,  we believe that our set of  sufficient conditions provides a simple toolkit to analyze the performance of debiasing estimators for functional estimation problem, that could potentially be applicable to a broader range of functionals, such as:
\begin{enumerate}
    \item  eigenvalues of other differential operators depending on a density $f$;
    \item integral functionals $\int \psi(u_f)$ of the solution $u_f$ to a PDE depending on a density $f$;
    \item  Besov or Sobolev norms of a density $f$;
    \item optimal transport costs between $f$ and a fixed reference density $g$;
    \item solutions to optimal control problems parametrized by a density $f$.
\end{enumerate}

\subsection{Related work}\label{sec:related}

 The graph Laplacian matrix $L_n$ can naturally be extended to an operator acting on smooth functions on $M$. For a smooth function $u$ and a point $x \in M$, the rate of convergence of $L_n u(x)$ to $\Delta_{f} u(x)$, known as the \textit{pointwise convergence} rate of the graph Laplacian, has been studied by multiple authors for kernel-based graphs in \cite{hein2005graphs, belkin2006convergence,singer2006graph, gine2006empirical, berry2016variable} and for neighborhood graphs in \cite{ting2011analysis, guerin2022strong}. However, for most applications, the convergence of the spectrum of graph Laplacians, known as \textit{spectral convergence}, is more relevant. Yet, relating pointwise convergence results to spectral convergence is far from trivial. Although spectral consistency was established as early as 2006 by Belkin and Niyogi \cite{belkin2006convergence}—see also \cite{von2008consistency, trillos2018variational}—the first quantitative spectral rates appeared only much later. García Trillos, Gerlach, Hein and Slepčev \cite{garcia2020error} obtained the rate $\cO(n^{-\frac{1}{2d}})$ for both eigenvalues and eigenfunctions when the underlying density $f$ is Lipschitz continuous. 
Their technique relies on fine stability properties of the spectra of metric-measure Laplacians with respect to the underlying measure, see also \cite{burago2015graph}. This technique was further improved by Calder and García Trillos \cite{calder2022improved}, who obtained a convergence rate of order $\cO(n^{-\frac{1}{4+d}})$ when the density $f$ is of H\"older regularity $\Cf^s$ for some $s>2$. Since then, this rate has   been steadily improved, with a series of papers over the past few years obtaining different convergence rates under slightly different conditions. For instance, Wormell and Reich \cite{wormell2021spectral} obtained a spectral error of order $\cO(n^{-\frac{4}{12+d}})$ for $s>2$ and $\alpha=1$ on the flat torus, Cheng and Wu \cite{cheng2022convergence} obtained an eigenvalue error of order $\cO(n^{-\frac{2}{4+d}})$ and an eigenfunction error of order $\cO(n^{-\frac{2}{6+d}})$ for a smooth density $f$ and $\alpha=0$, and Wahl \cite{wahl2024kernel} obtained a spectral error of order $\cO(n^{-\frac{2}{8+d}})$ for a uniform density $f$. 

Note that the results of our paper indicate that all these rates are not minimax, at least if the manifold $M$ is known.

\subsection{Open questions} 
We conclude this introductory section with a list of open questions raised by our work.

\subsubsection*{The bounded domain case.} Instead of considering datasets  supported on compact manifolds, it seems very reasonable to consider the case where the dataset is supported on a bounded domain $\Omega\subset \R^d$. 
This raises natural questions: 
does the plug-in estimator still achieve minimax rates in this setting? Can the debiasing techniques we developed be extended to handle boundary effects?

\subsubsection*{The unknown manifold  case.} Suppose  that \( M \) is an \textit{unknown} \( d \)-dimensional submanifold of \( \mathbb{R}^D \). Can our estimators be adapted to this case, at least in principle? We conjecture that if \( M \) is sufficiently smooth, one could adapt the chart-based manifold estimators of Aamari and Levrard \cite{aamarilevrard} to construct an estimated manifold \( \hat M \) that is \( \Cf^k \)-close to \( M \). Our method could then be adapted by computing the spectrum of a modified weighted Laplace operator based on the metric \( \hat g \) of \( \hat M \) rather than the true metric \( g \) of \( M \). Aamari and Levrard show that their estimators achieve a rate of the form \( n^{-\frac{m-k}d} \) in \( \Cf^k \)-norm (up to logarithmic factors) if the manifold is of regularity $m$. This suggests that \( \hat g \) and \( g \) can be made arbitrarily close if we assume that the manifold is of class $\Cf^\infty$. We thus conjecture that crude stability results for the spectrum of weighted Laplacians under perturbations of the metric could be developed, allowing our estimators to be used on \( \hat M \) without loss in convergence rate. In summary, we conjecture that the absence of direct access to the true manifold \( M \) does not deteriorate the statistical rate of convergence for spectrum estimation should $M$ be smooth enough, although a precise analysis remains to be conducted.

\subsubsection*{Minimax optimality of graph Laplacians} Instead of relying on the intricate construction described in the previous paragraph, it would be  more satisfactory to prove that graph Laplacians achieve the minimax rate. This question has already been addressed for $s=2$ and the $\H^1$-loss in \cite{trillos2025minimax}, but remains open (i) for $s>2$, (ii) for eigenfunction estimation using the more common $\L^2$-loss, and (iii) for eigenvalue estimation.  We believe that bias-reduction techniques would likely be necessary to attain minimax rates for eigenvalue estimation.

\subsubsection*{Towards score-based spectral methods}  Score-based methods are powerful tools for estimating the score function $\nabla \log f$ of a density $f$, with neural networks achieving impressive empirical results in  high-dimensional settings. Our theoretical analysis shows that accurately estimating the score function is sufficient to obtain precise approximations of the spectrum of the weighted Laplace operator $\Delta_f$. 
This insight motivates the following two-step procedure:
\begin{enumerate}
    \item Estimate the score function $\nabla \log f$ via score matching \cite{hyvarinen2005estimation} or denoising score matching \cite{vincent2011connection}, yielding an approximation $\hat s$;
    \item Use this estimate to define the operator $\hat{\Delta} = \Delta + \alpha \hat s\cdot \nabla$, and approximate its spectrum using, e.g., modern neural PDE solvers, see  \cite{huang2025partial} for a review.
\end{enumerate}

Comparing the spectral properties and embedding quality of this  approach with standard techniques like spectral clustering or classical dimension reduction methods would be an interesting direction for further investigation.

\subsection*{Organization}
\Cref{sec:functional} deals with the general problem of  functional estimation over H\"older-Zygmund spaces, and can be read independently from the rest of the paper. Spectral estimation of weighted Laplace operators is discussed in \Cref{sec:spectral_estimation}. Minimax lower bounds are proven in \Cref{sec:lower_bounds}. Most technical results are collected in the Appendix.

\section{Regular functional estimation}\label{sec:functional}

We focus on the problem of estimating the value of a real-valued functional ${\mathrm{T}}_f$ depending on a density $f$, under the assumption that $f$ belongs to a ball in the H\"older-Zygmund space $\cC^s$ for some $s>0$.

Let $T : \Theta \to \R$ be a map over an open subset $\Theta \subset \mathscr C^s$. We say that \textit{$\mathrm{T}$ can be estimated at a rate $(r_n)_{n\geq 1}$ over $\mathscr C^s$} if for all $R>0$ there is $C>0$ such that the following holds. For any $n\geq 1$ there exists an estimator $\widehat {\mathrm{T}}_n : M^n \to \R$,  such that for each density $f\in \Theta$ with $\|f\|_{\cC^s}\leq R$, we have
\[
\mathbf E_{f^{\otimes n}}[({\mathrm{T}}_f - \hat {\mathrm{T}}_n)^2]\leq Cr_n^2.
\]
In what follows for simplicity we will write ${\hat{\mathrm T}} = \hat {\mathrm{T}}_n$ when $n$ is implicit.
When  a parametric rate of convergence holds, that is when $r_n=n^{-1/2}$, we will additionally show that the estimators we propose are \emph{asymptotically efficient}, see \cite[Chapter 25]{van2000asymptotic}. The  efficient influence function $ \psi_f$ is defined as a function $\psi_f\in \L^2$  such that for all $g\in \cC^s$ with $\int gf=0$,
\begin{equation}\label{def:efficient}
    \frac{{\mathrm{T}}_{f(1+tg)}-{\mathrm{T}}_f}{t} \xrightarrow[t\to0]{} \int \psi_f g f 
\end{equation}
(should such a function $\psi_f$ exist). 
Remark that ${\mathrm{T}}_{f(1+tg)}$ is  well-defined for $t$ small enough as $\Theta$ is open. 
The function  $\psi_f$ is not uniquely defined, but only up to an additive constant. With a slight abuse, we will still refer to $\psi_f$ as \emph{the}  efficient influence function any such function. We say that an estimator ${\hat{\mathrm T}}$ of ${\mathrm{T}}_f$ is \emph{asymptotically efficient} at $f$ if it holds that
\begin{equation}
    {\hat{\mathrm T}}-{\mathrm{T}}_f = \frac{1}{n} \sum_{i=1}^n (\psi_f(X_i)-\E[\psi_f(X_i)]) + o_{\P}(n^{-1/2}).
\end{equation}
In particular, by the central limit theorem, any asymptotically efficient estimator is asymptotically normal with asymptotic variance $\Var_{X\sim f}(\psi_f(X))$. This variance should be thought of as the minimal asymptotic variance an estimator of ${\mathrm{T}}_f$ can attain up to a factor $1/n$, an idea which can be made precise, see \cite[Chapter 8]{van2000asymptotic}.
\medskip

\begin{remark}
A bounded domain $\Omega\subset \R^d$ can be isometrically embedded in the  flat torus $\T^d$. This remark implies that all the results in this section also applies to densities $f$ with H\"older-Zygmund regularity on $\R^d$, as long as they are supported on a fixed, known set $\Omega$. Our results however do not cover densities with unbounded supports or densities that belong to H\"older-Zygmund spaces on domains (e.g., $f$ is the uniform density on $\Omega$).
\end{remark}

The debiasing method we propose requires estimating multilinear forms appearing in the Taylor expansion of the functional $\mathrm{T}$. We  give a set of conditions that allows for the estimation of such multilinear forms. Unlike previous results appearing in the literature that rely on  the multilinear forms having specific expression (such as the one appearing in \Cref{lem:example_form}), our conditions are only phrased in terms of certain mapping properties of the multilinear form, and are blind to its exact expression. 

Recall that the Hilbert-Schmidt norm of an operator $G:\cH_1\to \cH_2$ between two separable Hilbert spaces is defined as 
$$
\|G\|_{\mathrm{HS}} = \p{\sum_{i_1,i_2\geq 1} |\dotp{e^2_{i_2},G[e^1_{i_1}]}_{\cH_2}|^2}^{1/2},
$$
where $(e^j_i)_{i\geq 1}$ is an orthonormal basis of $\cH_j$ for $j=1,2$.  We also let $\cL(E)$ be the set of bounded operators $E\to E$.

\begin{definition}[Regular multilinear form]\label{def:regular}
Let $s, L>0$.
    \begin{enumerate}[start=1, label={\emph{(G\arabic*)}}]
        \item \label{G1} A linear form $G$ is said to be an $s$-regular form with norm $L$ if there exists $\psi\in \cC^s$ with $\|\psi\|_{\cC^s}\leq L$ such that $G:v\mapsto \int v\psi$. 
        \item \label{G2} A symmetric bilinear form $G$ is said to be an $s$-regular form with norm $L$ if for all $v\in \cC^s$, $G[\cdot,v]$ is an $s$-regular form with norm $L\|v\|_{\cC^s}$  and if $G:\L^2\times \L^2 \to \R$  has norm smaller than $L$.
        \item \label{G3} A symmetric trilinear form $G$ is said to be an $s$-regular form with norm $L$ if for all $v\in \cC^s$,  $G[\cdot,\cdot,v]$ is an $s$-regular bilinear form  with norm $L\|v\|_{\cC^s}$, if $G: \L^\infty \times \L^\infty \times \L^\infty\to \R$ has norm smaller than $L$ and if for $t=(s+d)/2$,  
        \begin{equation}\label{eq:HS_condition}
            G: \H^{t}\times \H^{t}\times \L^{2}\to \R
        \end{equation}
         has Hilbert-Schmidt norm smaller than $L$.
    \end{enumerate}
\end{definition}

\begin{remark}\label{remark:sufficient_G3}
 We show in \Cref{app:hilbert-schmidt} that the Hilbert-Schmidt condition \eqref{eq:HS_condition}  is implied by a control of the operator norm of  $G:  \L^2\times \L^2 \times \L^2\to \R$, which is typically easier to verify in practice.
\end{remark}

\begin{lemma}\label{lem:example_form}
Let $s> 0$.  Let $J\in \{1,2,3\}$ and define a $J$-multilinear form $G$ by
\begin{equation}\label{eq:example_form}
            G[v_1,\dots,v_J]= \int  v_1\cdots v_J\cdot \psi
        \end{equation}  for some function $\psi\in \cC^s$.
        Then, $G$ is an $s$-regular form.
\end{lemma}
\begin{proof}
    The result follows from the definition for $J=1$. For $J=2$, it follows from Cauchy-Schwarz inequality and   the fact that $\cC^s$ is a multiplication algebra, see \Cref{app:besov}. The only nontrivial condition to check for $J=3$ is the Hilbert-Schmidt condition in \emph{\ref{G3}}. We show in \Cref{app:hilbert-schmidt} that the Hilbert-Schmidt norm of $G:\H^{t}\times \H^{t}\times \L^2$ is finite for all $t>d/2$, implying in particular the result for $t=(s+d)/2$.
\end{proof}
As already mentioned in the introduction, the estimation of $G[f,\dots,f]$ for $G$ the functional in  \eqref{eq:example_form}  constitutes a classical problem in statistics. The previous lemma states that they fall within the scope of $s$-regular functionals for which we will be able to give rates of estimation.

\begin{definition}[Regular functional]\label{def:regular_functional}
Let $s> 0$  and $L>0$. Let $J_{\max}=2$ if $s> d/4$ and $J_{\max}=3$ otherwise.  
We say  that $\mathrm{T}:\Theta \to \R$ is an $s$-regular functional over $ \Theta$  with norm smaller than $L$ if  $\mathrm{T}$ is bounded by $L$  and if there exist $s'<s$ and $\eta>0$ such that for every $f,f'\in\Theta$, with $h=f'-f$ satisfying $\|h\|_{\cC^{s'}}\leq \eta$, it holds  
\begin{equation}\label{eq:taylor_expansion_T}
\begin{split}
{\mathrm{T}}_{f'}&= {\mathrm{T}}_f + \sum_{J=1}^{J_{\max}} {\mathrm{T}}^{(J)}_f[h,\cdots,h]+ R_f(h),
\end{split}
\end{equation}
where, for $J\in [J_{\max}]$, ${\mathrm{T}}^{(J)}_f$ is an $s$-regular form with norm smaller than $L$ and $|R_f(h)|\leq L \|h\|_{\L^\infty}^{J_{\max}+1}$.
\end{definition}
In the previous definition, the choice of $J_{\max}=2$ or $J_{\max}=3$ indicates where the Taylor expansion should be stopped to estimate the functional. If $s>d/4$, a cubic remainder of order $\cO(n^{-\frac{3s}{2s+d}})$ is negligible in front of $\cO(n^{-\frac 12})$, so that stopping the Taylor expansion at $J_{\max}=2$ is enough. If $s\leq d/4$, the cubic remainder is not negligible, and we must stop the Taylor expansion at $J_{\max}=3$ to obtain a negligible remainder of order four.

We now describe a simple family of regular functionals.
\begin{lemma}\label{lem:example_functional}
Let $s>0$ and let $\Theta\subset \cC^s$ be a set of functions taking values in an open interval $I$. Let $t=\max(s+3,4)$ and let $\phi: I\to \R$ be a function in $\cC^{t}$. Then, $\mathrm{T}:f\in \Theta\mapsto  \int  \phi(f(x))\dd x$ is an $s$-regular functional  over $\Theta$.
\end{lemma}
\begin{proof}
    This is a direct consequence of  \Cref{lem:example_form}.
\end{proof}
The minimax estimation of integral functionals described in the previous lemma was analyzed by Laurent \cite{laurent1996efficient, laurent1997estimation} ($s\geq 1/4$) and by Kerkyacharian and Picard \cite{kerkyacharian1996estimating} ($s\leq 1/4$).

We are now ready to state our main theorem: $s$-regular functionals  can be estimated at rate $\cO(n^{-\frac{4s}{4s+d}}+n^{-\frac 12})$, with efficient estimation when $s>d/4$.

\begin{theorem}\label{thm:functional_estimation}
 Let $s> 0$  and $L>0$.  Let $\mathrm{T}$ be an $s$-regular functional with norm smaller than $L$, defined over a set $ \Theta\subset \cC^s$. Let $\Theta_0\subset \Theta$ be a set such that $\Theta$ is included in a $\eta$-neighborhood of $\Theta_0$ for the $\cC^{s'}$-norm, where $\eta$ and $s'$ are the constants appearing in \Cref{def:regular_functional}. 
  \begin{enumerate}[label=\emph{(\roman*)}]
      \item  The functional $\mathrm{T}:\Theta_0\to \R$ can be estimated at rate $n^{-\frac{4s}{4s+d}}+n^{-\frac 12}$ over $\cC^s$. 
      \item Moreover, assume that $f\in \Theta\mapsto {\mathrm{T}}_f^{(1)}$ is continuous from $\Theta$ endowed with the $\cC^{s'}$-norm to the set of linear forms on $\L^2$ endowed with the operator norm and that $s>\frac d4$. Then, there exists an asymptotically efficient estimator of ${\mathrm{T}}_f$ at all densities $f \in \Theta_0$, with efficient influence function $\psi_f$ defined by ${\mathrm{T}}^{(1)}_f[u]= \int u\psi_f$ for all $u \in \cC^s$.
  \end{enumerate}
\end{theorem}
As already mentioned, \Cref{thm:functional_estimation} covers the large class of integral functionals  defined in \Cref{lem:example_functional}. More importantly to us, it also covers the estimation of other types of functionals, such as  eigenvalues of  families of operators depending on a density $f$. An example of application to the estimation of eigenvalues of weighted Laplace operators is given in \Cref{sec:asymptotic_expansion}.

The rest of this section is dedicated to proving \Cref{thm:functional_estimation}. The proof is decomposed into four steps: first, we estimate regular $J$-multilinear forms for $J=1,2,3$, and then  regular functionals. We now fix a density $f\in \Theta_0$ with $\|f\|_{\cC^s}\leq R$ for some $R>0$ (in particular, $\sup f\leq R$). We let $X_1,\dots,X_n$ be an i.i.d.~sample with density $f$, with $\mu_n$  the associated empirical measure. 

The estimators  use as building blocks a family of  symmetric kernel functions $K_D:M\times M \to \R$ indexed by an integer $D$. These kernel functions give rise to linear operators $\pi_D:u\mapsto \int K_D(x,y)u(y) \dd y$ defined for any distribution of sufficiently small order. The exact way these kernels are constructed do not matter; instead, we list properties that they should satisfy. 

\begin{definition}\label{def:admissible}
    Let $s_\star>0$ be an integer. We say that a family of  symmetric kernels $(K_D)_{D\geq 1}$ of class  $\Cf^{s_\star}$ over $M\times M$ (with associated operators $(\pi_D)_{D\geq 1}$) is admissible if for all $\ell,s\in \left[-s_\star,s_\star\right]$ with $\ell< s$ and $1\leq p'\leq p \leq \infty$ and $1\leq q,q' \leq \infty$, there exists $C_{\star}$ such that for all $D\geq 1$:
    \begin{enumerate}[start=1, label=\emph{(K\arabic*)}]
    \item \label{K1} $\|1-\pi_D\|_{\B^{s}_{p,q}\to \B^{\ell}_{p',q'}}\leq C_{\star} D^{-(s-\ell)}$;
    \item \label{K2}  $\|\pi_D\|_{\B^s_{p,q}\to \B^{s}_{p,q}}\leq C_{\star}$;
    \item \label{K3} for any integer $J\geq 1$,  $\ell_1,\dots,\ell_J\in [0,s_\star]$, integers $D_1,\dots,D_J\geq 1$, and any multilinear form $G:\H^{\ell_1}\times \cdots\times \H^{\ell_J}\to \R$ with Hilbert-Schmidt norm smaller than $L$, $$\int |G[K_{D_1}(x_1,\cdot), \dots,K_{D_J}(x_J,\cdot)]|^2 \dd x_1\dots \dd x_J \leq C_{\star} L^2 \prod_{j=1}^J D_j^{2\ell_j}. $$
    \item \label{K4} $\|y\mapsto \|K_D(\cdot,y)\|_{\B^{\ell}_{p,p}}\|_{\L^p} \leq C_{\star} \kappa_{p,d,\ell}(D)$, where $\kappa_{p,d,\ell}(D)$ is equal to $D^{\ell+d(1-1/p)}$ if the exponent $\ell+d(1-1/p)$ is positive, $\log(D)$ if it is equal to zero, and $1$ if it is negative. 
\end{enumerate}
\end{definition}
We show in \Cref{app:besov} that admissible kernels, constructed using wavelets, exist. In the remainder of the proof, we fix such a family of admissible kernels with $s_\star\geq (s+d)/2$. The next lemma gives some properties of $s$-regular multilinear forms.

\begin{lemma}\label{lem:auxiliary_regular}
Let $s\in \left]0,s_\star\right]$ 
  and $L>0$.  Let $J\in \{1,2,3\}$. There exist $C_0>0$ such that for all symmetric $s$-regular $J$-multilinear forms $G$ with norm smaller than $L$, the following properties hold.
    \begin{enumerate}[label=\emph{(\arabic*)}]
         \item \emph{\textbf{Bias control.}} For all $v_1,\dots,v_J\in \cC^s$ with $\cC^s$-norm smaller than $1$ and integers $D\geq 1$, it holds $|G[(1-\pi_{D})v_1,v_2,\dots,v_J]|\leq C_0 L D^{-2s}.$ \label{item:bias}
    \item  \label{item:variance}  \emph{\textbf{Variance control.}} Let $1\leq j\leq J$ and let $v_{j+1},\dots,v_J\in \cC^s$ with $\cC^s$-norm smaller than $1$.  Let $\widetilde s=s$ if $j=3$ and $\widetilde s=0$ otherwise. Then, for all integers $1\leq D_1\leq \cdots \leq D_j$, 
   \begin{equation*} \int G[\pi_{D_1}\delta_{x_1},\dots,\pi_{D_j}\delta_{x_j},v_{j+1},\dots,v_J]^2 \dd x_1\cdots \dd x_j\leq C_0 L^2 D_j^{2\widetilde s} \prod_{k=1}^{j-1}D_k^{d}.
   \end{equation*}
    \end{enumerate}
\end{lemma}
Before proving the lemma, let us state the following useful formula: for any $x\in M$ and $u\in \cD'_{s^*}$, by definition of the operator $\pi_D$,
\begin{equation}\label{eq:projection_dirac}
    \dotp{\pi_D\delta_x,u}= \dotp{K_D(x,\cdot),u}=\pi_Du(x).
\end{equation}
\begin{proof}
    By definition of $s$-regular forms, there is a $(J-1)$-multilinear form $\widetilde G:(\cC^s)^{J-1}\to \cC^s$ of norm smaller than $L$ such that for all $v_1\in (\cC^s)^*$, $v_2,\dots,v_J\in \cC^s$, it holds that $G[v_1,\dots,v_J]= \dotp{v_1, \widetilde G[v_2,\dots,v_J]}$.
    \begin{enumerate}[leftmargin=*, labelindent=0pt,  align=left]
        \item Let  $v_1,\dots,v_J\in \cC^s$ with $\cC^s$-norm smaller than $1$ and let $D\geq 1$ be an integer. By using \emph{\ref{K1}} and the relation $B^{-s}_{1 1}\subset (\cC^{s})^*$ (see \Cref{app:besov}), we obtain that
        \begin{align*}
              |G[(1-\pi_D)v_1,v_2,\dots,v_J]|&= |\dotp{(1-\pi_D)v_1, \widetilde G[v_2,\dots,v_J]}|\\
              &\leq  \|(1-\pi_D)v_1\|_{\B^{-s}_{1 1}} \|\widetilde G[v_2,\dots,v_J]\|_{\cC^{s}} \\
              &\leq   C_{\star}L D^{-2s}.
        \end{align*}
      \item Let $j=1$ and let $v_2,\dots,v_J\in \cC^s$.  First, remark that $G[\pi_D\delta_x,v_2,\cdots,v_J]= \pi_D\widetilde G[v_2,\cdots,v_J](x)$. Thus, by \emph{\ref{K2}} and the embedding $\cC^{s}\hookrightarrow\L^2$, we have
      \begin{align*}
        \int G[\pi_{D_1}\delta_{x},v_2,\dots,v_J]^2\dd x &= \|\pi_{D_1}\widetilde G[v_2,\dots,v_J]\|^2_{\L^2} \\
        &\leq C_{\star} \|\widetilde G[v_2,\dots,v_J]\|^2_{\L^2} \\
        &\leq  C_2C_{\star} \|\widetilde G[v_2,\dots,v_J]\|^2_{\cC^{s}}\\
        &\leq C_2C_{\star}L^2.
      \end{align*}
      for some constant $C_2$ depending on $s$. 
      
      Let $J\geq j=2$. Condition \emph{\ref{G2}} implies that $\widetilde G:\L^2\times (\cC^s)^{J-2}\to \L^2$ has operator norm smaller than $L$. Thus, by \emph{\ref{K2}} and \emph{\ref{K4}}, using the symmetry of $G$,
       \begin{align*}
        \int G[\pi_{D_1}\delta_{x_1},\pi_{D_2}\delta_{x_2},v_3]^2\dd x_1\dd x_2 &=\int \|\pi_{D_2}\widetilde G[\pi_{D_1}\delta_{x_2},v_3]\|^2_{\L^2} \dd x_2\\
        &\leq C_{\star} \int \|\widetilde G[\pi_{D_1}\delta_{x_2},v_3]\|^2_{\L^2} \dd x_2 \\
        &\leq  C_{\star}L^2  \int \|\pi_{D_1}\delta_{x_2}\|^2_{\L^2} \dd x_2 \\
        &\leq C_3C_{\star}L^2   D_1^{d}
      \end{align*}
      for some constant $C_3$. 
      The statement for $J=j=3$ is a direct consequence of \emph{\ref{K3}} and the Hilbert-Schmidt condition in \emph{\ref{G3}}.
      \qedhere
    \end{enumerate}
\end{proof}

We  rely on $U$-statistics of order $J$ to estimate $J$-multilinear forms. Let $J\geq 1$ be an integer and let $b:M^J\to \R$ be a symmetric function. The associated $U$-statistic is defined as
\begin{equation}
    \U_n^{(J)} b= \binom{n}{J}^{-1} \sum_{1\leq i_1<\cdots <i_J\leq n} b(X_{i_1},\dots,X_{i_J}).
\end{equation}
We will focus on situations where $b(x_1,\dots,x_J)= B[\delta_{x_1},\dots,\delta_{x_J}]$ for some symmetric $J$-multilinear form  $B$ defined over signed measures. In that case, the Hoeffding decomposition of the $U$-statistic takes the particularly convenient form
\begin{equation}\label{eq:hoeffding}
    \U_n^{(J)} b = B[f,\dots,f] + \sum_{j=1}^J \binom{J}{j} \U_n^{(j)}b_j,
\end{equation}
where $b_j(x_1,\dots,x_j)= B[\delta_{x_1}-f,\dots,\delta_{x_j}-f,f,\dots,f]$. The different $U$-statistics in the above decomposition are centered and orthogonal, with
\begin{equation}\label{eq:var_u_stat}
\begin{split}
    \Var(\U_n^{(j)}b_j)&=\binom{n}{j}^{-1}\E[b_j(X_1,\dots,X_j)^2]\\
    &\leq R^j \binom{n}{j}^{-1}  \int B[\delta_{x_1},\dots,\delta_{x_j},f,\dots,f]^2 \dd x_1\cdots \dd x_j.
    \end{split}
\end{equation}
We will rely on the variance control in \Cref{lem:auxiliary_regular} to control variance terms of this form.

\subsection{Estimation of linear functionals}
We first prove \Cref{thm:functional_estimation} in the particular case where $\mathrm{T}$ is a linear functional. In that case $\mathrm{T}$ is an $s$-regular linear form which takes the form $\mathrm{T}:v\mapsto \int v\psi$ for some $\psi\in \cC^s$. 
Define the estimator ${\hat{\mathrm T}} = \mu_n(\psi)$. The estimator has no bias and has  variance bounded by   
\begin{align*}
  \frac{1}n \int (\psi(x)-\E[\psi(X)])^2 f(x)\dd x&\leq  \frac{1}{n} \int \psi(x)^2 f(x)\dd x \leq \frac{RL^2}{n}.
\end{align*}
 The estimator ${\hat{\mathrm T}}$ is also asymptotically efficient by definition.

\subsection{Estimation of quadratic functionals}
Assume now that $\mathrm{T}$ is a quadratic functional. The functional $\mathrm{T}$ is then equal to the second order term in its Taylor expansion. In particular, $\mathrm{T}$ induces an $s$-regular bilinear form $G$, with ${\mathrm{T}}_h= G[h,h]$. Definition \emph{\ref{G2}} implies  that $G[u,v]= \dotp{u,\widetilde G[v]}$ for an operator $\widetilde G$ that sends both $\cC^s$ to $\cC^{s}$ and $\L^2$ to $\L^2$ with operator norm smaller than $L$. We let $D\geq 1$ be an integer of order $n^{\frac{2}{4s+d}}$ and define $b(x_1,x_2)= G[\pi_D\delta_{x_1},\pi_D\delta_{x_2}]$. Let ${\hat{\mathrm T}}$ be the associated $U$-statistic. Recall the Hoeffding decomposition \eqref{eq:hoeffding}. 

First note that the bias of the estimator is bounded by 
$$
|G[f,f]-G[\pi_Df,\pi_Df]| \leq |G[(1-\pi_D)f,f]|+|G[(1-\pi_D)f,\pi_Df]|.
$$
According to \Cref{lem:auxiliary_regular} and \emph{\ref{K2}}, it is therefore bounded by 
    \begin{equation}\label{eq:bias2}
    \begin{split}
        C_0 LD^{-2s}(\|f\|^2_{\cC^s} + \|f\|_{\cC^s}\|\pi_Df\|_{\cC^s})&\leq C_0 (1+C_{\star})LD^{-2s}\|f\|^2_{\cC^s}\\
        &\leq C_1 LR^2n^{-\frac{4s}{4s+d}}
        \end{split}
    \end{equation}
    for some constant $C_1$.

By \eqref{eq:var_u_stat} and \Cref{lem:auxiliary_regular}, the variance of $\U_n^{(2)} b_2$ is bounded by 
    \begin{equation}\label{eq:varU2}
         \binom{n}{2}^{-1} L^2R^2D^{d}\leq C_2L^2R^2n^{-\frac{8s}{4s+d}}
    \end{equation}
    for some constant $C_2$.

Likewise, the variance of $\U_n^{(1)}b_1$ is  bounded by
\begin{align*}
    \E[\U_n^{(1)}b_1^2]\leq \frac{R}n\int G[\pi_D\delta_x,\pi_Df]^2 \dd x \leq \frac{C_3R}n \|\pi_Df\|_{\cC^s}^2.
\end{align*}
for some constant $C_3$. 
By \emph{\ref{K2}}, $ \|\pi_Df\|_{\cC^s}\leq C_{\star} \|f\|_{\cC^s} \leq C_{\star}L$, 
so that $  \E[\U_n^{(1)}b_1^2]\leq C_4R^3L^2 n^{-1}$ for some constant $C_4$. 

In total, the risk of ${\hat{\mathrm T}}$ is of order $\cO(n^{-\frac{4s}{4s+d}}+n^{-\frac 12})$, up to a multiplicative constant depending on $L$ and $R$.

Finally, assume that $s> d/4$. Let us show that ${\hat{\mathrm T}}$ is an asymptotically efficient estimator. First, remark from \eqref{def:efficient} that the efficient influence function is the map $2\widetilde G[f]$. As $G[f,f] = \E[\widetilde G[f](X_i)]$, we can decompose ${\hat{\mathrm T}}$ into
\begin{equation}\label{eq:efficient_bi}
    {\hat{\mathrm T}} -{\mathrm{T}}_f= \frac{2}{n}\sum_{i=1}^n (\widetilde G[f](X_i)-\E[\widetilde G[f](X_i)]) + Z_1,
\end{equation}
where $Z_1$ is given by
\begin{align*}
     G[\pi_Df,\pi_Df] -G[f,f]+2 \U_n^{(1)} b_1+\U_n^{(2)}b_2 - \frac{2}{n}\sum_{i=1}^n (\widetilde G[f](X_i)-G[f,f]).
\end{align*}
We write a bias-variance decomposition and use \eqref{eq:bias2}, \eqref{eq:varU2} to obtain
\begin{align*}
    \E[Z_1^2] &\leq |G[\pi_Df,\pi_Df]-G[f,f]|^2 + 2\E[\U_n^{(2)} b_2^2] \\
    &\qquad + 2\E\left[ \p{\frac{2}{n}\sum_{i=1}^n (G[\pi_D(\delta_{X_i}-f),\pi_Df]-\widetilde G[f](X_i)+G[f,f]) }^2\right] \\
    &\leq C_5L^2(R^2+R^4)n^{-\frac{8s}{4s+d}} + \frac{8}{n} R \left\|\pi_D \widetilde G[\pi_D f] - \widetilde G[f]\right\|^2_{\L^2}
\end{align*}
for some constant $C_5$. We use \emph{\ref{K1}}, \emph{\ref{K2}}, and the mapping properties of $\widetilde G$ stated at the beginning of the section to obtain that there are constants $C_6,C_7$ such that
\begin{align*}
    \|\pi_D \widetilde G[\pi_D f] - \widetilde G[f]\|_{\L^2}&\leq C_6\|(1-\pi_D) \widetilde G[\pi_D f]\|_{\L^\infty} +\|\widetilde G[(1-\pi_D)f]\|_{\L^2}\\
    &\leq C_6C_{\star} D^{-s}\|\widetilde G[\pi_D f]\|_{\cC^{s}} + L\|(1-\pi_D)f]\|_{\L^2}\\
    &\leq C_6C_{\star} LD^{-s}\|\pi_D f\|_{\cC^s} + C_{\star}L\|f\|_{\L^2}D^{-s} \\
    &\leq C_7L RD^{-s} .
\end{align*}
In total, one sees that for some $C_8 > 0$ we have
\begin{equation}
    \E[Z_1^2] \leq C_8 L^2 (R^2+R^4)\p{n^{-\frac{8s}{4s+d}} + n^{-1}n^{-\frac{4s}{4s+d}} }.
\end{equation}
As $s>d/4$, this term is negligible in front of $n^{-1}$. This shows that ${\hat{\mathrm T}}$ is asymptotically efficient.

\subsection{Estimation of cubic functionals}
In this section, we assume that $\mathrm{T}$ is a cubic functional and that $0<s\leq d/4$.  
 By assumption, $T=G$ is an $s$-regular $3$-multilinear form. Let $D_1\leq D_2\leq D_3$ be three integers of orders respectively $n^{\frac{1}{4s+d}}$, $n^{\frac{3/2}{4s+d}}$ and $n^{\frac{2}{4s+d}}$.  We define the function $b : M^3 \to \R$ by
\begin{equation}\label{eq:U3}
\begin{split}
   b(x_1,x_2,x_3)& =G[\pi_{D_1}\delta_{x_1},\pi_{D_1}\delta_{x_2},\pi_{D_1}\delta_{x_3}] \\  
   &\quad +G[  (\pi_{D_2}-\pi_{D_1})\delta_{x_1}, (\pi_{D_2}-\pi_{D_1})\delta_{x_2}, (\pi_{D_2}-\pi_{D_1})\delta_{x_3}]\\
   &\qquad+3G[\pi_{D_1}\delta_{x_1},\pi_{D_1}\delta_{x_2},(\pi_{D_3}-\pi_{D_1})\delta_{x_3}] \\
    &\qquad\quad + 3G[\pi_{D_1}\delta_{x_1}, (\pi_{D_3}-\pi_{D_1})\delta_{x_2},  (\pi_{D_3}-\pi_{D_1})\delta_{x_3}]\\
    &\qquad\qquad+ 3G[(\pi_{D_2}-\pi_{D_1})\delta_{x_1},(\pi_{D_2}-\pi_{D_1})\delta_{x_2},(\pi_{D_3}-\pi_{D_2})\delta_{x_3}].
\end{split}
\end{equation}
and let ${\hat{\mathrm T}} = \U_n b$. 
This definition is motivated by the fact that we can decompose $G[f,f,f]$ into
\begin{align*}
    &G[\pi_{D_1}f,\pi_{D_1}f,\pi_{D_1}f] +G[  (\pi_{D_2}-\pi_{D_1})f, (\pi_{D_2}-\pi_{D_1})f, (\pi_{D_2}-\pi_{D_1})f]\\
    &\quad +3 G[\pi_{D_1}f,\pi_{D_1}f,(\pi_{D_3}-\pi_{D_1})f] \\
    &\qquad + 3G[\pi_{D_1}f, (\pi_{D_3}-\pi_{D_1})f,  (\pi_{D_3}-\pi_{D_1})f]\\
    &\qquad\quad+ 3G[(\pi_{D_2}-\pi_{D_1})f,(\pi_{D_2}-\pi_{D_1})f,(\pi_{D_3}-\pi_{D_2})f] \\
    &\qquad\qquad + G[(\pi_{D_3}-\pi_{D_2})f,(\pi_{D_3}-\pi_{D_2})f,(\pi_{D_3}-\pi_{D_2})f] \\
    &\qquad\qquad\quad + 3G[(\pi_{D_3}-\pi_{D_2})f,(\pi_{D_3}-\pi_{D_2})f,(\pi_{D_2}-\pi_{D_1})f]\\
    &\qquad\qquad\qquad + G[f,f,f]-G[\pi_{D_3}f,\pi_{D_3}f,\pi_{D_3}f].
\end{align*}
By construction, the expectation of ${\hat{\mathrm T}}$ is equal to the sum of the  first four lines  in the above expression, so that the bias is equal to the sum of the last three lines. 

\begin{remark}
    This decomposition is best understood if $G[u,v,w]=\int uvw$. In that case, the decomposition is obtained by writing $(\pi_{D_3}f)^3 = ((\pi_{D_3}f-\pi_{D_1}f)+\pi_{D_1}f)^3$, expanding the product, and then further expanding the term $(\pi_{D_3}f-\pi_{D_1}f)^3$ by writing $(\pi_{D_3}f-\pi_{D_1}f)^3=((\pi_{D_3}f-\pi_{D_2}f)+(\pi_{D_2}f-\pi_{D_1}f))^3$. This multiscale decomposition was already used in \cite{tchetgen2008minimax} to estimate the functional ${\mathrm{T}}_f=\int f^3$. 
\end{remark}

According to \emph{\ref{G3}}, for $v_1,v_2,v_3\in \cC^s$, 
$$ |G[v_1,v_2,v_3]|\leq L \|v_1\|_{\L^\infty}\|v_2\|_{\L^\infty}\|v_3\|_{\L^\infty}.$$
By \emph{\ref{K1}}, and as $D_1\leq D_2\leq D_3$, this implies that the bias is bounded by
    \[
         C_1 LR^3(D_2^{-3s}+ D_2^{-2s}D_1^{-s})+ |G[f,f,f]-G[\pi_{D_3}f,\pi_{D_3}f,\pi_{D_3}f]|
    \]
    for some absolute constant $C_1$. 
The two first  terms are of order at most $n^{-4s/(4s+d)}$. We bound the last term by
    \begin{equation*}
       |G[(1-\pi_{D_3})f,f,f]|+ |G[(1-\pi_{D_3})f,\pi_{D_3}f,f]|
      +  |G[(1-\pi_{D_3})f,\pi_{D_3}f,\pi_{D_3}f]|.
    \end{equation*}
    According to \Cref{lem:auxiliary_regular} and \emph{\ref{K2}}, each of these terms is bounded by $$C_2L R^3D_3^{-2s}\leq C_2LR^3 n^{-4s/(4s+d)}$$
    for some constant $C_2$. We rely on the Hoeffding decomposition \eqref{eq:hoeffding} to control the variance of ${\hat{\mathrm T}}$. By inspecting the expression of $b$, one can see that this variance is bounded by a sum of terms of the form
    \begin{align}
        &\frac Rn \int G[\pi_{D_{p}}f,\pi_{D_{q}}f, \pi_{D_{r}}\delta_{x_1}]^2 \dd x \label{eq:1}\\
        &\frac {R^2}{n^2}\int G[\pi_{D_{p}}f,\pi_{D_{q}}\delta_{x_1}, \pi_{D_{r}}\delta_{x_2}]^2 \dd x_1\dd x_2 \label{eq:2}\\
        &\frac {R^3}{n^3}\int G[\pi_{D_{p}}\delta_{x_1},\pi_{D_{q}}\delta_{x_2}, \pi_{D_{r}}\delta_{x_3}]^2\dd x_1\dd x_2\dd x_3\label{eq:3}
    \end{align}   
 where $(p,q,r)$ lies in the set 
 $$
P = \{(p,q,r)~:~1 \leqslant p \leqslant q \leqslant r \leqslant 3\ \ \text{and} \ \ p + q \leqslant 4\}
 $$
     The set $P$ is of cardinal $8$. For $(p,q,r) \in P$, one can bound the terms \eqref{eq:1}, \eqref{eq:2} and \eqref{eq:3} as follows. Recall that $D_1\leq D_2\leq D_3$.
    
 First, we apply \Cref{lem:auxiliary_regular} and \emph{\ref{K2}} to bound  \eqref{eq:1} by $C_1L^2R^5n^{-1}$ for some constant $C_1$. This is negligible in front of $n^{-\frac{8s}{4s+d}}$.
 
 Likewise, we find that \eqref{eq:2} is bounded by $C_1L^2R^4 n^{-2}D_3^{d}$, which is of order  $n^{-\frac{8s}{4s+d}}$.
 
We  apply \Cref{lem:auxiliary_regular}  to  bound \eqref{eq:3} by
    \begin{equation}
        \frac{R^3C_0 L^2}{n^3} D_3^{2s} \max( D_2^{2d}, D_3^{d}D_1^{d} ) \lesssim D_3^{2s} n^{-\frac{12s}{4s+d}} = n^{-\frac{8s}{4s+d}}.
    \end{equation}
    
   In total, we have shown that both the squared bias and all the variance terms are at most of order $n^{-\frac{8s}{4s+d}}$, up to a multiplicative constant depending on $R$ and $L$. This concludes the proof. Remark that there are no efficiency result to prove as we assume that $s\leq  d/4$.

\subsection{Debiasing}\label{sec:step4}
Let now $\mathrm{T}$ be a general $s$-regular functional. We split the dataset into two subsets of size respectively $n_1$ and $n_2$. 
We use the second half of the dataset to build a  density estimator $\hat f$. By \Cref{cor:minimax_standard}, if $f\in \Theta_0$, there exists an estimator $\hat f$ that satisfies: (i) $\hat f\in \Theta$, and (ii)  for all $0\leq t\leq s$, there exists $C$ 
 such that 
\begin{equation}
    \E[\|\hat f-f\|^m_{\cC^{t}}] \leq C \p{\frac{\log n_2}{n_2}}^{-\frac{m(s-t)}{2s+d}}
\end{equation}
for some arbitrary integer $m\geq 1$ to fix.

Let $s'$ and $\eta$ be given by  \Cref{def:regular_functional} and 
consider the event $E$ where $\|f-\hat f\|_{\cC^{s'}}\leq \eta$.   By Markov's inequality, it holds that 
\begin{equation}\label{eq:bound_E}
    \P(E^c) \leq \frac{\E[\|\hat f-f\|^m_{\cC^{s'}}]}{\eta^m} \leq  \frac{C}{\eta^m}  \p{\frac{\log n_2}{n_2}}^{-\frac{m(s-s')}{2s+d}} \leq C'n_2^{-2-6s-6d},
\end{equation}
where we choose $m$ such that $m(s-s')> (2s+d)(6s+6d+2)$ and the constant $C'$ depends on $R$, $m$,  $s$, $s'$,  $d$ and $\eta$.
Remark that if $E$ is satisfied, then  we have a Taylor expansion of the form
\begin{equation}\label{eq:Taylor_hatf}
    {\mathrm{T}}_{f} = {\mathrm{T}}_{\hat f} + \sum_{J=1}^{J_{\max}} {\mathrm{T}}_{\hat f}^{(J)}[f-\hat f,\dots,f-\hat f] +R_{\hat f}(f-\hat f),
\end{equation}
where $|R_{\hat f}(f-\hat f)|\leq L\|f-\hat f\|_{\L^\infty}^{J_{\max}+1}$. 
\medskip

Assume first that $s\leq d/4$, so that $J_{\max}=3$. The Taylor expansion is written as
\begin{equation}
    {\mathrm{T}}_{f} = A^{(0)} + A^{(1)}[f]+ A^{(2)}[f,f] + A^{(3)}[f,f,f]+R_{\hat f}(f-\hat f),
\end{equation}
where the terms $A^{(J)}$, $0 \leqslant J \leqslant 3$, are defined by
\begin{align*}
     A^{(0)}&= {\mathrm{T}}_{\hat f} - {\mathrm{T}}^{(1)}_{\hat f}[\hat f]+ T^{(2)}_{\hat f}[\hat f,\hat f]- T^{(3)}_{\hat f}[\hat f,\hat f,\hat f],\\
     A^{(1)}[\cdot]&= {\mathrm{T}}^{(1)}_{\hat f}[\cdot]-2T^{(2)}_{\hat f}[\hat f,\cdot] +3T^{(3)}_{\hat f}[\hat f,\hat f,\cdot],\\
     A^{(2)}[\cdot,\cdot] &= T^{(2)}[\cdot,\cdot]-3T^{(3)}_{\hat f}[\hat f,\cdot,\cdot],\\
     A^{(3)} &= T^{(3)}.
 \end{align*}
Conditionally on the second half of the sample, we use the first half to build estimators $\hat A^{(J)}$ of $A^{(J)}[f,\cdots,f]$ for $1\leq J\leq 3$. As the ${\mathrm{T}}_{\hat f}^{(J)}$s ($1\leq J\leq 3$) are $s$-regular forms of order $\ell$, so are the $A^{(J)}$s, with their norms being of the same order as the norms of the ${\mathrm{T}}_{\hat f}^{(J)}$s, up to a multiplicative constant of order $1+\|\hat f\|_{\cC^s}^3$. 
Thus, conditionally on the second half of the dataset, we may estimate each term $A^{(J)}$ using an estimator $\hat A^{(J)}$ as defined in one of the three first parts of the proof, with the risk of the estimator of order $(1+\|\hat f\|_{\cC^s}^3)( n_1^{-\frac{4s}{4s+d}}+n_1^{-\frac 12})$ up to a multiplicative constant depending only on $L$ and $R$.  Our final estimator is ${\hat{\mathrm T}} = A^{(0)} +\hat A^{(1)}+ \hat A^{(2)}+\hat A^{(3)}$.  We control the risk of ${\hat{\mathrm T}}$:
\begin{align*}
    \E[({\hat{\mathrm T}}-{\mathrm{T}}_f)^2]\leq \E[({\hat{\mathrm T}}-{\mathrm{T}}_f)^2\ones\{E\}]+ \E[({\hat{\mathrm T}}-{\mathrm{T}}_f)^2\ones\{E^c\}].
\end{align*}
Let us first bound $\E[({\hat{\mathrm T}}-{\mathrm{T}}_f)^2\ones\{E\}]$. As the Taylor expansion holds on $E$, we can apply the risk bounds obtained in the three first steps (conditionally on the second half the sample), leading to
    \begin{align*}
        &\E[({\hat{\mathrm T}}-{\mathrm{T}}_f)^2\ones\{E\}]\\
        &\leq 16\sum_{J=1}^3 \E[(\hat A^{(J)}-A^{(J)}_{\hat f}[f,\dots,f])^2] + 16L^2 \E[\|\hat f-f\|_{\L^\infty}^8] \\
        &\leq C_1 (1+\E[\|\hat f\|_{\cC^s}^6])\p{n_1^{-\frac{8(s-\ell)}{4s+d}}+n_1^{-1}}+ 16L^2 \E[\|\hat f-f\|_{\L^\infty}^8]
    \end{align*}
    for some constant $C_1$.
    It holds that $\E[\|\hat f\|_{\cC^s}^4]\lesssim 1$. Furthermore, the embedding $\cC^\eps\hookrightarrow \L^\infty$ for $\eps>0$ (see \Cref{app:besov}) implies that  $$\E[\|\hat f-f\|_{\L^\infty}^6]\lesssim (\log n_2/n_2)^{-\frac{8(s-\eps)}{2s+d}},$$
     which is neligible in front of $n_2^{-\frac{8s}{4s+d}}$ for $\eps$ small enough.
    
To bound $\E[({\hat{\mathrm T}}-{\mathrm{T}}_f)^2\ones\{E^c\}]$, we require a control on the $\L^\infty$-norm of the estimator.  By assumption, the functional $\mathrm{T}$ is bounded by $L$ and  ${\mathrm{T}}_{\hat f}^{(J)}$ is continuous over $(\cC^s)^*\times (\cC^s)^{J-1}$ with norm smaller than $L$. In particular, $$|{\mathrm{T}}_{\hat f}^{(J)}[u_1,\dots,u_J]|\leq L \|u_1\|_{(\cC^s)^*}\prod_{j=2}^J \|u_j\|_{\cC^s} \leq L \prod_{j=1}^J \|u_j\|_{\cC^s}.$$
Given the expression of ${\hat{\mathrm T}}$ and \emph{\ref{K4}}, as all integers $D_j$ involved in the definitions of the estimators of multilinear forms in the three first steps are at most of order $D_{\max} =\cO (n_1^{\frac{2}{4s+d}})=o( n_1^2)$, the difference $|{\mathrm{T}}_f-{\hat{\mathrm T}}|$ is  bounded by an expression of the form:
\begin{equation}\label{eq:linfty_bound}
\begin{split}
    C_2 L(1+ \sup_{x} \|\pi_{D_{\max}}\delta_x\|_{\cC^s}^3)&\leq C_2 L(1+C_{\star}\kappa_{\infty,d,s}(D_{\max})^3)\\
    &\leq  C_3Ln_1^{3s+3d}
    \end{split}
\end{equation}
for some constants $C_2,C_3$. 
Equation \eqref{eq:bound_E} therefore gives
    \begin{align*}
        \E[({\hat{\mathrm T}}-{\mathrm{T}}_f)^2\ones\{E^c\}]&\leq C_3^2L^2n_1^{6s+6d}\P(E^c)\leq C_4 L^2 n_1^{6s+6d} n_2^{-2-6s-6d}.
    \end{align*}
    We choose $n_1$ and $n_2$ of order $n$ so that the last display is at most of order $n^{-2}\ll n^{-1}$. Thus, the term $\E[({\hat{\mathrm T}}-{\mathrm{T}}_f)^2\ones\{E^c\}]$ is negligible and the risk of ${\hat{\mathrm T}}$ is of order $n^{-\frac{4s}{4s+d}}$. This concludes the proof in the regime $s\leq d/4$.

The proof is similar (and simpler) in the regime $s>d/4$, so we skip the details. The only difference is that  we  use a Taylor expansion up to $J_{\max}=2$ instead of $J_{\max}=3$, with a remainder term of order $\|\hat f-f\|_{\L^\infty}^3$. 
 As $s>d/4$,  this remainder term is negligible in front of $n_2^{-\frac{4s}{4s+d}}$.

It remains to show that efficient estimation is possible in the regime $s>d/4$. Let $n_2$ be of order $n^{1-\eps}$ for some $\eps>0$ to fix, and let $n_1=n-n_2\sim n$. Define $\psi_f$ as in the statement of \Cref{thm:functional_estimation} and remark that the efficient influence function as defined in \eqref{def:efficient} is indeed equal to $\psi_f$. We must therefore show that
\begin{equation}
    Z_2 = {\hat{\mathrm T}}-{\mathrm{T}}_f - \frac{1}{n}\sum_{i=1}^n (\psi_f(X_i)-\E[\psi_f(X_i)])
\end{equation}
is negligible in front of $n^{-\frac 12}$. We have the inequality $    \Var(\psi_f(X))\leq R\|\psi_f\|^2_{\L^2}\leq C_5R\|\psi_f\|^2_{\cC^{s}}\leq C_5RL^2$ for some constant $C_5$. Using the $\L^\infty$-bound \eqref{eq:linfty_bound} on $|{\hat{\mathrm T}}-{\mathrm{T}}_f|$ and Cauchy-Schwarz inequality, we find that for all integers $m\geq 1$ and some constant $C_3$,
\begin{align*}
    \E[|Z_2|\ones&\{E^c\}] \leq  C_3Ln_1^{3s+3d}\P(E^c) +L\sqrt{\frac {C_5R}n \P(E^c)} \\
    &\lesssim  n_1^{3s+3d}\p{\frac{\log n_2}{n_2}}^{-\frac{m(s-s')}{2s+d}} + n^{-1/2} \p{\frac{\log n_2}{n_2}}^{-\frac{m(s-s')}{4s+2d}}.
\end{align*}
As $n_2$ is of order $n^{1-\eps}$, for $m$ large enough with respect to $\eps$, this quantity is negligible in front of $n^{-\frac 12}$. 

Let us now show that $\E[|Z_2|\ones\{E\}]$ is also negligible in front of $n^{-\frac 12}$. We have
\begin{align*}
    ({\hat{\mathrm T}}-{\mathrm{T}}_f)\ones\{E\} = (\hat A^{(1)}-A^{(1)}[f]+ \hat A^{(2)}-A^{(2)}[f,f] + Z_3)\ones\{E\},
\end{align*}
where $\E[|Z_3|]\leq L \E[\|f-\hat f\|_{\L^\infty}^3]\lesssim (\log n_2/n_2)^{-\frac{3(s-\delta)}{2s+d}}$ for any $\delta>0$. In particular, $\E[|Z_3|]$ is negligible in front of $n^{-1/2}$ if $\eps$  is chosen small enough. Let $\phi_f$ be such that ${\mathrm{T}}_f^{(2)}[f,u]= \int u\phi_f$ for all $u\in \cC^s$. As $s>d/4$, by \eqref{eq:efficient_bi} and the definition of the estimator in the linear case, the estimators $\hat A^{(1)}$ and $\hat A^{(2)}$ satisfy
\begin{align*}
    \hat A^{(1)}-A_{\hat f}^{(1)}[f] &= \frac{1}{n_1}\sum_{i=1}^{n_1} (\psi_{\hat f}(X_i)-2\phi_{\hat f}(X_i)- \dotp{f,\psi_{\hat f}-2\phi_{\hat f}}) \\
    \text{and} \qquad \hat A^{(2)}-A_{\hat f}^{(2)}[f,f] &= \frac{1}{n_1}\sum_{i=1}^{n_1} (2\phi_{\hat f}(X_i)- \dotp{f,2\phi_{\hat f}}) + Z_1,
\end{align*}
where  $\E[|Z_1|^2]$ is negligible in front of $n_1^{-\frac 12}\sim n^{-\frac 12}$.
Thus, 
\begin{equation}\label{eq:twomoresteps}
     ({\hat{\mathrm T}}-{\mathrm{T}}_f)\ones\{E\} = \p{\frac{1}{n_1}\sum_{i=1}^{n_1} (\psi_{\hat f}(X_i)-\dotp{f,\psi_{\hat f}})}\ones\{E\} + Z_4,
\end{equation}
with $\E[|Z_4|]$ negligible in front of $ n^{-\frac 12}$. 
\medskip

To conclude, it remains to replace $\hat f$ by $f$ and $n_1$ by $n$   in \eqref{eq:twomoresteps}. 
We first replace $\hat f$ by $f$ and write
\begin{align*}
    \frac{1}{n_1}\sum_{i=1}^{n_1} (\psi_{\hat f}(X_i)-\dotp{f,\psi_{\hat f}}) = \frac{1}{n_1} \sum_{i=1}^{n_1} (\psi_{ f}(X_i)-\dotp{f,\psi_{ f}}) + Z_5,
\end{align*}
where $Z_5 = \frac{1}{n_1} \sum_{i=1}^{n_1} (\psi_{\hat f}-\psi_f)(X_i) - \dotp{\psi_{\hat f}-\psi_f,f}$.  The random variable $Z_5$ is centered conditionally on the second half of the sample, so that we can use the continuity of $f\in \cC^{s'}\mapsto \psi_f\in \L^2$ to obtain
$$\E[|Z_5|^2] \lesssim \frac{1}{n_1} \E[\|\psi_f-\psi_{\hat f}\|_{\L^2}^2]=o(n_1^{-1}).$$
 As $n_1$ is of order $n$, $\E[|Z_5|]$ is negligible in front of $ n^{-\frac 12}$.

At last,
\begin{align*}
     \frac{1}{n_1} &\sum_{i=1}^{n_1} (\psi_{ f}(X_i)-\dotp{f,\psi_{ f}})- \frac{1}{n} \sum_{i=1}^{n} (\psi_{ f}(X_i)-\dotp{f,\psi_{ f}})\\
     &= \frac{1}{n}\sum_{i=n_1+1}^{n} (\psi_{ f}(X_i)-\dotp{f,\psi_{ f}}) + \p{\frac{1}{n}- \frac{1}{n_1}}\sum_{i=1}^{n_1} (\psi_{ f}(X_i)-\dotp{f,\psi_{ f}}).
\end{align*}
This random variable is centered, and its variance is bounded by
\[R\|\psi_f\|_{\L^2}^2 \p{\frac{n_2}{n^2}+ n_1\p{\frac{1}{n}-\frac{1}{n_1}}^2}.\] 
As $n\sim n_1$ and $n_2=o(n_1)$, the variance is negligible in front of $n^{-1}$. In total, we have shown that $\E[|Z_2|\ones\{E\}]= o(n^{-\frac 12})$, proving that ${\hat{\mathrm T}}$ is an asymptotically efficient estimator.

\begin{remark}\label{rem:polynomial_norm}
    It appears clearly in the last step that estimation is still possible if the norm $L$ of the functional $\mathrm{T}$ is allowed to \emph{depend polynomially with respect to $\|f\|_{\cC^s}$}. More precisely, the theorem still holds if there exists $ \widetilde L$ and $N\geq 0$ such that for all $f\in \Theta$, (i) $|{\mathrm{T}}_f|\leq \widetilde L(1+\|f\|_{\cC^s})^N$, (ii) $|R_f(h)|\leq  \widetilde L(1+\|f\|_{\cC^s})^N\|h\|^{J_{\max}+1}_{\L^\infty}$ for all $h=f'-f$ with $f'\in \Theta$ and $\|h\|_{\cC^{s'}}<\eta$, and the norm of ${\mathrm{T}}_f^{(J)}$ for $J\in [J_{\max}]$ is smaller than $\widetilde L(1+\|f\|_{\cC^s})^N $. Indeed, under such conditions, one can  easily modify the proof in Step 4 by applying Cauchy-Schwarz inequality: for instance, one can bound 
    \begin{align*}
      \E[|R_{\hat f}(f-\hat f)|] &\leq \E[\widetilde L(1+\|\hat f\|_{\cC^s})^N \|f-\hat f\|_{\L^\infty}^{J_{\max}+1}] \\
      &\leq \widetilde L \p{\E[(1+\|\hat f\|_{\cC^s})^{2N}] \E[\|f-\hat f\|_{\L^\infty}^{2J_{\max}+2}]}^{1/2}\\
      &\lesssim \p{\frac{\log n}{n}}^{-\frac{(J_{\max}+1)(s-s'-\delta)}{2s+d}}
    \end{align*}
    for any $\delta>0$.
    We will still call \textit{$s$-regular functionals} such functionals $\mathrm{T}$ whose norm grows polynomially with $\|f\|_{\cC^s}$. 
   This generalization will be helpful to estimate the eigenvalue functional in \Cref{sec:spectral_estimation}.
\end{remark}

\section{Spectral estimation of weighted Laplace operators}\label{sec:spectral_estimation}
We return in this section to the estimation of the spectrum of weighted Laplace operators $\Delta_f=\Delta+\alpha A_V$, where $V=\log f$ and $A_V=\nabla V\cdot \nabla$. In particular, we are interested in the stability of the spectrum of $\Delta_f$ under perturbations of~$f$. 

\subsection{Spectral theory of weighted Laplace operators}\label{sec:spectral_theory_intro}
For any positive density $f \in \cC^s$ for $s>1$ written as $f=e^{V}$, the operator $A_{V}$ is bounded $\H^2 \to \H^{\eps}$ for $\eps\in (0,s-1)$ according to \eqref{A3} below. Hence, $A_V$ is a compact operator $\H^2 \to \L^2$.
It follows that  $\Delta_f : \H^2 \to \L^2$ is a compact perturbation of $\Delta$, hence the spectrum of $\Delta_f$ is also discrete. The resolvent 
\[
R_f(z) = (z + \Delta_f)^{-1}, \quad z \notin \sigma(-\Delta_f)
\]
defines a family of operators $\L^2 \to \H^2$ which is meromorphic in $z \in \mathbb C$. Moreover, $-\Delta_f$ is self-adjoint on $\L^2(f^\alpha)$ and non-negative, hence $\sigma(-\Delta_f)$ consists in eigenvalues
\[
0 = \lambda_{0, f} < \lambda_{1,f} \leqslant \cdots \leqslant \lambda_{\ell, f} \to \infty.
\]
Besides, we have 
$
\|R_f(z)\|_{\L^2(f^\alpha)\to \L^2(f^\alpha)} = {\mathrm{dist}(z, \sigma(-\Delta_f))}^{-1}
$
and since $\L^2(f^\alpha)$ and $\L^2$ are equivalent with equivalence constant depending on $\alpha, \inf f$ and $\sup f$, one obtains
\begin{equation}\label{eq:boundrfzl2}
\|R_f(z)\|_{\L^2\to\L^2} \leqslant \frac{C}{\mathrm{dist}(z, \sigma(-\Delta_f))}, \quad z \notin \sigma(-\Delta_f),
\end{equation}
where $C$ depends on $\alpha, \inf f$ and $\sup f$. In particular this implies that each eigenvalue is semi-simple (that is, there are no Jordan blocks), and for each $\lambda \in \sigma(-\Delta_f)$, we have
\begin{equation}\label{eq:laurent}
R_f(z)= S_{\lambda, f}(z) + \frac{\Pi_{\lambda, f}}{z - \lambda},
\end{equation}
where  $S_{\lambda, f}(z)$ is holomorphic near $z = \lambda$
and
$$
\Pi_{\lambda, f} = \frac{1}{2\pi i} \int_{\Upsilon_\lambda} R_f(z) \dd z
$$
is the spectral projector on the eigenspace associated with $\lambda$. Here $\Upsilon_\lambda$ is a small circle around $\lambda$ enclosing no other eigenvalue of $-\Delta_f$. Note also that we have the relations
\begin{equation}\label{eq:spicommute}
S_{\lambda, f}(z) \Pi_{\lambda, f} = \Pi_{\lambda, f}S_{\lambda, f}(z) = 0.
\end{equation}

\subsection{Resolvent estimates for weighted Laplace operators}

We  first control the operator norms of the operator $A_V$ and of the resolvent of $\Delta_f$.

\begin{lemma}\label{lem:operator_norm_Ah}
    Let $t\in \R$ and let $1< p, q,r<\infty$ with $\frac 1p+\frac 1r=\frac 1q$. Then, there exists $C>0$ such that 
    \begin{align}
       \|A_V\|_{\W^{1+t,p}\to \W^{t,q}} &\leqslant C \|V\|_{\W^{1+|t|,r}}, &  V &\in \W^{1+|t|,r},&  \tag{A1}\label{A1}\\
       \|A_V\|_{\cC^{1+t}\to \cC^{t}} &\leqslant C \|V\|_{\cC^{1+t}},  & V & \in \cC^{1+t},& t & > 0. \tag{A2}\label{A2} 
    \end{align}
 If $p=q$ and $t\in \N$, \eqref{A1} holds with the $\W^{1+|t|,r}$-norm replaced with the $\Cf^{1+|t|}$-norm. 
 \end{lemma}
\begin{proof}
     The fractional Sobolev space $\W^{t,p}$ on $M$ is equivalent to the Triebel–Lizorkin space $\mathrm{F}^{t,p}_2$  which is defined through a resolution of unity \cite[Chapter 7]{triebel1992theory}. This implies that it suffices to treat the case $M=\R^d$. We first prove \eqref{A1} for $t\geq 0$. Note that for $t\in \N$, by Leibniz rule and H\"older's inequality, $\|uv\|_{\W^{t,q}}\lesssim \|u\|_{\W^{t,p}}\|v\|_{\W^{t,r}}$. The inequality is then extended to $t\geq 0$ by complex interpolation \cite{lunardi2018interpolation}. Thus, for $\phi \in \W^{1+t,p}$,
    \begin{align*}
        \| A_V\phi\|_{\W^{t,q}}&= \|\nabla V\cdot \nabla \phi \|_{\W^{t,q}}\lesssim \|\nabla V\|_{\W^{t,r}} \|\nabla \phi\|_{\W^{t,p}} \lesssim \|V\|_{\W^{1+t,r}}  \| \phi\|_{\W^{1+t,p}}.
    \end{align*}
     Equation  \eqref{A1} for $t<0$  is similarly obtained by using the dual definition of the $\W^{-t,q}$-norm. The case $p=q$ is easily proven when $t\in \N$. \Cref{A2} and \Cref{A3} follow from $\cC^t$ and $\Cf^t$ being  multiplication algebras, see \Cref{app:besov}. We leave details to the reader.
\end{proof}

\begin{corollary}\label{cor:operator_norm_Ah}
 Let $t\in \R$, $t'>|t|$ and let $1< p<\infty$. Then, there exists $C>0$ such that for all $V\in \cC^{1+t' }$,
    \begin{align}
       &\|A_V\|_{\W^{1+t,p}\to \W^{t,p}} \leq C \|V\|_{\cC^{1+t'}}, \tag{A3}\label{A3}
    \end{align}
\end{corollary}
\begin{proof}
   The previous proof applies in the case $p=q$ to show that for $t\in \R$
   \[
   \|A_V\|_{\W^{1+t,p}\to \W^{t,p}} \leq C \|V\|_{\W^{1+t,\infty}},
   \]
   with the norm $\|V\|_{\W^{1+t,\infty}}$ appearing in the upper bound, where $\|u\|_{\W^{s,\infty}}=\|\Lambda^s u\|_{\L^\infty}$ by definition. For $s'>s>0$, the fact that $\Lambda^s:\cC^{s'}\to \cC^{s'-s}\hookrightarrow \L^\infty$ implies the embedding $\cC^{s'}\hookrightarrow \W^{s,\infty}$, see \cite[Eq. (8.11)]{taylor1996partial}. This  concludes the proof.
\end{proof}

Now, we turn to the estimation of the resolvent $R_f(z)$.

\begin{proposition}\label{prop:resolvbounded}
Let $s > 1$,\, $t \in \left]-s, s - 1\right[$,\, $q > 1$ and 
$L>0$. Then there are $C, N > 0$ such that for each $f=e^{V} \in \cC^{s}$ with $\|V\|_{\L^\infty}\leq L$, one has
\begin{equation}\label{cor:estresolv}
\left\|R_f(z)\right\|_{\W^{t,q} \to \W^{t + 2, q}} \leq \frac{C (1 + \|V\|_{\cC^s}+|z|)^N}{\dist(z, \sigma(-\Delta_f))}, \quad z \notin \sigma(-\Delta_f).
\end{equation} 
The same bound  holds for $t=-s$ or $t=s-1$ if $s$ is an integer, with $\|V\|_{\cC^s}$ replaced by $\|V\|_{\Cf^s}$ in the right-hand side. Moreover, 
the  bound also holds for $\left\|R_f(z)\right\|_{\cC^{t} \to \cC^{t + 2}}$ for $t\in \left[0,s-1\right]$.
\end{proposition}

\begin{proof}
We write $V=\log f$. We first give a bound for the operator norm of $R_f(z)$ seen as an operator $\L^2\to \H^1$. Let $v\in \L^2$ and $u=R_f(z)v$, so that $zu+\Delta u + \alpha \nabla V\cdot \nabla u=v$. We compute
\begin{align*}
    \dotp{v,u} &= \Re(z)\|u\|^2_{\L^2} + \dotp{\Delta u, u} + \alpha \dotp{\nabla V\cdot \nabla u,u}.
\end{align*}
So, by Cauchy-Schwarz inequality,
\begin{equation}
 \|\nabla u\|^2_{\L^2} \leq   |z|\|u\|_{\L^2}^2+  \|v\|_{\L^2}\|u\|_{\L^2}  + \alpha\|\nabla V\|_{\L^\infty}\|\nabla u\|_{\L^2} \|u\|_{\L^2}.
\end{equation}
By \eqref{eq:boundrfzl2}, we have
$$
\|u\|_{\L^2} \leqslant \frac{C\|v\|_{\L^2}}{\delta(z)}
$$
where $C \geqslant 1$ depends on $\inf f$ and $\sup f$ (which are controlled in terms of $\|V\|_{\L^\infty}$) and $\delta(z) = \dist(z, \sigma(-\Delta_f))$.
Thus one gets
\begin{equation}
  \|\nabla u\|^2_{\L^2} \leq {C^2\delta(z)^{-2}|z|\|v\|_{\L^2}^2+  C\delta(z)^{-1}\|v\|^2_{\L^2}} +\alpha C\delta(z)^{-1}{\|\nabla V\|_{\L^\infty}\|\nabla u\|_{\L^2} \|v\|_{\L^2}}.
\end{equation}
The latter inequality implies
$$
 \|\nabla u\|_{\L^2} \leqslant C \delta(z)^{-1} (\sqrt{\delta(z) + |z|} + \alpha\|\nabla V\|_{\L^\infty}) \|v\|_{\L^2}.
$$
This implies that $R_f(z):\L^2\to \H^1$ is continuous with the bound
\begin{equation}\label{eq:firstboundrfz}
\|R_f(z)\|_{\L^2 \to \H^1} \leqslant C \delta(z)^{-1} (\sqrt{\delta(z) + |z|} + \alpha\|\nabla V\|_{\L^\infty}),
\end{equation}
for some $C$ depending on $\alpha, \inf f$ and $\sup f$. 

We now bound the norm of ${R_f(z) : \L^2 \to\H^2}$. One has the identity
\begin{equation}\label{eq:resolv2}
    R_f(z) =- \Lambda^{-2}+\Lambda^{-2}(1 + z + \alpha A_{V})R_f(z).
\end{equation}
The operator $\Lambda^{-2}$ is bounded $\L^2 \to \H^2$ with operator norm $1$. Moreover, since $V\in \Cf^1$, $A_{V}$ is bounded $\H^1 \to \L^2$ by \Cref{lem:operator_norm_Ah}. Recalling \eqref{eq:firstboundrfz} one obtains
\begin{equation}\label{eq:secondboundrfz}
\|R_f(z)\|_{\L^2 \to \H^2} \leqslant 1 + \widetilde{C}\delta(z)^{-1}(1 + |z|+ \alpha\|\nabla V\|_{\L^\infty})(\sqrt{\delta(z) + |z|}+\alpha\|\nabla V\|_{\L^\infty})
\end{equation}
for some $\widetilde C$ depending on $\|V\|_{\L^\infty}$.

From now on, we do not keep track of the constants and we only explain the boundedness properties of $R_f(z)$. We iterate the above process. For $t\in \left]0,1\right[$ with $t<s-1$, by \eqref{A3},
\[
\L^2 \xrightarrow{R_f(z)} \H^2 \xrightarrow{A_V}  \H^t  \xrightarrow{\Lambda^{-2}}  \H^{t+2}
\]
so that, as $\H^t\hookrightarrow \L^2$, the operator $\Lambda^{-2}(1 + z + \alpha A_{V})R_f(z)$ is bounded $\H^t\to \H^{t+2}$. So is $\Lambda^{-2}$. Thus, by \eqref{eq:resolv2}, $R_f(z)$ is bounded $\H^t\to \H^{t+2}$. This process can be iterated, proving that $R_f(z)$ is bounded $\H^t\to \H^{t+2}$ for all $t\in \left[0,s-1\right[$.

For negative values of $t$, we use the following identity between operators $\L^2 \to \H^2$:\begin{equation}\label{eq:resolv1}
    R_f(z) = -\Lambda^{-2}+R_f(z)(1 + z +\alpha A_{V})\Lambda^{-2}.
\end{equation}
For $t\in \left]-1,0\right]$, by \eqref{A3},
\[
\H^t \xrightarrow{\Lambda^{-2}} \H^{t+2} \xrightarrow{A_V}  \L^2  \xrightarrow{R_f(z)}  \H^{2}.
\]
As $\H^2\hookrightarrow \H^{t+2}$, \eqref{eq:resolv1} shows that $R_f(z) $ is written as the sum of two bounded operators $\H^t\to \H^{t+2}$, and is therefore bounded $\H^t\to \H^{t+2}$. Once again, this process can be iterated using \eqref{A3}, proving that $R_f(z)$ is bounded $\H^t\to \H^{t+2}$ for all $t\in \left]-s,0\right]$. This concludes the case $q = 2$. If $s$ is an integer and $V\in \Cf^s$, we can use the edge case in \Cref{lem:operator_norm_Ah} to show that the estimates are still valid for $t=-s$ and $t=s-1$.

Next, since $\Lambda^{-2}$ is a pseudo-differential operator of order $-2$, it maps $\W^{t, q}$ to $\W^{t+2, q}$ continuously for each $t \in \R$ and $1 < q < \infty$, see \cite[Proposition 6.5]{taylor1996partial}. Hence we are left with analyzing $\Lambda^{-2}(1 + z + \alpha A_V)R_f(z)$. Let $q_1 = \p{\frac 12-\frac 1d}_+^{-1}$. 
For any $2\leq q< q_1$, we have $R_f(z) : \L^2 \to \H^2$ and $\H^2\hookrightarrow \W^{1,q}$ by a Sobolev embedding, see \Cref{app:besov}. Thus, for any $0\leq t < 1$ with $t<  s-1$, as $A_V:\W^{1,q}\to \W^{t-1,q}$ by \eqref{A3}, 
 we have 
$$
\Lambda^{-2}(1 + z + \alpha A_V)R_f(z) : \L^2 \to \W^{t + 1, q}.
$$
By the  embedding $\W^{t,q}\hookrightarrow \L^q\hookrightarrow \L^2$,  $R_f(z)$ maps $\W^{t, q}$ in  $\W^{t+ 1, q}$ by \eqref{eq:resolv2} and injecting this again in \eqref{eq:resolv2}, as $A_V:\W^{t+1,q}\to \W^{t,q}$ by \eqref{A3}, we get that for any $t\in [0,1)$ with $t<  s-1$
$$
R_f(z) : \W^{t, q} \to \W^{t + 2,q}.
$$
Let $q_2=\p{\frac 1{q_1}-\frac 1{d}}_+^{-1}$. Iterating the argument we obtain $R_f(z) : \W^{t, q} \to \W^{t + 2, q}$ for each $2\leq q< q_2$ %$q \geq 2$ 
and $t \in [0, 1)$ such that $t < s - 1$. This defines a sequence $q_1,q_2,\dots$, with $q_n=+\infty$ after  $\lceil d/2\rceil$  steps. Thus, $
R_f(z) : \W^{t, q} \to \W^{t + 2, q}
$ for all $q\geq 2$.
  Injecting this in \eqref{eq:resolv2} we obtain the same mapping property, if we assume $t \in \left[1,2\right[$ 
and $t < s - 1$. Again an induction argument allows to obtain the sought result for $t \in \left[0, s - 1 \right[$. The same property can be shown with the H\"older-Zygmund space $\cC^t$ for $t\in [0,s-1]$ using \eqref{A2} and the fact that $\Lambda^{-2}$ is bounded $\cC^t\to \cC^{t+2}$, see \cite[Eq. (8.11)]{taylor1996partial}. 

For $t \in \left]-s, 0\right]$, one proceeds similarly, using \eqref{eq:resolv1} instead of \eqref{eq:resolv2}. Finally, recalling \eqref{eq:secondboundrfz} and keeping track of the constants in the above development, one sees that \eqref{cor:estresolv} holds.
\end{proof}

\subsection{Perturbation theory for functions of weighted Laplace operators}\label{sec:perturbation_function}
The aim of this paragraph is to derive perturbative results for the estimation of a function of $\Delta_f$. For  a fixed function
  $\chi \in \Cf^\infty_c(\R)$,  define the operator $\chi(-\Delta_f)$  thanks to the spectral theory of $-\Delta_f$ by
$$
\chi(-\Delta_f)u = \sum_{\ell = 0}^\infty \chi(\lambda_{\ell,f}) \langle \varphi_\ell, u \rangle_f \varphi_{\ell}, \quad u \in \L^2,
$$
where $(\varphi_\ell)_{\ell\geq 0}$ is an orthonormal basis of eigenfunctions of $-\Delta_f$. 

The main result of this section is the following perturbation result.

\begin{proposition}\label{prop:perturbativechi}
Let  $\delta > 0$ and $p, q, r > 1$ such that
$$
\frac{1}{p} + \frac{1}{r} = \frac{1}{q}.
$$
Then there is $C,N > 0$ such that for any $L>0$, $f, h \in  \Cf^2$ and $\chi \in \Cf^\infty_c(\R)$ satisfying the bounds $\|f\|_{\Cf^2}$, $\|h\|_{\Cf^2}$, $\|\chi\|_{\Cf^3} < L$,\, $\operatorname{supp}(\chi) \subset [-L, L]$, $\inf f, \inf h > \delta$ and $\sup f,\sup h<\delta^{-1}$, there holds
$$
\|\chi(-\Delta_f) - \chi(-\Delta_{h})\|_{\W^{1,p} \to \L^q} \leq C (1+L)^N\|f - h\|_{\W^{-1,r}}.
$$
\end{proposition}

\begin{proof}
We write $f=e^{V}$ and $h=e^{U}$. 
The proof is based on the Helffer--Sj\"ostrand formula, which writes \cite{davies1995spectral}
\begin{equation}\label{eq:hs}
\chi(-\Delta_f) = \frac{1}{2\pi i} \int_{\mathbb C} \partial_{\bar z} \widetilde \chi(z) R_f(z) \dd \bar z \wedge \dd z
\end{equation}
for any quasi-analytic extension $\widetilde \chi \in \Cf^\infty_c(\mathbb C)$ of $\chi$, which means that $\widetilde\chi$ is a function  such that there exists an integer  $N \geq 1$ and $C>0$ such that
\begin{equation}\label{eq:qa}
|\partial_{\bar z}\widetilde \chi(z)| \leq C |\operatorname{Im}z|^N, \quad z \in \mathbb C.
\end{equation}
In what follows we make the choice
\begin{equation*}
\widetilde \chi(z) = \varrho(\operatorname{Im}z) \sum_{k = 0}^2\frac{\chi^{(k)}(\Re z)}{k!}(i \Im z)^k
\end{equation*}
for some function $\varrho \in \Cf^\infty_c(\R)$ which is equal to $1$ near $0$ and vanishes outside $[-1, 1]$. It is immediate to check that \eqref{eq:qa} holds for $\widetilde \chi$ with $N = 2$, and that we have
\begin{equation}\label{eq:tildechichi}
|\partial_{\bar z} \widetilde \chi (z)| \leq C_0  |\operatorname{Im}z|^2
\end{equation}
for some constant $C_0$ depending on $L$ and $\varrho$. 
As 
\begin{equation*}
\|R_f(z)\|_{\L^2 \to \L^2} \leq C_1 |\operatorname{Im}z|^{-1}, \quad z \in \mathbb C \setminus \R.
\end{equation*}
 for some constant $C_1$ (and the same control holds for $h$), 
we may write 
\begin{equation*}
\chi(-\Delta_f) - \chi(-\Delta_{h}) = \frac{\alpha}{2\pi i} \int_{\mathbb C} \partial_{\bar z} \widetilde \chi(z) R_f(z)A_{U-V}R_h(z) \dd \bar z \wedge\dd z.
\end{equation*}
Recalling \eqref{eq:tildechichi} one sees that the quantity $\|\chi(-\Delta_f) - \chi(-\Delta_{h})\|_{\W^{1,p} \to \L^q}$ is bounded from above up to a multiplicative constant depending on $L$ and $\alpha$ by
\begin{equation*}
\sup_{z \in \operatorname{supp} \widetilde \chi} |\Im z|^2  \|R_f(z)\|_{\W^{-2,q} \to \L^q} \|A_{U-V}\|_{\W^{3,p} \to \W^{-2,q}} \|R_h(z)\|_{\W^{1,p} \to \W^{3,p}}.
\end{equation*}
Since $\operatorname{supp} \widetilde \chi \subset \operatorname{supp} \chi \times [-1,1]$ and $z\in \C$ is at distance at least $|\Im z|$ from $\sigma(-\Delta_f)$ and $\sigma(-\Delta_h)$,  \Cref{cor:estresolv} implies that there is a constant $C_3$ depending on $L, \alpha, p$ and $q$ such that
\begin{equation*}
\sup_{z \in \operatorname{supp} \widetilde \chi} |\Im z|^2 \|R_f(z)\|_{\W^{-2,q} \to \L^q} \|R_h(z)\|_{\W^{1,p} \to \W^{3,p}} \leq C_2.
\end{equation*}
According to \eqref{A1} and the fact that $A_\phi\psi=A_\psi\phi$, $\|A_{U-V}\|_{\W^{3,p} \to \W^{-2,q}} \lesssim \|U-V\|_{\W^{-1,r}} \lesssim \|f - g\|_{\W^{-1,r}}$ because $f$ and $g$ are bounded from below and are of class $\Cf^1$. As it is clear from \Cref{cor:estresolv} that all constants appearing in the different estimates depend polynomially on $L$, this concludes the proof.
\end{proof}

\begin{remark}\label{remark:epsilon1}
Let $t\in [0,2]$. As $\|A_{U-V}\|_{\W^{3-t,p}\to \W^{t-2,q}}\lesssim \|U-V\|_{\W^{t-1,r}}$ according to \eqref{A1}, 
reproducing the proof shows that under the same hypothesis as in Proposition \ref{prop:perturbativechi}, one has the estimate
$$
\|\chi(-\Delta_f) - \chi(-\Delta_{h})\|_{\W^{1-t,p} \to \W^{t,q}} \leq C(1+L)^N \|f - h\|_{\W^{t - 1, r}}.
$$
\end{remark}

\subsection{Perturbation theory for spectral projectors}\label{sec:perturbation_projector}
In this paragraph, we show some perturbation results about spectral projectors for $\Delta_f$ as a direct application of the preceding section. Let $\Upsilon$ be a simple loop in $\mathbb C$, with counterclockwise orientation, bounding a compact domain $\Omega \subset \mathbb C$. {We} assume that 
$$
\partial \Omega \cap \sigma(-\Delta_f) = \emptyset.
$$
Recall that the spectral projector $\Pi_{\Upsilon,f}$ of $\Delta_f$ associated with $\Upsilon$ is given by
$$
\Pi_{\Upsilon,f} = \frac{1}{2 \pi i} \int_{\Upsilon} R_f(z) \dd z,
$$
so that recalling \eqref{eq:laurent} one has
\begin{equation}\label{eq:sumprojectors}
\Pi_{\Upsilon,f} = \sum_{\lambda} \Pi_{\lambda, f}.
\end{equation}
Since $\Pi_{\Upsilon,f}^2 = \Pi_{\Upsilon,f}$, \eqref{cor:estresolv} yields that $\Pi_{\Upsilon,f}$ is bounded
\begin{equation}\label{eq:pibounded}
\Pi_{\Upsilon,f} : \W^{t_1, p} \to \W^{t_2, q}
\end{equation}
for any $p,q \in \left]1, +\infty\right[$ 
and $t_1, t_2 \in \left]-s, s+1\right[$, with operator norm depending on $p,q,s, \sup_{z \in \Omega}|z|$, $ \dist(\Upsilon, \sigma(-\Delta_f))$,  $\|f\|_{\cC^s}$ and $\inf f$. 
Moreover, using the boundedness of the resolvent over H\"older-Zygmund spaces, we find that  for any $p\in \left]1, +\infty\right[$  and $t \in \left]-s, s+1\right[$
\begin{equation}\label{eq:pibounded_bis}
\Pi_{\Upsilon,f} : \W^{t_1, p} \to \cC^{s+1}
\end{equation}
is also bounded. 
{We} denote by $E_{ \Upsilon,f} = \operatorname{ran} \Pi_{ \Upsilon,f}$ the associated eigenspace, so that
$$
E_{ \Upsilon,f} = \bigoplus_{\lambda \in \Omega} \operatorname{ker}(\lambda +\Delta_f).
$$
If $h \in \cC^s$ also satisfies the assumptions above we set for $q \geq 2$
$$
D_{q}(\Pi_{ \Upsilon,f}, \Pi_{ \Upsilon,h}) = \left\| (1 - \Pi_{\Upsilon, h}) \Pi_{\Upsilon, f}\right\|_{\L^{2 }\to \mathrm{L}^q} + \left\| (1 - \Pi_{\Upsilon, f}) \Pi_{\Upsilon, h}\right\|_{\L^{2} \to \L^q}.
$$
The number $D_q(\Pi_{ \Upsilon,f}, \Pi_{ \Upsilon,h})$ is a measure of the $\L^q$-angle between the finite dimensional spaces $E_{ \Upsilon,f}$ and $E_{ \Upsilon,h}$. Note that we took the operator norms $\L^2 \to \L^q$ but this is equivalent to taking norms $\L^q \to \L^q$ by \eqref{eq:pibounded}. 
{We} have the following perturbation result, which is a consequence of Proposition \ref{prop:perturbativechi}.

\begin{proposition}\label{prop:eigenfunctionsperturbative}
Let   $L,\delta > 0$ and $1 < q<r < \infty$. Then there are $\eta> 0$, $C, N > 0$ and $0<\nu\leq\delta$ such that  for each simple loop $\Upsilon \subset \C$ and $f \in \Cf^2$ satisfying
\begin{equation}
\inf f > \delta, \quad \|f\|_{\Cf^1}< L, \quad \sup_{z\in \Upsilon} |\Re z| < L \quad \text{and} \quad \dist(\Upsilon, \sigma(-\Delta_f))> \delta,
\end{equation}
 and for any $h \in \Cf^2$ with  $\|f-h\|_{\Cf^1} < \eta$, the spectral projector $\Pi_{\Upsilon, h}$ is well-defined, has the same rank as $\Pi_{\Upsilon,f}$ and satisfies $\dist(\Upsilon, \sigma(-\Delta_h))> \nu$. Let  $\beta = \max(\|f\|_{\Cf^2},\|h\|_{\Cf^2})$. Then
\begin{equation*}
D_q(\Pi_{ \Upsilon,f}, \Pi_{ \Upsilon,h}) \leq C (1+\beta)^N\|f - h\|_{\W^{-1, r}}.
\end{equation*}
\end{proposition}

\begin{proof}
First, remark that the spectrum of $\Delta_f$ is included in $\R$, so that we can always make a deformation contour to assume that $\Upsilon$ is contained in $\{|\Im z|\leq 1\}$ without changing $\Pi_{\Upsilon,f}$, so that the following bounds only depend on $\sup_{z \in \Upsilon}|\Re z|$. 

Write $f=e^{V}$ and $h=e^{U}$. 
Since $f \in \Cf^2$, $R_f(-1)$ is bounded $\L^2 \to \H^2$ by \eqref{cor:estresolv}. Hence for $u \in \H^2$ we have
$$
\|u\|_{\H^2} \leq \|R_f(-1)\|_{\L^2 \to \H^2} \|(-1 + \Delta_f)u\|_{\L^2}.
$$
Thus there is a constant $C>0$ depending only on $\alpha$, $L$ and $\delta$ such that
\begin{equation}\label{eq:relbd1}
\|u\|_{\H^2} \leq C(\|u\|_{\L^2} + \|\Delta_fu\|_{\L^2}).
\end{equation}
Now, according to \Cref{lem:operator_norm_Ah}, for $u \in \H^2$ there holds 
\begin{equation}\label{eq:relbd2}
\|A_{U-V}u\|_{\L^2}\lesssim \|U-V\|_{\Cf^1} \|u\|_{\H^2}.
\end{equation}
We make use of classical perturbation theory: the relative bound estimates \eqref{eq:relbd1} and \eqref{eq:relbd2}, together with \eqref{cor:estresolv} and \cite[Theorem IV.3.3.18]{kato2013perturbation}, imply that there are $\eta, C_1 > 0$ depending on  $\alpha$, $L$ and $\delta$ 
such that, provided that we have $\|f-h\|_{\Cf^1} < \eta$, it holds that $\Pi_{\Upsilon,h}$ is well-defined, has the same rank as $\Pi_{\Upsilon,f}$, and moreover 
\begin{equation}\label{eq:distsp}
\|R_h(z)\|_{\L^2 \to \L^2} \leq C_1, \quad z \in \Upsilon.
\end{equation}
Since $R_h(z)\phi = (z-\lambda)^{-1}\phi$ if $\lambda \in \sigma(-\Delta_h)$ and $\phi$ is an associated eigenfunction, \eqref{eq:distsp} implies that $\dist(\Upsilon, \sigma(-\Delta_{h})) \geq C_1^{-1}$. We define $\nu=\min(C_1^{-1},\delta)/2$.  

Denote by $\Omega$ the open set bounded by $\Upsilon$. Then by what precedes, one sees that for any $h \in \Cf^2$ with $\|f-h\|_{\Cf^1} < \eta $, we have
$$
\Pi_{\Upsilon,h} = \chi(-\Delta_{h})
$$
where $\chi \in \Cf^\infty_c(\R)$ is any smooth function equal to $1$ on $\overline\Omega \cap \R$ and equal to zero on the set $\{\lambda \in \R~:~ \dist(\lambda, \Omega) > \nu / 2\}$. Remark that we can always make $\eta$ smaller so that the condition $\|f-h\|_{\Cf^1}<\eta$ implies that $\inf h\geq \delta/2$. In particular, Proposition \ref{prop:perturbativechi} implies that whenever  $\|f - h\|_{\Cf^1} < \eta$,
\begin{equation}\label{eq:pertub}
    \|\Pi_{\Upsilon,f} - \Pi_{\Upsilon,h}\|_{\W^{1,p} \to \L^q} \leq C_2 (1+\beta)^{N}\|f - h\|_{\W^{-1,r}}
\end{equation}
where $p$ is such that $1/p + 1/r = 1/q$ (such a $p$ exists since $r>q$). Here $C_2$ and $N$ depend on $\alpha$,\, $p$, \,$q$,\, $r$, $L$, $\delta$ and $\|\chi\|_{\Cf^3}$. However one can always chose $\chi$ so that $\|\chi\|_{\Cf^3}$ is a function of $\nu$. 
We write 
\begin{align*}
    \left\| (1 - \Pi_{\Upsilon, h}) \Pi_{\Upsilon, f}\right\|_{\L^{2}\to \mathrm{L}^q} \leq \|\Pi_{\Upsilon,f} - \Pi_{\Upsilon,h}\|_{\W^{1,p} \to \L^q}  \|\Pi_{\Upsilon,f} \|_{\L^2\to \W^{1,p}}.
\end{align*}
Recall now from \eqref{eq:pibounded} that $\|\Pi_{\Upsilon,f} \|_{\L^2\to \W^{1,p}}$ is bounded. By symmetry, the second term defining $D_q(\Pi_{\Upsilon,f},\Pi_{\Upsilon,h})$ is likewise bounded. This completes the proof.
\end{proof}

\begin{remark}\label{remark:H1_rate}
Taking Remark \ref{remark:epsilon1} into account, one gets, under the same hypothesis, 
\begin{equation}\label{eq:stability_Wt}
\|(1 - \Pi_{\Upsilon, h}) \Pi_{\Upsilon, f}\|_{\L^2 \to \W^{t,q}} + \|(1 - \Pi_{\Upsilon, f}) \Pi_{\Upsilon, h}\|_{\L^2 \to \W^{t,q}} \leq C \|f - h\|_{\W^{t-1, r}}
\end{equation}
for $t \in [0,1]$. \Cref{eq:stability_Wt} implies that  the $\W^{t,q}$-norm between eigenfunctions is controlled by the $\L^r$-norm between densities. Together with \Cref{app:density}, such a result can be seen to imply that eigenfunctions can be estimated at rate $\cO(n^{-\frac{s+1-t}{2s+d}})$ for $t\in [0,1]$ and $d\geq 3$, if the $\W^{t,q}$-norm is used as a loss function instead of the $\L^q$-norm. For $t=1$, we obtain the same rates of estimation as \cite{trillos2025minimax}, with the key difference that their estimators are given by the eigenvectors of a graph Laplacian matrix whereas our risk bounds are stated for plug-in estimators on a fixed manifold $M$.

Moreover, remark that we have $\W^{t, q} \hookrightarrow \L^\infty$ whenever $ t q > d$ (see \Cref{app:besov}). Hence,  we immediately get the following corollary.
\end{remark}

\begin{corollary}\label{cor:perturbation_infty}
Let  $\varepsilon > 0$, $r>d/\eps$ and $L, \delta > 0$. Let $\eta$ be the constant in \Cref{prop:eigenfunctionsperturbative}.  Then there are $C, N > 0$ such that for every simple loop $\Upsilon \subset \C$ and $f \in \Cf^2$ satisfying
\begin{equation*}
\inf f > \delta, \quad \|f\|_{\Cf^1}< L, \quad \sup_{z\in \Upsilon} |\Re z| < L \quad \text{and} \quad \dist(\Upsilon, \sigma(-\Delta_f))> \delta,
\end{equation*}
the following holds. For any $h \in \Cf^2$ with  $\|f-h\|_{\Cf^1} < \eta$, if $\beta = \max(\|f\|_{\Cf^2},\|h\|_{\Cf^2})$, 
we have the bound
$$
D_\infty(\Pi_{ \Upsilon,f}, \Pi_{ \Upsilon,h}) \leq C (1+\beta)^N\|f - h\|_{\W^{-1 + \varepsilon, r}}.
$$
\end{corollary}

This estimate in $\L^\infty$-norm will allow us to obtain rates of convergence for the \textit{empirical} $\L^2$-loss.

\subsection{Upper bound for the estimation of eigenvectors}

In this section, we combine results obtained concerning the perturbation theory of eigenprojectors to prove \Cref{thm:cv_spectral_eigenfunction}. 

\begin{proof}[Proof of \Cref{thm:cv_spectral_eigenfunction}]
Let $s\geq 2$,  $L,\delta>0$ and $r>q\geq 2$. Let $\Upsilon$ be a  curve as in \Cref{thm:cv_spectral_eigenfunction}. Let $\Theta_0 = \cP_s(\Upsilon,\delta,L)$. Let $\eta$, $\nu$ and $N$ be the constants of \Cref{prop:eigenfunctionsperturbative} associated to the constants $\alpha$, $q$, $r$, $\delta$ and $L$. 

We let $\Theta$ be the set of functions $h\in \Cf^2$ with  $\dist(\Upsilon, \sigma(-\Delta_h))\geq \nu$. We can rephrase \Cref{prop:eigenfunctionsperturbative} by saying that the $\eta$-neighborhood of $\Theta_0$ for the $\Cf^1$-norm is included in $\Theta$. In particular, the $\eta$-neighborhood of $\Theta_0$ for the $\cC^{1+\eps}$-norm is also included in $\Theta$, for any small number $0<\eps<1$.  Let $r_n$ be the rate appearing in \Cref{thm:cv_spectral_eigenfunction}. By \Cref{cor:minimax_standard}, we can design an estimator $\hat f$ of a density $f\in \Theta_0$ that satisfies: (i) $\hat f\in \Theta$ almost surely and (ii) for all integers $m\geq 1$, there exists $C>0$ (depending on $s$, $m$, $p$ and $L$) such that for all $0\leq t \leq 2$,
\begin{equation}\label{eq:rate_estimator}
    \E\left[\|\hat f-f\|_{\cC^t}^m \right]^{\frac 1m} \leq C\p{\frac{\log n}n}^{\frac{s-t}{2s+d}},\qquad \E\left[\|\hat f-f\|_{\W^{-1}_{p}}^2 \right]^{\frac 12} \leq Cr_n.
\end{equation}

Remark that as $\hat f\in \Theta$, $\Pi_{ \Upsilon,\hat f}$ is  well-defined. Let $E$ be the event where $\|\hat f-f\|_{\Cf^1}<\eta$. We decompose the risk
\begin{equation*}
    \E[D_q(\Pi_{ \Upsilon,f}, \Pi_{ \Upsilon,\hat f})] \leq \E[D_q(\Pi_{ \Upsilon,f}, \Pi_{ \Upsilon,\hat f})\ones\{E\}] +  \E[D_q(\Pi_{ \Upsilon,f}, \Pi_{ \Upsilon,\hat f})\ones\{E^c\}].
\end{equation*}
Let $\hat \beta = \max(\|f\|_{\Cf^2},\|\hat f\|_{\Cf^2})$. By \Cref{remark:moment_control_density}, the moments of $\hat \beta$ are all bounded up to constants depending  on $\|f\|_{\Cf^2}\leq L$. 
By \Cref{prop:eigenfunctionsperturbative}, the first term is of order 
$$\E\left[(1+\hat \beta)^N\|\hat f-f\|_{\W^{-1}_{r}} \right]\leq \p{\E\left[(1+\hat \beta)^{2N}\right]\E\left[\|\hat f-f\|_{\W^{-1}_{r}}^2 \right]}^{\frac 12}$$
for some $r,N\geq 1$. The second equation in \eqref{eq:rate_estimator} shows that the term 
$$
\E[D_q(\Pi_{ \Upsilon,f}, \Pi_{ \Upsilon,\hat f})\ones\{E\}]
$$
is at most of  order $r_n$. 

To bound the second term, we remark that $D_q(\Pi_{ \Upsilon,f}, \Pi_{ \Upsilon,\hat f})$ can be bounded by a term of the form $C(1+\hat \beta)^N$ for some  absolute constant depending on $L$, $\delta$ and $q$ (using the mapping properties given in \eqref{eq:pibounded}). By Cauchy-Schwarz inequality and as the moments of $\|\hat f\|_{\cC^2}$ are bounded, we bound the second term $\E[D_q(\Pi_{ \Upsilon,f}, \Pi_{ \Upsilon,\hat f})\ones\{E^c\}]$ up to a multiplicative constant by
$$\sqrt{\P(E^c)} = \sqrt{\P(\|\hat f-f\|_{\Cf^1}\geq \eta)}.$$ By Markov's inequality, 
\begin{equation}
    \P(\|\hat f-f\|_{\Cf^1}\geq \eta)\leq \eta^{-9} \E[\|\hat f-f\|_{\Cf^1}^9].
\end{equation}
As the $\Cf^1$-norm is controlled by the $\cC^{1+\eps}$-norm for any $0<\eps<1$, the first equation in \eqref{eq:rate_estimator} shows that  the square root of this term is negligible in front of $r_n$. This concludes the proof of \Cref{thm:cv_spectral_eigenfunction}. 
\end{proof}

\begin{remark}\label{rem:p=infty}
Let $\eps>0$ be a small number and assume that $d\geq 3$. 
 \Cref{cor:perturbation_infty} and \Cref{prop:minimax_standard} can  be leveraged to obtain a rate of convergence of order $n^{-\frac{s+1-\eps}{2s+d}}$ between $\Pi_{ \Upsilon,\hat f}$ and  $\Pi_{ f,\Upsilon}$ with respect to the $\L^\infty$-angle $D_\infty$. Assuming for the sake of simplicity that the curve $\Upsilon$ encloses only one eigenvalue $\phi$ with simple multiplicity, the bound on the $\L^\infty$-angle ensures the existence of an eigenfunction $\hat \phi$ of $\Delta_{\hat f,\alpha}$ with $\E[\|\hat \phi-\phi\|_{\L^\infty}]$ controlled. In particular, so is the empirical $\L^2$-norm given by
 \begin{equation}
 \|\hat \phi-\phi\|_n^2 = \frac 1n \sum_{i=1}^n |\hat \phi(X_i)-\phi(X_i)|^2.
 \end{equation}
The empirical $\L^2$-norm is the standard loss function used to compare eigenfunctions in situation where the manifold $M$ is unknown. Thus, obtaining rates of convergence with respect to this norm is necessary to develop fair comparisons to the rates obtained using graph Laplacians.
\end{remark}

\begin{remark}
The reader may remark that \Cref{prop:perturbativechi} may be used to show that estimation of any operator $\chi(-\Delta_f)$ for $\chi$ a $\Cf^3$ function with compact support is also possible, with same rate of estimation as in \Cref{thm:cv_spectral_eigenfunction}.
\end{remark}

\subsection{Asymptotic expansion of eigenvalues}\label{sec:asymptotic_expansion}
In this paragraph we state a result which implies that the eigenvalues $\lambda_{k,f}$s of $-\Delta_f$ define $s$-regular functionals in the sense of \Cref{def:regular_functional}. To take into account multiplicity, we must rather consider the average of the eigenvalues of $-\Delta_f$ within a contour $\Upsilon$ whose boundary is at a controlled distance from the spectrum.

\begin{proposition}\label{prop:eigenvalue_perturbative}
Let   $L,\delta > 0$ and $s\geq 2$. Let $\Upsilon \subset \C$ be a simple loop bounding a domain $\Omega$ satisfying
$
 \sup_{z\in \Upsilon} |\Re z| < L.
$
The map $f\in \cP_s(\Upsilon,\delta,L)\mapsto \mu_{\Upsilon,f}$ (defined in \eqref{eq:def_mu_upsilon})
is an $s$-regular functional with norm depending on $L$, $\delta$ and $s$. Furthermore, the first differential $f\mapsto \mu_{\Upsilon,f}^{(1)}$ is continuous from $\Cf^1$ to the set of linear forms on $\L^2$ endowed with the operator norm.
\end{proposition}

Note that this proposition, combined with \Cref{thm:functional_estimation}, immediatly implies \Cref{thm:cv_eigenvalue}. The rest of this section is dedicated to giving a proof of \Cref{prop:eigenvalue_perturbative}.

To prove \Cref{prop:eigenvalue_perturbative}, we will rely on perturbation theory to obtain a Taylor expansion of the functional $\mu_\Upsilon$. We then must check that the different terms in the Taylor expansions satisfy the conditions \emph{\ref{G1}}, \emph{\ref{G2}} and \emph{\ref{G3}} given in \Cref{def:regular}.

We fix $L ,\delta > 0$, $s\geq 2$ and $\Upsilon$ as in the statement of \Cref{prop:eigenvalue_perturbative}. Let $\eta$ be the constant given by the statement of \Cref{prop:eigenfunctionsperturbative}. 
Next, let $f,h\in \cP_s(\Upsilon,\delta,L)$ be such that $w = h - f$ satisfies $\|w\|_{\Cf^1}<\eta$. 
 In particular, by \Cref{prop:eigenfunctionsperturbative}, the number $N_{h}$ of eigenvalues of $-\Delta_{h}$ enclosed by $\Upsilon$ is  equal to  $N = N_f$. 
Remark that we have 
$$
\Delta_h = \Delta_f +  \alpha A_{\log(h/f)} 
$$
as well as the identity 
\begin{equation}\label{eq:starting_point}
-\Delta_f\Pi_f = \frac{1}{2\pi i} \int_{\Upsilon} zR_f(z)\dd z.
\end{equation}
The latter follows from the formula $-\Delta_fR_f(z) = -1+zR_f(z)$ and the fact that the integral of $1$ along $\Upsilon$ vanishes. The same formula holds for $h$. 
Now for $z \in \Upsilon$, define 
$$
Q(z) = \alpha A_{\log(h/f)} R_f(z),
$$
so that $R_h(z) = (z+\Delta_f + \alpha A_{\log(h/f)} )^{-1}= R_f(z)(1+Q(z))^{-1}$. 
Notice that for $J \in \mathbb N$, we have the identity 
$$
(1+Q(z))^{-1} = \sum_{j=0}^J(-1)^j Q(z)^j + (-1)^{J+1} (1+Q(z))^{-1} Q(z)^{J+1}.
$$
Hence we obtain the development
\begin{equation}\label{eq:devresolve}
R_h(z)  = \sum_{j=0}^J (-1)^j R_f(z)Q(z)^j + (-1)^{J+1} R_h(z)Q(z)^{J+1} .
\end{equation}
Integrating over $\Upsilon$, one obtains the formula
\begin{equation}\label{eq:exact_expansion}
-\Delta_h \Pi_h = \sum_{j=0}^J P_{j} + R_J
\end{equation}
where we set, for any $J \in \N$ and $j\geq 0$,
\begin{equation*}
P_{j} = \frac{(-1)^j}{2 \pi i} \int_\Upsilon zR_f(z)Q(z)^j \dd z \quad \text{and} \quad R_J= \frac{(-1)^{J+1}}{2\pi i} \int_\Upsilon zR_h(z)Q(z)^{J+1}  \dd z.
\end{equation*}%
For any eigenvalue $\lambda$ enclosed in $\Upsilon$, let us write
\begin{equation}\label{eq:laurent_development}
R_{f}(z) = \sum_{k = -1}^\infty R_{\lambda, k} (z - \lambda)^k
\end{equation}
the Laurent development \eqref{eq:laurent} of $R_f(z)$ near $z = \lambda$. Then $R_{\lambda, - 1} = \Pi_{\lambda, f}$ and the residue at $z = \lambda$ of $zR_f(z)Q(z)^j$ is given by a linear combination of terms of the form
\begin{equation}\label{eq:opresidue}
R_{\lambda, k_1} \prod_{\ell = 2}^{j+1} (\alpha A_{\log (h/f)})R_{\lambda, k_\ell}
\end{equation}
where $k_1, \dots, k_{j+1} \in \Z_{\geqslant -1}$ are such that $k_1 + \cdots + k_{j+1} = -1$ or $-2$. The rank of the operator \eqref{eq:opresidue} is not greater than $\operatorname{rank}\Pi_{\lambda,f}$ and we deduce
\begin{equation}\label{eq:rankpj}
\operatorname{rank} P_j \leqslant \sum_{\lambda \in \Omega} C_j\operatorname{rank}\Pi_{\lambda,f} = C_j N,
\end{equation}
for some constant $C_j$ depending only on $j$. Here $\Omega \subset \C$ is the domain enclosed by $\Upsilon$. Note that we also have
$$
\operatorname{rank} R_J \leqslant \operatorname{rank}(\Delta_h\Pi_h) + \sum_{j=0}^J \operatorname{rank}(P_j) \leqslant D_J N
$$
for some $D_J > 0$ depending only on $J$.

We may thus take the trace in \eqref{eq:exact_expansion} to obtain the identity
\begin{equation}\label{eq:devlimmu}
\mu_{\Upsilon,h}=\mu_{\Upsilon,f} + \frac 1N\sum_{j=1}^J \tr(P_j) + \frac 1N\tr(R_J).
\end{equation}
Define the $j$-multilinear form $ G_j[u_1,\dots,u_j]$ as \emph{the symmetrization} of $\tr\p{  \widetilde G_j[u_1,\dots,u_j]}$ where
\begin{equation}\label{eq:def_Gj}
\widetilde G_j[u_1,\dots,u_j] = \frac{(-1)^j}{2 \pi i} \int_\Upsilon z(z + \Delta_f)^{-1}\p{\prod_{k=1}^jB_z[u_k]} \dd z
\end{equation}
and $B_z[u]=A_{u}(z + \Delta_f)^{-1}$. 
Then it holds 
\begin{equation}\label{eq:trpj}
\tr(P_j) = \alpha^j G_j[\log(h/f),\dots,\log(h/f)].
\end{equation}

\begin{lemma}\label{lem:C0_norm_of_Gj}
Let $j\geq 1$. There exists $C_j>0$ depending on $L$ and $\delta$ such that for all $p_1,\dots,p_j\in \left]1,+\infty\right[$ with $\sum_{k=1}^j \frac 1{p_k}=1$, the operator norms of $G_j$ seen as operators
\begin{align*}
&G_j: \L^{p_1}\times \cdots \times \L^{p_j}\to \R\ \text{ and }\ G_j: (\cC^{s})^*\times (\cC^s)^{j-1}\to \R
\end{align*}
are bounded by $C_j$.
\end{lemma}
We will assume the lemma for the moment and explain why it implies \Cref{prop:eigenvalue_perturbative}.
\begin{proof}[Proof of \Cref{prop:eigenvalue_perturbative}]
We will prove that the development \eqref{eq:devlimmu} of $\mu_{\Upsilon, f}$ fulfills the requirements of \Cref{def:regular_functional}. 

First, to obtain an actual Taylor expansion with terms that are multilinear with respect to $h-f$, we perform a Taylor expansion of $\log(h/f)$ around $f$ up to order $J$. Let $w=(h-f)/f$. The Taylor expansion for the $j$th term $\tr(P_j)$ yields a sum of terms of the form $ G_j[w^{a_1},\dots,w^{a_j}]$ for  positive integers $a_1,\dots,a_j$ whose sum is smaller than $J$. The remainder term in this expansion is a linear combination of $ \tr(R_J)$ and of terms of the form $ G_j[u_1,\dots,u_j]$ for some functions $u_1,\dots,u_j$ with 
\[ \prod_{k=1}^j \|u_k\|_{\L^\infty} \lesssim \|w\|_{\L^\infty}^{J+1}\lesssim \|h-f\|_{\L^\infty}^{J+1}.  \] Here, the terms $u_k$ are either equal to $w^{a_j}$ for some integer $j$ or to the remainder term of order $J+1$ in the Taylor expansion of $\log(h/f)$ around $f$.  

By \Cref{lem:C0_norm_of_Gj}, the residual terms of the form $ G_j[u_1,\dots,u_j]$ are therefore at most of order  $\|h-f\|_{\L^\infty}^{J+1}$, and one can show in a similar fashion that $ \tr( R_J)$ is also at most of this order. In total, this shows that the remainder term in the Taylor expansion \eqref{eq:devlimmu} is of the expected order for an $s$-regular functional.

\Cref{lem:C0_norm_of_Gj} also shows that for $j = 1,2$ the symmetric multilinear forms induced by the homogeneous forms
$$
w \mapsto  G_j[w^{a_1},\dots,w^{a_j}],
$$
where $a_1 + \cdots + a_j = j$, are $s$-regular linear functionals. For $j=3$, \Cref{lem:C0_norm_of_Gj} shows that those symmetric forms satisfy the condition \emph{\ref{G3}} except for the requirement involving the Hilbert–Schmidt norm.

Therefore, to conclude the proof, it remains to show two properties. First, we must show that the first differential, given by the map $G_1$, is continuous as a function of the base point $f$ with respect to the $\Cf^1$-norm (and thus for the $\cC^{s'}$-norm for $s'\in ]1,s[$)). Second, we must show that the symmetric multilinear forms induced by $w \mapsto  G_j[w^{a_1},\dots,w^{a_j}]$ with $a_1 + \cdots + a_j = 3$, satisfies the Hilbert-Schmidt norm requirement in \emph{\ref{G3}}. We will show only the first property below~; the proof of the second one is post-poned in \Cref{subsec:endproofeiganvalue}.

Write $f=e^{U}$ and consider a function $h=e^V$ nearby $f$. If $\|f-h\|_{\Cf^1}$ is small enough, then the arguments above 
imply that the first differential of $\mu_\Upsilon$ at $h$ is given  for a smooth function $u$ by
\[
 \mu_{\Upsilon,h}^{(1)}[u] = -\alpha\tr\p{ \frac{1}{2\pi N i} \int_{\Upsilon} z R_h(z) A_u R_h(z) \dd z}.
\]
We have
\begin{equation}\label{eq:differencemu1}
\begin{aligned}
R_h(z)  A_u R_h(z) - R_f(z)  A_u R_f(z)  &=  \alpha R_h(z)  A_u R_h(z)  A_{U-V} R_f(z) \\
&\qquad\qquad  + \alpha R_h(z) A_{U-V} R_f(z)  A_u  R_f(z).
\end{aligned}
\end{equation}
Next, let us write
$$
\mu^{(1)}_{\Upsilon,h}[u] - \mu_{\Upsilon,f}^{(1)}[u] = \frac{-\alpha}{2 \pi N i} \int_\Upsilon (R_h(z)  A_u R_h(z) - R_f(z)  A_u R_f(z)) \dd z.
$$
Then by the residue theorem one obtains by \eqref{eq:differencemu1} that $\mu^{(1)}_{\Upsilon,h}[u] - \mu_{\Upsilon,f}^{(1)}[u]$ is a linear combination of terms of the form
$$
R_1 A_1 R_2 A_2 R_3
$$
where $A_1 = A_u$ and $A_2 = A_{U-V}$ or $A_2 = A_u$ and $A_1 = A_{U-V}$, the $R_j$s are bounded $\W^{t,q} \to \W^{t+2,q}$ for $-s < t < s-1$ and $q > 1$ (see Proposition \ref{prop:resolvbounded}), and there is $j_\star$ such that $R_{j_\star}$s is equal to $\Pi_{\lambda, f}$ or $\Pi_{\mu, h}$ for $\lambda \in \sigma(-\Delta_f)$ or $\mu \in \sigma(-\Delta_h)$ encircled by $\Upsilon$. By cyclicity of the trace and the fact that $R_{j_\star}^2 = R_{j_\star}$, one obtains
$$
\mu^{(1)}_{\Upsilon,h}[u] - \mu_{\Upsilon,f}^{(1)}[u] = \tr(R_{j_\star} W R_{j_\star}) 
$$
where $W$ is an operator given by a linear combination of products involving $A_1, A_2$ and the $R_j$s for $j \neq j_\star$. Let $p > 1$. Using again that $R_{j_\star}$ is a projector one deduces that
$$
|\mu^{(1)}_{\Upsilon,h}[u] - \mu_{\Upsilon,f}^{(1)}[u]|\leqslant \operatorname{rank}(R_{j_\star}) \|R_{j_\star} W R_{j_\star}\|_{\L^2 \to \L^2} \lesssim \operatorname{rank}(R_{j_\star})\|W\|_{\mathscr C^3 \to W^{-1,p}},
$$
where the second bound comes from the fact that $R_{j_\star}$ is bounded $\L^2 \to \mathscr C^3$ and also $\W^{-1,p} \to \L^2$. Lemmas \Cref{lem:operator_norm_Ah}, \Cref{cor:operator_norm_Ah} and  \Cref{prop:resolvbounded} then yield
 \[ \|W\|_{\mathscr C^3 \to \W^{-1,p}} \lesssim \|u\|_{\L^2} \|U-V\|_{\Cf^1}.\] This implies that $\mu_{\Upsilon}^{(1)}$ is a continuous function of $f$ for the $\Cf^1$-norm. 
\end{proof}

We complete this section by giving a proof of \Cref{lem:C0_norm_of_Gj}.

\begin{proof}[Proof of \Cref{lem:C0_norm_of_Gj}]
    We will estimate $G_j(u_1, \dots, u_j)$ for smooth functions $u_1,\dots,u_j$ and the general result will then follow by density. 
As is the case for $P_j$, we have $\operatorname{rank}\widetilde G_j[u_1,\dots,u_j] \leqslant C_j N$ where $N = \operatorname{rank} \Pi_f$. The trace is then expressed as
\[ 
\sum_{i=1}^{N} \dotp{\widetilde G_j[u_1,\dots,u_j]\phi_i,\phi_i}_f
\]
for some orthonormal basis $(\phi_1,\dots,\phi_{N})$ (in $\L^2(f^\alpha)$) of the range of $\Pi_f$, with $\Delta_f \phi_i=-\lambda_i \phi_i$ for $i=1,\dots,N$. By \eqref{eq:def_Gj}, it suffices to show the following claim.

\begin{claim}\label{lem:sym_Sj_map} There exists $C_j>0$ depending on $L$ and $\delta$ such that for any $p_1,\dots,p_j\in \left]1,+\infty\right[$ and $p_{j+1}\in \left]1,+\infty\right]$ with $\sum_{k=1}^{j+1} \frac 1{p_k}=1$ and $\psi_1,\psi_2$, $z\in \Upsilon$,  the symmetrization $S_j$ of the multilinear form
\begin{equation}\label{eq:S_j}
(u_1,\dots,u_j)\mapsto\Bigl\langle{\prod_{k=1}^j B_z[u_k]\psi_1,\psi_2}\Bigr\rangle_f
\end{equation}
defines bounded operators
$$
S_j:\L^{p_1}\times \cdots \times \L^{p_j}\to \R \quad \text{and} \quad S_j:(\cC^s)^*\times (\cC^s)^{j-1}\to \R,
$$
with norms respectively smaller than $C_j \|\psi_1\|_{\L^\infty}\|\psi_2\|_{\W^{1,p_{j+1}}}$ and $C_j\|\psi_1\|_{\cC^{s+1}}\|\psi_2\|_{\cC^{s+1}}$. 
\end{claim}
The final result is then obtained by letting $\psi_1=\phi_i$ and $\psi_2 = R_f(z)\phi_i = \frac{\phi_i}{z-\lambda_i}$, with $p_{j+1}=+\infty$, together with estimates on the regularity of the eigenfunctions given by \eqref{eq:pibounded_bis} which imply that both $\|\phi_i\|_{\cC^{s+1}}$ and $\|\phi_i\|_{\W^{1,\infty}}\lesssim \|\phi_i\|_{\cC^{s+1}}$   are bounded.

We show the claim by induction on $j$. 
First, we remark that by an integration by parts the right-hand side of \eqref{eq:S_j} reads
$$
\begin{aligned}
&-\left\langle{\p{A_{\overline\psi_2} R_f(z)+\overline\psi_2  (1-zR_f(z))}\prod_{k=2}^j B_z[u_k] \psi_1, u_1 }\right\rangle_f \\
&\quad \quad =  -\Bigl\langle{\prod_{k=2}^j B_z[u_k] \psi_1, \psi_2 u_1}\Bigr\rangle_f - \left\langle{\p{A_{\overline\psi_2} R_f(z)-z \overline \psi_2R_f(z))}\prod_{k=2}^j B_z[u_k] \psi_1, u_1 }\right\rangle_f.
\end{aligned}
$$
In particular, for $j=1$, $\prod_{k=2}^j B_z[u_k] \psi_1=\psi_1$, and the conclusion holds by  \Cref{lem:operator_norm_Ah} and \Cref{prop:resolvbounded}. 
For $j\geq 2$, we treat the two terms in the right-hand side of the above equality separately.

\noindent \textbf{The first term.} The symmetrization of the second term is also equal to the symmetrization of 
$$
\begin{aligned}
-\frac 12\Bigl\langle{(u_1 B_z[u_2]+u_2B_z[u_1])\prod_{k=3}^j B_z[u_k] \psi_1, \psi_2}\Bigr\rangle_f
=-\frac 12\Bigl\langle{B_z[u_1u_2]\prod_{k=3}^j B_z[u_k] \psi_1, \psi_2}\Bigr\rangle_f.
\end{aligned}
$$
Applying the inductive hypothesis to $u_1u_2,u_3,\dots,u_j$, and then  either H\"older's inequality for the bound on the first operator norm or the definition of the dual norm $(\cC^s)^*$ for the bound on the second operator norm, we obtain that this term satisfies the conclusion of \Cref{lem:sym_Sj_map}.

\noindent \textbf{The second term.} To show that the symmetric form induced by the second term satisfies the claim, we must show that it is bounded  by an expression proportional to
\begin{equation}\label{eq:first_bound_Sj}
    \|\psi_1\|_{\L^\infty}\|\psi_2\|_{\W^{1,p_{j+1}}}\prod_{k=1}^j \|u_k\|_{\L^{p_k}}
\end{equation}
and that for all $ k_0\in \{1,\dots,j\}$ it is bounded by expressions proportional to
\begin{equation}\label{eq:k0_somewhere}
  \|\psi_1\|_{\cC^{s+1}}\|\psi_2\|_{\cC^{s+1}}\|u_{k_0}\|_{(\cC^s)^*}\prod_{k\neq k_0} \|u_k\|_{\cC^s}.
\end{equation}
We use an integration by parts and the fact that the resolvent is self-adjoint, to write the second term as
\begin{align*}
\kappa= \left\langle{\p{A_{\overline\psi_2}R_f(z) -z \overline\psi_2 R_f(z)}\prod_{k=2}^j B_z[u_k] \psi_1, u_1 }\right\rangle_f = \left\langle{ \prod_{k=2}^j B_z[u_k] \psi_1,\widetilde \psi_2}\right\rangle_f,
\end{align*}
where $\widetilde \psi_2$ is defined by
\[ \widetilde \psi_2 = - R_f(\overline z)(\overline z \psi_2 u_1 + \nabla u_1\cdot \nabla \psi_2 + u_1\Delta_f\psi_2).
\]
By induction, $\kappa$ is bounded for all $k_0\in \{2,\dots,j\}$ up to a multiplicative constant by
\[ \|\psi_1\|_{\L^\infty}\|\widetilde\psi_2 \|_{\cC^{s+1}} \| u_{k_0}\|_{(\cC^s)^*} \prod_{2\leq k\neq k_0} \|u_k\|_{\cC^{s}}.\]
But we have $ \|\widetilde \psi_2 \|_{\cC^{s+1}}\lesssim \|u_1\|_{\cC^{s}}\|\psi_2\|_{\cC^{s+1}}$ by \Cref{lem:operator_norm_Ah} and \Cref{prop:resolvbounded}. This proves that \eqref{eq:k0_somewhere} is satisfied for $k_0\in \{2,\dots,j\}$. Let $p<\infty$ be such that $\frac 1p = \frac 1{p_{j+1}}+\frac 1{p_1}$.  Likewise, we use the induction hypothesis to obtain a bound on $\kappa$ of order
\[
   \|\psi_1\|_{\L^\infty}\|\widetilde\psi_2\|_{\W^{1,p}}\prod_{k=2}^j \|u_k\|_{\L^{p_k}}.
\]
Furthermore, $ \|\widetilde \psi_2 \|_{\W^{1,p}}\lesssim \|\overline  z \psi_2 u_1 + \nabla u_1\cdot \nabla \psi_2 + u_1\Delta_f\psi_2\|_{\W^{-1,p}}$  by \Cref{prop:resolvbounded}. As $f$ is a lower bounded function in $\Cf^2$, the $\W^{-1,p}$-norm is equivalent to the weighted $\W^{-1,p}(f^\alpha)$-norm. We let $v\in \W^{1,p^*}(f^\alpha)$ where $\frac 1p + \frac 1{p^*}=1$ and bound the $\W^{-1,p}(f^\alpha)$-norm by duality, as follows:
\begin{align*}
   \dotp{\overline  z \psi_2 u_1 + \nabla u_1&\cdot \nabla \psi_2 + u_1\Delta_f\psi_2,v}_f \\
   &\leq |z| \|\psi_2\|_{\W^{1,p_{j+1}}}\|u_1\|_{\L^{p_1}}\|v\|_{\W^{1,p^*}} -\int u_1 \div(f^\alpha v \nabla \psi_2)+ \int u_1 \Delta_f \psi_2 v f^\alpha \\
   &= |z| \|\psi_2\|_{\W^{1,p_{j+1}}}\|u_1\|_{\L^{p_1}}\|v\|_{\W^{1,p^*}} - \int u_1 \nabla v \cdot \nabla \psi_2 f^\alpha \\
   &\lesssim  \|\psi_2\|_{\W^{1,p_{j+1}}}\|u_1\|_{\L^{p_1}}\|v\|_{\W^{1,p^*}}.
\end{align*}
This shows that $\|\widetilde \psi_2\|_{\W^{1,p}}\lesssim \|\psi_2\|_{\W^{1,p_{j+1}}}\|u_1\|_{\L^{p_1}}$, thus proving the bound \eqref{eq:first_bound_Sj}.

It remains to prove \eqref{eq:k0_somewhere} for $k_0=1$. 
From \Cref{lem:operator_norm_Ah} and \Cref{prop:resolvbounded}, we get
\[
\kappa\lesssim\|\psi_2\|_{\cC^{s+1}}  \|u_1\|_{(\cC^{s})^*} \left\|\prod_{k=2}^j B_z[u_k] \psi_1\right\|_{\cC^{s-1}}.\]
But then, a repeated application of \Cref{lem:operator_norm_Ah} and \Cref{prop:resolvbounded} yields
\[
\left\|\prod_{k=2}^j B_z[u_k] \psi_1\right\|_{\cC^{s-1}} \lesssim \|\psi_1\|_{\cC^{s+1}}\prod_{k=2}^j \|u_k\|_{\cC^s},
\]
thus proving the bound \eqref{eq:k0_somewhere} for $k_0=1$. 

The induction, and the proof of \Cref{lem:sym_Sj_map}, is complete.
\end{proof}

\section{Minimax lower bounds}\label{sec:lower_bounds}

In this section, we provide minimax lower bounds for eigenfuction and eigenvalue estimations, that are the contents of Theorems \ref{thm:lowerbound_eigenfunction} and \ref{thm:lowerbound_eigenvalue}.

\subsection{Minimax lower bound for eigenfunction estimation}

The minimax lower bound for eigenfunction estimation relies on a reverse stability bound on the angle between eigenspaces, which shows that the $\L^2$-angle between eigenspaces is larger that the $\H^{-1}$-norms between the associated densities, modulo error terms involving weaker norms.

\begin{proposition}\label{prop:angle_lowerbound}
  Let  $L > 0$, $r \in \left]\frac{2d}{d+2}, 2\right[$ and $\delta > 0$. Let $\Upsilon \subset \C$ be a simple loop and a smooth positive function $f$ such that
\begin{equation}\label{eq:satisfy}
 \sup_{z\in \Upsilon} |\Re z| \leq L \quad \text{and} \quad \dist(\Upsilon, \sigma(-\Delta_f))\geq \delta.
\end{equation}
Assume also that there is a positive eigenvalue $\lambda \in \sigma(-\Delta_f)$ which is enclosed in $\Upsilon$, let $\varphi$ be an associated eigenfunction and take $x_\star \in M$ such that $\nabla \varphi(x_\star) \neq 0.$

Then there are $\rho, c, C,\eta>0$ and a chart $\Psi : B(x_\star, \rho) \to \R^d$ with the following properties. For any $h \in \Cf^2$ with  $\inf h \geq \delta$, $\|h\|_{\Cf^2}\leq L$ and $\|f-h\|_{\Cf^1} < \eta$, the spectral projector $\Pi_{\Upsilon, h}$ is well defined. If moreover 
$$
\supp \log(h/f) \subset B(x_\star, \rho) \qquad \text{and} \qquad \left|\mathcal F\left[\log(h/f) \circ \Psi^{-1}\right]\right| \quad \text{is radial},
$$
where $\mathcal F$ is the Fourier transform on $\R^d$, then we have the estimate
$$
D_2(\Pi_{ \Upsilon,f}, \Pi_{ \Upsilon,h}) \geq c \|f - h\|_{\H^{-1}} - C \|f-h\|_{\W^{-1,r}}.
$$
\end{proposition}
\begin{remark}\label{remark:fourierradial}
We will use the above result for families of functions of the form $(h_{\varepsilon,y})$, with
$$
h_{\varepsilon,y}(x) = f(x)\left(1 + \varepsilon^{s} \chi\left(\frac{|\Psi(x)-\Psi(y)|}{\varepsilon}\right)\right), \quad \varepsilon > 0, \quad x,y \in B(x_\star, \rho),
$$
where $\chi \in C^\infty_c(\R_+)$ is a a radial function such that $\chi \equiv 1$ near $0$ and $\Psi : B(x_\star, \rho) \to \R^d$ is a chart centered at $x_\star$. Then for every $y$ we have
$$
|\mathcal F\Psi_*\log(h_{\varepsilon, y}/f)| = |\mathcal F\Psi_*\log(h_{\varepsilon, x_\star}/f)|,
$$
where $\Psi_*u = u \circ \Psi^{-1}$. However $|\mathcal F\Psi_*\log(h_{\varepsilon, x_\star}/f)|$ is radial, since
$$
\Psi_*\log(h_{\varepsilon, x_\star}/f)(\bar x) = \log\left(1 + \varepsilon^s \chi\left(\frac{\bar x}{\varepsilon}\right)\right), \quad \bar x \in \R^d.
$$
\end{remark}
Before proving Proposition \ref{prop:angle_lowerbound}, we need a preliminary result.
\begin{lemma}\label{lem:Za}
Let $Z$ be a $\Cf^2$ vector field on $\R^d$ such that $Z(0) \neq 0$. Then there are $\rho, C > 0$ such that for any function $a \in \Cf^2(\R^d)$ supported in $B(0, \rho)$ and such that $|\mathcal F a|$ is radially symmetric we have the estimate
$$
\|a\|_{\H^{-1}(\R^d)} \leqslant C\|Za\|_{\H^{-2}(\R^d)} + C\|a\|_{\H^{-2}(\R^d)}.
$$
\end{lemma}
\begin{proof}
Up to taking $\rho$ small enough and using a change of coordinates of class $\Cf^2$ (which preserves $\H^2_{\mathrm{loc}}(\R^d)$ hence $\H^{-2}_{\mathrm{loc}}(\R^d)$), we can assume that $Z = \partial_{1}$ in $B(0, \rho),$ so that if $\supp a \subset B(0,\rho)$ one has
\begin{equation}\label{eq:Za}
\|Za\|_{\H^{-2}(\R^d)} = \|\partial_1 a\|_{\H^{-2}(\R^d)}.
\end{equation}
Recall that for $s \in \R$ the $\H^s$ norm on $\R^d$ is given by
$$
\|u\|_{\H^s(\R^d)}^2 = \int |\mathcal Fu(\xi)|^2 \langle \xi \rangle^{2s} \dd \xi, \quad u \in \Cf^\infty_c(\R^d),
$$
where $\langle \xi \rangle = \sqrt{1 + |\xi|^2}$. Then using $\mathcal F \partial_1a(\xi) = i \xi_1 \mathcal Fa(\xi)$ one gets
$$
d \|\partial_1 a\|_{\H^{-2}(\R^d)}^2 = d\int |\mathcal Fa(\xi)|^2 |\xi_1|^2 \langle \xi \rangle^{-4} \dd \xi.
$$
Then using the fact that $|\mathcal F a|$ is radially symmetric one obtains
$$
\begin{aligned}
d \|\partial_1 a\|_{\H^{-2}(\R^d)}^2 &= \sum_{j=1}^d \int |\mathcal Fa(\xi)|^2 |\xi_j|^2 \langle \xi \rangle^{-4} \dd \xi  \\
&= \int |\mathcal F a(\xi)|^2 \langle \xi \rangle^{-2} \dd \xi - \int |\mathcal F a(\xi)|^2 \langle \xi \rangle^{-4} \dd \xi\\
&= \|a\|_{\H^{-1}(\R^d)}^2 - \|a\|_{\H^{-2}(\R^d)}^2.
\end{aligned}
$$
where we used that $|\xi|^2 - \langle \xi \rangle^2 = - 1$ in the last equality. Thus we obtained the desired estimate, with a constant $C$ which depends on the derivatives of the flow-box chart we used to assume $Z = \partial_1$.
\end{proof}
\begin{proof}[Proof of Proposition \ref{prop:angle_lowerbound}]
We take the notation $f, h, \Pi_\tau, \Delta_\tau$ as in the proof of Proposition \ref{prop:eigenvalue_perturbative}. We have, if $V = \log(h/f)$, 
$$
\Pi_{\Upsilon, h} - \Pi_{\Upsilon, f} =- \frac{\alpha}{2\pi i} \int_\Upsilon R_h(z) A_{V} R_f(z) \dd z.
$$
Let $\varphi$ be a normalized eigenvector of $-\Delta_f$ for an eigenvalue $ \lambda>0$ which is contained inside $\Upsilon$. We want to bound from below $
\|(\Pi_{\Upsilon,h} - 1)\Pi_{\Upsilon, f}\varphi\|_{\L^2}.
$
To that aim, note that $\Pi_{\Upsilon, f} \varphi = \varphi$ which gives
$$
(\Pi_{\Upsilon, h} - 1)\Pi_{\Upsilon, f}\varphi = -\frac{\alpha}{2\pi i} \int_\Upsilon R_h(z)(z - \lambda)^{-1} A_{V}\varphi\,\dd z = S_{\Upsilon, h}(\lambda)A_V\varphi
$$
where the operator $S_{\Upsilon, h}(\lambda)$ is defined by
$$
S_{\Upsilon, h}(\lambda) = -\frac{\alpha}{2\pi i} \int_\Upsilon R_h(z)(z - \lambda)^{-1} \dd z.
$$
In fact, in a first step, we will bound from below $\|(1 - \Pi_{\Upsilon, f})A_V\varphi\|_{\H^{-2}}$ using \Cref{lem:Za}. Then the fact that $S_{\Upsilon, h}(\lambda)$ is invertible on $\operatorname{ran}(1 - \Pi_{\Upsilon, f})$ will allow us to conclude.

In order to bound $\|(1 - \Pi_{\Upsilon, f})A_V\varphi\|_{\H^{-2}}$, we start by writing
$$
A_{V} \varphi = \nabla \log(h/f) \cdot \nabla \varphi.
$$
Let $x_\star \in M$ such that $\nabla \varphi(x_\star) \neq 0$. Note that $\nabla \varphi$ is a smooth vector field since $f$ is a smooth positive function. Hence we can apply Lemma \ref{lem:Za} and there are $C, \rho >0$ such that the following holds. There is a chart $\Psi : B(x_\star,\rho) \to \R^d$ such that if $h/f \in \Cf^2$ is such that $|\mathcal F\Psi_*\log(h/f)|$ is radial and $\supp(\log(h/f)) \subset B(x_\star,\rho)$, we have
$$
\|\Psi_*\log(h/f)\|_{\H^{-1}(\R^d)} \leqslant {C} \|Z\Psi_*\log(h/f)\|_{\H^{-2}(\R^d)} + C\|\Psi_*\log(h/f)\|_{\H^{-2}(\R^d)}
$$
where $Z = \Psi_*\nabla \varphi \in \Cf^2(\R^d, T\R^d).$ Note that up to reducing $\rho$, we may assume that for any $s\in \R$ there is $B_s > 0$ such that
$$
B_s^{-1}\|u\|_{\H^{s}} \leqslant \|\Psi_*u\|_{\H^{s}(\R^d)} \leqslant B_s \|u\|_{\H^{s}} 
$$
for any $u$ such that $\supp u \subset B(0, \rho)$. Hence one obtains
$$
\|A_{V}\varphi\|_{\H^{-2}} \geqslant C_1^{-1}\|\log(h/f)\|_{\H^{-1}} - C_1\|\log(h/f)\|_{\H^{-2}}
$$
where $C_1$ depends only on $\Psi$ and $\varphi$. Next, note that 
$$
\log(h/f) = (h-f)\int_0^1 \frac{\dd \tau}{f + \tau(h-f)}.
$$
Therefore one sees that 
\begin{equation}\label{eq:logvsdiff}
C_2^{-1} \|h-f\|_{\H^{-1}} \leqslant \|\log(h/f)\|_{\H^{-1}} \leqslant C_2 \|h-f\|_{\H^{-1}} 
\end{equation}
where $C_2 > 0$ depends on $\|h\|_{\Cf^1}, \|f\|_{\Cf^1}$, $\inf f$, $\inf h$, that is $C_2$ depends on $L$, $\delta$ and $f$. A similar inequality holds for the $\H^{-2}$-norm. We therefore obtain
\begin{equation}\label{eq:boundbelowavphi}
\|A_{V}\varphi\|_{\H^{-2}} \geqslant C_3^{-1}\|f-h\|_{\H^{-1}} - C_3\|f-h\|_{\H^{-2}}
\end{equation}
provided that $\supp(f-h) \subset B(x_\star,\rho)$, for some $C_3 > 0$ depending on $\Psi, \varphi, L$ and $\delta$. Next, let $1 < q, r < 2$ and $p > 2$ such that $1/p + 1/r = 1/q$. By \eqref{eq:pibounded}, we have that $\Pi_{\Upsilon, h}$ is bounded $\W^{-2,q} \to \H^{-2}$, hence 
$$
\|\Pi_{\Upsilon, f}A_{V}\varphi\|_{\H^{-2}}  \leqslant C_4 \|A_{V}\varphi\|_{\W^{-2,q}},
$$
for some $C_4$ depending on $q,\delta, L$ and $\Upsilon.$
Moreover Lemma \ref{lem:operator_norm_Ah} and \eqref{eq:logvsdiff} give
$$
\|A_{V}\varphi\|_{\W^{-2,q}}=\|A_{\varphi}V\|_{\W^{-2,q}}\leqslant C_5 \|f-h\|_{\W^{-1,r}}\|\varphi\|_{\W^{3,p}},
 $$
where $C_5 > 0$ depends on $\delta$ and $L$, and the two last inequality yields 
\begin{equation}\label{eq:boundabovepiavphi}
\|\Pi_{\Upsilon, f}A_{V}\varphi\|_{\H^{-2}} \leqslant C_6\|f-h\|_{\W^{-1,r}}
\end{equation}
Putting everything together, one finally obtains
$$
\begin{aligned}
\|(1-\Pi_{\Upsilon, f})A_{V}\varphi\|_{\H^{-2}} &\geqslant \|A_{V}\varphi\|_{\H^{-2}} - \|\Pi_{\Upsilon, f}A_{V}\varphi\|_{\H^{-2}} \\
&\geqslant C_3^{-1}\|f-h\|_{\H^{-1}} - C_3 \|f-h\|_{\H^{-2}} - C_6\|f-h\|_{\W^{-1,r}}
\end{aligned}
$$
where $C_6$ depends on $\varphi$. Finally by a Sobolev embedding (see Appendix \ref{app:besov}), we have
$$
\|u\|_{\H^{-2}} \leqslant C_7 \|u\|_{\W^{-1,r}}, \quad u \in \Cf^2,
$$
provided that
$
\displaystyle\frac{2d}{d + 2} < r < 2.
$
Hence we proved 
\begin{equation}\label{eq:final1-pi}
\|(1-\Pi_{\Upsilon, f})A_{V}\varphi\|_{\H^{-2}} \geqslant C_3^{-1}\|f-h\|_{\H^{-1}} - C_8 \|f-h\|_{\W^{-1,r}}
\end{equation}
where $C_8 = C_3 + C_6C_7$. 

Next, note that $S_{\Upsilon,f}(\lambda) : \L^2 \to \L^{2}$ is continuous. However by definition of $S_{\Upsilon,f}(\lambda)$ one has $\ker S_{\Upsilon,f}(\lambda) \subset \operatorname{ran} \Pi_{\Upsilon, f}$ so the closed graph theorem gives us $c_1 > 0$ such that
\begin{equation}\label{eq:boundbelowS}
\|S_{\Upsilon,f}(\lambda)A_V\varphi\|_{\L^2} \geqslant c_1 \|(1 - \Pi_{\Upsilon, f})A_V\varphi\|_{\L^2}.
\end{equation}
On the other hand one as
$$
S_{\Upsilon, h}(\lambda) - S_{\Upsilon, f}(\lambda) = - \frac{\alpha}{2\pi i} \int_\Upsilon R_h(z)A_VR_f(z) (z - \lambda)^{-1}\dd z.
$$
\Cref{lem:operator_norm_Ah} and \Cref{prop:resolvbounded} yield
\begin{equation}\label{eq:sdiff}
\|S_{\Upsilon, h}(\lambda) - S_{\Upsilon, f}(\lambda)\|_{\L^2 \to \L^2} \leqslant C_9\|V\|_{\mathrm C^1}.
\end{equation}
This inequality combined with \eqref{eq:boundbelowS} gives us
$$
\begin{aligned}
\|S_{\Upsilon, h}(\lambda)A_V\varphi\|_{\L^2} &\geqslant \|S_{\Upsilon, f}(\lambda)A_V\varphi\|_{\L^2} - \|(S_{\Upsilon, h}(\lambda) - S_{\Upsilon, f}(\lambda))A_V\varphi\|_{\L^2} \\
&\geqslant c_1\|(1 - \Pi_{\Upsilon, f})A_V\varphi\|_{\L^2} - C_9\|V\|_{\mathrm C^1}\|A_V\varphi\|_{\L^2} \\
&\geqslant (c_1 - C_9\|V\|_{\mathrm C^1})\|(1 - \Pi_{\Upsilon, f})A_V\varphi\|_{\H^{-2}} - C_9\|V\|_{\mathrm C^1}\|\Pi_{\Upsilon, f}A_V\varphi\|_{\L^2},
\end{aligned}
$$
where we used $\|u\|_{\H^{-2}}\leqslant\|u\|_{\L^2}$ in the last inequality. Suppose that $\eta$ is small enough so that $\|V\|_{\mathrm C^1} \leqslant c_1 / 2C_9$, which is equivalent to $c_1 - C_9\|V\|_{\mathrm C^1} \geqslant c_1/2$. Then one gets 
\begin{equation}\label{eq:almostfin}
\|S_{\Upsilon, h}(\lambda)A_V\varphi\|_{\L^2} \geq \frac{c_1}{2}\|(1 - \Pi_{\Upsilon, f})A_V\varphi\|_{\H^{-2}} - \frac{c_1}{2}\|\Pi_{\Upsilon, f}A_V\varphi\|_{\L^2}.
\end{equation}
Since $\Pi_{\Upsilon, f} = \Pi_{\Upsilon, f}^2$ is bounded $\H^{-2} \to \L^2$, Eq. \eqref{eq:boundabovepiavphi} gives us $C_{10} > 0$ such that
$$
\|\Pi_{\Upsilon, f}A_V\varphi\|_{\L^2} \leqslant C_{10}\|f-h\|_{\W^{-1,r}}.
$$
Combining this inequality with \eqref{eq:almostfin} and \eqref{eq:final1-pi}, one finally obtains
\begin{equation}\label{eq:finalboundebelow}
\|S_{\Upsilon, h}(\lambda)A_V\varphi\|_{\L^2} \geq c\|f-h\|_{\H^{-1}} - C\|f-h\|_{\W^{-1,r}}
\end{equation}
where $c = c_1C_3^{-1}/2$ and $C = c_1C_8/2 + c_1C_{10}/2$.

By definition, $D_2(\Pi_{ \Upsilon,f}, \Pi_{ \Upsilon,h})$ is bounded from below by
$$
\begin{aligned}
\|(\Pi_{\Upsilon, h} - 1)\Pi_{\Upsilon, f}\|_{\L^2\to\L^2}   \geqslant \|(\Pi_{\Upsilon, h} - 1)\Pi_{\Upsilon, f}\varphi\|_{\L^{2}} = \|S_{\Upsilon, h}(\lambda)A_V\varphi\|_{\L^2}.
\end{aligned}
$$
Therefore by \eqref{eq:finalboundebelow} we obtain the desired estimate. The proof is complete.
 \end{proof}

We are now ready to prove \Cref{thm:lowerbound_eigenfunction} by leveraging Fano's method.
\begin{proof}[Proof of \Cref{thm:lowerbound_eigenfunction}]
    We assume first that $d\geq 3$. 
Let $f_0=1$ be the uniform density on $M$ (that we take without loss of generality of volume $1$). Let $\Upsilon \subset \C$ be a simple curve such that $\sigma(-\Delta) \cap \Upsilon = \emptyset$ and such that $\Upsilon$ encloses a \textit{positive} eigenvalue $\lambda$ of $-\Delta$. Let $\Pi_{\Upsilon,f_0}$ be the spectral projector of $-\Delta_{f_0}=-\Delta$ associated with $\Upsilon$ and let $\varphi$ be an eigenfunction associated with $\lambda$. We let $\delta$ be small enough and $L$ be large enough so that $\Upsilon$ satisfies \eqref{eq:rezl} and $f_0\in \cP_s(\Upsilon, \delta, L)$. 

The function $\varphi$ is a nonzero smooth function of mean zero. Thus, there is a point $x_\star$ with $\nabla \phi(x_\star)\neq 0$. We are in position to apply \Cref{prop:angle_lowerbound}. Let $r \in \left]2d/(d+2), 2\right[$. Consider the radius $\rho>0$ and the chart $\Psi:B(x_\star,\rho)\to \R^d$ given by the proposition. Let $\eps>0$. We consider a  set of $N$ points $y_1,\dots,y_N\in B(x_\star,\rho)$ such that the points $\Psi(y_i) \in \R^d$,\, $i = 1, \dots, N$,  are $\eps$-separated. As $\Psi$ is Lipschitz continuous, the number $N$ can be taken of order $\eps^{-d}$, we will choose it equal to $a\eps^{-d}$ for some small number $a$. Next, let $\chi\in \mathrm C_c^\infty(\R_+)$ be a function supported on $[0,1/2]$ such that $\chi \equiv 1$ near $0$. For each $i\in [N]$, we define 
$$
\chi_{i}(x) = \chi\left(\frac{|\Psi(x)-\Psi(y_i)|}{\eps}\right), \quad x \in M.
$$
Note that we can always make $\rho$ smaller so that the image of $\Psi$ contains $B(\Psi(y_i),\eps)$, so that the function $\chi_i$ is defined everywhere on $M$. 
For any $\tau \in \{0,1\}^N$, we define
\begin{equation}\label{eq:perturbation_f0}
\widetilde f_\tau=1+t \sum_{i=1}^N \tau_i \chi_i
\end{equation}
for some parameter $t>0$, that we fix equal to $\eps^s$.
Then by Remark \ref{remark:fourierradial} we have that
$|\mathcal F\Psi_*\log(f_\tau / f_0)|$ is radial. Furthermore, one has $\|\chi_i\|_{\L^\infty} \leq \|\chi\|_{\L^\infty}$. As the functions $\chi_i$ have disjoint supports, one obtains the inequality $\|f_\tau-1\|_{\L^\infty} \leq \eps^s\|\chi\|_{\L^\infty}$. In particular $f_\tau\geq 1/2$ provided that $\eps$ is small enough, so that if we set
$$
f_\tau = \frac{\widetilde f_\tau}{c_\tau}, \quad \text{where} \quad c_\tau = \int f_\tau(x) \dd x,
$$
then $f_\tau$ is a density. For $\tau, \tau' \in \{0,1\}^N$ we wish to bound $D_2(\Pi_{\Upsilon, f_\tau},\Pi_{\Upsilon, f_{\tau'}})$. To proceed, we first note that $\nabla \log f_\tau = \nabla \log \widetilde f_\tau$, which yields $\Delta_{f_\tau} = \Delta_{\widetilde f_\tau}$. Therefore
\begin{equation}\label{eq:d2pi}
D_2\bigl(\Pi_{\Upsilon, f_\tau},\Pi_{\Upsilon, f_{\tau'}}\bigr) = D_2\bigl(\Pi_{\Upsilon, \widetilde f_\tau},\Pi_{\Upsilon, \widetilde f_{\tau'}}\bigr).
\end{equation}
It will be more convenient to work with $\widetilde f_\tau$ and $\widetilde f_{\tau'}$ instead of $f_\tau$ and $f_{\tau'}$ since 
\begin{equation}\label{eq:supptildef}
\operatorname{supp}\log\bigl(\widetilde f_\tau / \widetilde f_{\tau'}\bigr)\subset B(x_\star, \rho),
\end{equation}
which will allow us to apply Proposition \ref{prop:angle_lowerbound}---the latter support property is not verified for $\log(f_\tau / f_\tau')$.

It is clear that for $k\geq 0$ an integer, one has $\|1-f_\tau\|_{\Cf^k}\lesssim \eps^{s-k}$. By interpolation, it follows that $\|f_\tau\|_{\cC^s}\lesssim \|f_\tau\|_{\Cf^s}\lesssim 1$. \Cref{prop:eigenfunctionsperturbative} implies that for $\eps$ small enough, all the densities $f_\tau$ belong to $\cP_s(\Upsilon,\delta,L)$, possibly with a smaller constant $\delta$ independent of $\tau$. Next, let $\tau,\tau'\in \{0,1\}^N$ be such that 
$$
|\tau-\tau'|=\sum_{i=1}^N \ones\{\tau_i\neq \tau'_i\}\geq N/8.
$$
We wish to use \Cref{prop:angle_lowerbound} to lowerbound $D_2(\Pi_{\Upsilon,\widetilde f_\tau},\Pi_{\Upsilon,\widetilde f_{\tau'}})$. By the dual formulation of the $\H^{-1}$-norm, 
\begin{align*}
\bigl\|\widetilde f_\tau- \widetilde f_{\tau'}\bigr\|_{\H^{-1}}^2 \geq t \sum_{i=1}^N  \ones\{\tau_i\neq \tau'_i\}\int \chi_i \frac{g}{\|g\|_{\H^1}}
\end{align*}
for any function  $g\in \H^1$. Letting $g = \sum_{i=1}^N  \ones\{\tau_i\neq \tau'_i\} \chi_i$, we have $\|g\|_{\H^1}^2 \lesssim N\eps^{d-2}$.  Thus,
\begin{equation}\label{eq:lowerboundh-1}
\bigl\|1-\widetilde f_\tau\bigr\|_{\H^{-1}} \gtrsim \frac{t\eps}{\sqrt{N\eps^d}} \sum_{i=1}^N  \ones\{\tau_i\neq \tau'_i\} \int \chi_i^2 \gtrsim  \sqrt{N\eps^d} t\eps = a^{\frac 12}t\eps.
\end{equation}
Furthermore, we may upper bound $\|\widetilde f_\tau-\widetilde f_{\tau'}\|_{\W^{-1,r}}$ through the dual formulation of the $\W^{-1,r}$-norm. Let $r_\star$ be the conjugate exponent of $r$, and let $g$ be a smooth function with $\|g\|_{W^{1,r_\star}}\leq 1$. Then we have
\begin{equation}\label{eq:intftau-ftau}
\int \bigl(\widetilde f_\tau-\widetilde f_{\tau'}\bigr)g = t \sum_{i=1}^N  \ones\{\tau_i\neq \tau'_i\}\int \chi_i g.
\end{equation}
Next, let $c_i$ is the integral of $g$ on the support $B_i$ of $\chi_i$, and write
\begin{equation}\label{eq:decompg}
\int \chi g = \int \chi_i (g - c_i) + \int \chi_i c_i.
\end{equation}
Applying a local Poincaré inequality, see \cite{bjorn2018local}, one gets
\begin{align*}
\left|\int \chi_i (g-c_i)\right| \leq \|\chi_i\|_{\L^r} \|g-c_i\|_{\L^{r_\star}} \lesssim  \eps\|\chi_i\|_{\L^r} \|\ones\{B_i\}\nabla g\|_{\L^{r_\star}}.
\end{align*}
A scaling argument shows that $\|\chi_i\|_{\L^r}$ is of order $\eps^{\frac dr}$. Thus, by H\"older's inequality
\begin{align*}
\left|\eps^s \sum_{i=1}^N  \ones\{\tau_i\neq \tau'_i\}\int \chi_i (g-c_i)\right|& \lesssim t\eps^{1+\frac dr}\sum_{i=1}^N \|\ones\{B_i\}\nabla g\|_{\L^{r_\star}} \\
&\lesssim  t\eps^{1+\frac dr}N^{\frac 1r}\left\| \sum_{i=1}^N \ones\{B_i\}\nabla g\right\|_{\L^{r_\star}}\\
&\lesssim (N \eps^d)^{\frac 1r} t\eps = a^{\frac 1r}t\eps. 
\end{align*}
On the other hand, since $\int |\chi_i|$ is of order $\varepsilon^d$, one has
$$
\begin{aligned}
\left|\eps^s \sum_{i=1}^N  \ones\{\tau_i\neq \tau'_i\}\int \chi_i c_i\right| \lesssim t\varepsilon^{d} \sum_{i=1}^N \ones\{\tau_i \neq \tau_i'\} |c_i| \lesssim t\varepsilon^{ d} \|g\|_{\L^{r_\star}} \lesssim t\varepsilon^{d}.
\end{aligned}
$$
Therefore, combining \eqref{eq:intftau-ftau}, \eqref{eq:decompg} and what precedes, one gets
$$
\left|\int \bigl(\widetilde f_\tau-\widetilde f_{\tau'}\bigr)g\right| \lesssim a^{\frac 1r} t\eps + t\eps^{ d}
$$
Therefore there is $c > 0$ such that if $\varepsilon > 0$ is small enough,
\begin{equation}\label{eq:upperboundw-1r}
\bigl\|\widetilde f_\tau - \widetilde f_{\tau'}\bigr\|_{\W^{-1,r}} \lesssim a^{\frac 1r} t\eps.
\end{equation}
By \eqref{eq:supptildef}, one can apply Proposition \ref{prop:angle_lowerbound} and combining \eqref{eq:lowerboundh-1} and \eqref{eq:upperboundw-1r} one gets
\begin{equation}
D_2(\Pi_{ \Upsilon,\widetilde f_\tau}, \Pi_{ \Upsilon,\widetilde f_{\tau'}}) \geq (c a^{\frac 12} - C a^{\frac 1r}) t\eps.
\end{equation}
for every small $\eps$, where $c$ and $C$ are positive constants independent of $\eps$.
Since $r<2$, for $a$ small enough, the number $\widetilde c = ca^{\frac 12} -Ca^{\frac 1r}$ is positive. For such a choice of $a$, one finally obtains
\begin{equation}\label{eq:lower_bound_angle}
D_2(\Pi_{ \Upsilon,f_\tau}, \Pi_{ \Upsilon,f_{\tau'}}) \geq \widetilde c\, t\eps
\end{equation}
for every small $\varepsilon$, by \eqref{eq:d2pi}. 

Recall that $t=\eps^s$. Since the densities $f_\tau$ and $f_{\tau'}$ are upper and lower bounded, the Kullback-Leibler divergence between $f_\tau$ and $f_{\tau'}$ is controlled by the squared $\L^2$-norm between the two densities. As each function $\chi_i$ has a squared norm of order $\eps^{2s+d}$, it follows that the Kullback-Leibler divergence between $f_\tau$ and $f_{\tau'}$ is of order $|\tau - \tau'| \eps^{2s+d}$. 

We are now in position to apply Fano's method, see \cite[Theorem 2.7]{tsybakov_introduction_2009} by relying on the Varshamov-Gilbert lower bound to construct a $(N/8)$-packing of $\{0,1\}^N$ of size at least $2^{N/8}$, see \cite[Lemma 2.9]{tsybakov_introduction_2009}. Such an application proves that the minimax rate is at least of order $\eps^{s+1}$, where $\eps$ has to be selected so that $n\,\eps^{2s+d}\leq c$ for some constant $c>0$. Thus, for $\eps$ of order $n^{-\frac{1}{2s+d}}$, we find that the minimax rate is at least of order $n^{-\frac{s+1}{2s+d}}$. This concludes the proof for $d\geq 3$.

The lower bound of order $n^{-\frac 12}$ for $d\leq 2$ is  simpler to obtain, and is based on Le Cam's lemma with two points applied to a perturbation of the uniform density $f_0$ of the form \eqref{eq:perturbation_f0} with $t$ of order $1/\sqrt{n}$, and $N$ and $\eps$  of constant order chosen so that \eqref{eq:lower_bound_angle} holds.
\end{proof}

\subsection{Minimax lower bound for eigenvalue estimation}

The minimax lower bound for eigenvalue estimation relies on the following general lower bound for functional estimation, see \cite{birge1995estimation}.

\begin{lemma}[Birgé-Massart lower bound]\label{lem:birge-massart}
Let $(\cX,\cF,\mu)$ be a measured space. For $g\in \L^1(\mu)$ a density, let $P_g$ be the associated probability measure. 
Let $\Theta\subset \L^1(\mu)$ be a set containing the constant function $1$ and let $\mathrm{T}:\Theta\to \R$. 
Let $a>0$ and let $m,n\geqslant 1$ be integers. Assume that there exist 
disjoint measurable sets $A_1,\dots,A_m$ in $\cX$ 
 and functions $g_1,\dots,g_m$ of $\L^\infty$ norm smaller than $1$, with each $g_i$ supported on $A_i$,  $\mu(g_i)=0$ and $\mu(g_i^2)\leqslant a$ (for $1\leqslant i\leqslant m$). 
For $\tau \in \{-1,+1\}^m$, define $g_\tau := 1+\sum_{i=1}^m \tau_i g_i\in \Theta$. Assume that $mn^2 a^2\leqslant 3/5$ and that $T(g_\tau)-T(g)\geqslant 2\beta>0$. Then, for any estimator $\hat {\mathrm{T}}_n : \mathcal X^n \to \R$ we have
$$
\sup_{g\in \Theta} \E_{P_g^{\otimes n}} [|\hat {\mathrm{T}}_n -T(g)|] \geqslant \frac{\beta}{10}.
$$
\end{lemma}

We will apply this lower bound to the eigenvalue functional on the flat torus. Computations are simplified when eigenvalues have multiplicity one, so that we consider the rescaled flat torus $M$ with side lengths $(\kappa_1,\dots,\kappa_d)\in \R_+^d$, where the family $(\kappa_1^{-2},\dots,\kappa_d^{-2})$ is $\Q$-free. Up to rescaling, we may assume that the volume $\kappa_1\cdots \kappa_d$ of $M$ is $1$. Let $f$ be the uniform density on the flat torus. For any multiindex $\kbf=(k_1,\dots,k_d)\in \Z^d$, we let 
$$
\omega_\kbf = 2\pi \left(k_1 / \kappa_1, \dots, k_d / \kappa_d\right) \in \R^d.
 $$
For any $\kbf \in \Z^d$ we set 
$$
\lambda_\kbf = |\omega_\kbf|^2 \qquad \text{and} \qquad \varphi_\kbf(x) = \exp i \omega_\kbf \cdot x \quad \text{for} \quad x \in M.
$$
Then $(\varphi_\kbf)_{\kbf \in \Z^d}$ is an ortho-normal basis of $\L^2$ and we have
$$
-\Delta_f \varphi_\kbf = \lambda_\kbf \varphi_\kbf, \quad \kbf \in \Z^d,
$$
so that $\sigma(-\Delta_f) = \{\lambda_\kbf~:~\kbf \in \Z^d\}$. As the family $(\kappa_1^{-2},\dots,\kappa_d^{-2})$ is $\Q$-free, all these eigenvalues are pairwise distinct. 

Let $N,I\geq 1$ be  integers.  
 Define a grid $(x_j)_j$ indexed by $\jbf=(j_1,\dots,j_d)\in [N]^d$, with $x_\jbf = N^{-1} (j_1\kappa_1,\dots, j_d\kappa_d)$. Let $\eps  \in \left]0,1\right]$ and $t= \eps N$. Let $\rho > 0$ small so that the ball $B(0,\rho) \subset \R^d$ is embedded in $M$ through the canonical projection $\pi : \R^d \to M$ (for example take any $\rho < \min_j \kappa_j / 2$). Let $\chi : \R^d \to \R$ be a smooth function supported on $B(0,\rho)$, and set
 $$
 \chi_\jbf(x) = \eps^s\, \widetilde \chi\left(\frac{x-x_\jbf}{\eps}\right), \quad x \in M, \quad \jbf \in [N]^d,
 $$
 where $\widetilde \chi : M \to \R$ is the smooth function supported in $\pi(B(0, \rho))$ induced by $\chi$, that is
 $$
 \widetilde \chi(x) = \sum_{\kbf \in \Z^d}\chi(\bar x + (2\pi)^{-1} \omega_{\kbf}), \quad x = \pi(\bar x), \quad \bar x \in \R^d.
 $$
We will choose $\chi$ such that
 \begin{equation}\label{eq:def_g}
\|\chi\|_{\L^\infty}\leq 1/2 \qquad \text{and} \qquad \partial_\xi^\mbf \mathcal F \chi(0) = 0 \quad \text{for all} \quad |\mbf|\leqslant I,
 \end{equation}
where we recall that
 $$
 |\mbf| = m_1 + \cdots + m_d \quad \text{and} \quad \partial_\xi^\mbf = \frac{\partial^{m_1}}{\partial \xi_1} \cdots \frac{\partial^{m_d}}{\partial \xi_d} \quad \text{for} \quad \mbf = (m_1, \dots, m_d) \in \N^d.
 $$
 In particular, one has 
$
\int_M \widetilde \chi(x) \dd x = \int_{\R^d} \chi(x) \dd x = \mathcal F \chi (0) = 0
$
so that 
\begin{equation}
\int_M \chi_\jbf(x) \dd x = 0, \quad \jbf \in [N]^d.
\end{equation}
Next, for any $\tau = (\tau_1, \dots, \tau_d) \in \{-1,1\}^d$, we set
\begin{equation}
 f_\tau = 1 + w_\tau \quad \text{where} \quad w_\tau= \sum_{\jbf\in [N]^d} \tau_\jbf \chi_\jbf.
\end{equation}
Remark that as long as $t< \min_{i=1,\dots,d}\kappa_i/2$, the supports of the $\chi_j$'s are pairwise distinct. This implies that $\|w_\tau\|_{\L^\infty}\leq  \eps^{s}\|\chi\|_{\L^\infty}$. One can also check that $\|f_\tau\|_{\cC^s}\lesssim \|f_\tau\|_{\Cf^s}$ is bounded by a constant depending only on $\chi$. In particular, for $\eps$ small enough, all the functions $f_\tau$ are bounded from below by a constant depending only on $\chi$. 

Next, fix $\kbf_\star \in \Z^d \setminus \{0\}$ and put $\lambda_\star = \lambda_{\kbf_\star} > 0$. Let $2\delta > 0$ be the spectral gap of $-\Delta_f$ at $\lambda_\star$, and let $\Upsilon$ be the circle of radius $\delta$ centered at $\lambda_\star$, oriented in the counterclockwise direction. Then, by \Cref{prop:eigenfunctionsperturbative}, provided that $\varepsilon$ is small enough, we have that $\Upsilon$ encloses exactly one eigenvalue $\lambda_\tau$ of $-\Delta_{f_\tau}$ for each $\tau.$ As $\|w_\tau\|_{\L^\infty}=O(\eps^{s})$,   we have according to  \Cref{prop:eigenvalue_perturbative} that
\begin{equation}\label{eq:lambda_tau_DL}
    \lambda_\tau-   \lambda_\star = \mu_{\Upsilon, f}^{(1)}[w_\tau] + \mu_{\Upsilon, f}^{(2)}[w_\tau,w_\tau] + O(\eps^{3s}),
\end{equation}
where $\mu_{\Upsilon, f}^{(1)}$ and $\mu_{\Upsilon, f}^{(2)}$ are forms that can be made explicit, as indicated in the proof of  \Cref{prop:eigenvalue_perturbative}. 
Let $\Pi$ be the orthogonal projection on the eigenspace associated with $\varphi_{\kbf_\star}$. According to the proof of \Cref{prop:eigenvalue_perturbative} we have for any smooth function $u$,
$$
\mu_{\Upsilon, f}^{(1)}[u] = -\frac{\alpha}{2\pi i}\tr \int_\Upsilon zR_f(z)A_{u}R_f(z) \dd z.
$$
Using the residue theorem, the cyclicity of the trace, \eqref{eq:laurent} and \eqref{eq:spicommute}, one sees that
\begin{equation}\label{eq:first_diff_is_zero}
\mu_{\Upsilon, f}^{(1)}[u] = - \alpha\tr(\Pi A_u \Pi) = -\alpha \langle A_u \varphi_{\kbf_\star}, \varphi_{\kbf_\star}\rangle.
\end{equation}
Note that $A_u \varphi_{\kbf_\star} = (\nabla u \cdot \omega_\kbf) \varphi_{\kbf_\star}$, hence
$$
-\langle A_u \varphi_{\kbf_\star}, \varphi_{\kbf_\star}\rangle = - \int_M \nabla u \cdot \omega_{\kbf_\star} = 0,
$$
where the last equality comes from an integration by parts. For the second order term, one obtains
$$
\mu_{\Upsilon, f}^{(2)}[u] = -\mu_{\Upsilon, f}^{(1)}[u^2/2] + \frac{\alpha^2}{2\pi i}\tr \int_\Upsilon zR_f(z)A_uR_f(z)A_{u}R_f(z) \dd z.
$$
By what precedes the term $\mu_{\Upsilon, f}^{(1)}[u^2/2]$ is equal to zero. For the second one, we proceed as follows. First, let $B_\eta = \Delta_f + \eta A_u$ for $\eta \in \R.$ Then for small $\eta$, we have that $\Upsilon$ encloses exactly one eigenvalue of $-B_\eta$, hence $\tr \int_\Upsilon (z + B_\eta)^{-1} \dd z = 2\pi i$. Differentiating twice with respect to $\eta$ one gets
$$
0 = \left.\frac{\dd^2}{\dd \eta^2}\right|_{\eta = 0} \tr \int_\Upsilon (z + B_\eta)^{-1} \dd z = 2\tr \int_\Upsilon R_f(z)A_uR_f(z)A_{u}R_f(z) \dd z.
$$
Therefore one can write
$$
\mu_{\Upsilon, f}^{(2)}[u, u] = \frac{\alpha^2}{2\pi i}\tr \int_\Upsilon (z - \lambda_\star)R_f(z)A_uR_f(z)A_{u}R_f(z) \dd z.
$$
Again using the residue theorem, the cyclicity of the trace, \eqref{eq:laurent} and \eqref{eq:spicommute} one gets
$$
\mu_{\Upsilon, f}^{(2)}[u, u] = \tr \Pi A_u S A_u \Pi,
$$
where $S = S_{\lambda_\star, f}(\lambda_\star)$ with the notations of \eqref{eq:laurent}. The latter operator is diagonal in the basis $(\varphi_\kbf)$: it satisfies $S\varphi_{\kbf_\star} = 0$ and
$$
S\varphi_\kbf = \frac{\varphi_\kbf}{|\omega_{\kbf_\star}|^2 - |\omega_{\kbf}|^2}, \quad \kbf \neq \kbf_\star.
$$
Next, let us write
 $$
 u = \sum_\kbf u_\kbf \varphi_\kbf \quad \text{where} \quad u_\kbf = \int u(x) e^{-i \omega_\kbf \cdot x } \dd x.
 $$ Then $\nabla u = i\sum_\kbf u_\kbf \varphi_\kbf \omega_{\kbf}$ and $\nabla \varphi_{\kbf_\star} = i\varphi_{\kbf_\star} \omega_{\kbf_\star}$ hence
$$
A_u \varphi_{\kbf_\star} = -\sum_\kbf (\omega_{\kbf} \cdot \omega_{\kbf_\star}) u_{\kbf} \varphi_{\kbf} \varphi_{\kbf_\star} = -\sum_\kbf (\omega_{\kbf} \cdot \omega_{\kbf_\star}) u_{\kbf} \varphi_{\kbf + \kbf_\star}.
$$
Therefore one gets
$$
S A_u \varphi_{\kbf_\star} = -\sum_{\kbf \neq 0} \frac{ (\omega_{\kbf} \cdot \omega_{\kbf_\star}) u_{\kbf}}{|\omega_{\kbf_\star}|^2 - |\omega_{\kbf + \kbf_\star}|^2} \varphi_{\kbf + \kbf_\star},
$$
which gives in turn
$$
A_u S A_u \varphi_{\kbf_\star} = -\sum_{\kbf \neq 0} \sum_{\ell} \frac{ (\omega_{\kbf} \cdot \omega_{\kbf_\star})(\omega_\ell \cdot \omega_{\kbf + \kbf_\star})u_\kbf u_\ell}{|\omega_{\kbf_\star}|^2 - |\omega_{\kbf + \kbf_\star}|^2} \varphi_{\ell + \kbf + \kbf_\star}.
$$
Note that if $u$ is real one has $u_{-\kbf} = \overline{u_{\kbf}}$. Hence taking the scalar product with $\varphi_{\kbf_\star}$ yields
$$
\mu_{\Upsilon, f}^{(2)}[u, u] = \tr \Pi A_u S A_u \Pi = \langle \varphi_{\kbf_\star}, A_u S A_u \varphi_{\kbf_\star}\rangle = \sum_{\kbf \neq 0} \frac{(\omega_{\kbf} \cdot \omega_{\kbf_\star})(\omega_{\kbf} \cdot \omega_{\kbf + \kbf_\star})}{ |\omega_{\kbf + \kbf_\star}|^2-|\omega_{\kbf_\star}|^2} |u_\kbf|^2.
$$
Indeed $\langle \varphi_{\ell + \kbf + \kbf_\star}, \varphi_{\kbf_\star}\rangle = 0$ unless $\ell = - \kbf$. Next, note that 
$$
\frac{1}{|\omega_{\kbf + \kbf_\star}|^2-|\omega_{\kbf_\star}|^2} = \frac{1}{|\omega_\kbf|^2 + 2 \omega_\kbf \cdot \omega_{\kbf_\star}} =\frac{1}{|\omega_\kbf|^2} - \frac{2 \omega_\kbf \cdot \omega_{\kbf_\star}}{|\omega_\kbf|^2(|\omega_\kbf|^2 + 2 \omega_\kbf \cdot \omega_{\kbf_\star})}.
$$
Hence we can write 
$$
\frac{(\omega_{\kbf} \cdot \omega_{\kbf_\star})(\omega_{\kbf} \cdot \omega_{\kbf + \kbf_\star})}{ |\omega_{\kbf + \kbf_\star}|^2-|\omega_{\kbf_\star}|^2}  = (\omega_{\kbf} \cdot \omega_{\kbf_\star}) + \gamma(\kbf, \kbf_\star) + \delta(\kbf, \kbf_\star)
$$
where we set
\begin{align*}
\gamma(\kbf, \kbf_\star) 
&=-\frac{(\omega_\kbf \cdot \omega_{\kbf_\star})^2}{|\omega_\kbf|^2}
\ \text{ and } \
\delta(\kbf,\kbf_\star) = -\frac{2 \gamma(\kbf, \kbf_\star) \omega_\kbf \cdot \omega_{\kbf_\star}}{|\omega_\kbf|^2 + 2 \omega_\kbf \cdot \omega_{\kbf_\star}}.
\end{align*}
In particular, as $|u_{-\kbf}|=|u_{\kbf}|$, one gets by symmetry that
$$
\mu_{\Upsilon, f}^{(2)}[u, u] = P[u,u] + Q[u,u]
$$
where we set
$$
P[u,u] = \sum_{\kbf \neq 0} \gamma(\kbf, \kbf_\star) |u_\kbf|^2
\quad \text{and} \quad
Q[u,u] = \sum_{\kbf \neq 0} \delta(\kbf, \kbf_\star) |u_\kbf|^2.
$$
Note that there is $C > 0$ such that 
$
|\delta(\kbf, \kbf_\star)|\leqslant {C}{|\kbf|^{-1}}
$
for each $\kbf \neq 0$.
In particular, one has
\begin{equation}\label{eq:boundquu}
|Q[u,u]|\leqslant C\sum_{\kbf \neq 0} \frac{|u_\kbf|^2}{|\kbf|} \leq C \|u\|_{\H^{-1/2}}^2.
\end{equation}
Since the $\chi_\jbf$'s have distinct support, we get for $v\in \Cf^{1/2}$
\begin{align*}
\dotp{\omega_\tau,v}= \sum_{\jbf} \tau_{\jbf}\dotp{\chi_\jbf,v} =  \sum_{\jbf} \tau_{\jbf}\dotp{\chi_\jbf,v-v(x_{\jbf})} \leq CN^d \eps^{s+d+1/2} =Ct^d \eps^{s+1/2}.
\end{align*}
Moreover, we also get $\|\omega_\tau\|_{\L^2}\leq C't^{d/2} \eps^{s}$ for some other constant $C'$. 
As the $\alpha$-interpolation of the dual of $\cC^{1/2}$ and $\L^2$ is continuously embedded into  $\H^{-1/2}$ for $\alpha$ close enough to $0$, we get 
\begin{equation}\label{eq:omega_H-1/2}
\|\omega_\tau\|_{\H^{-1/2}} \leq C'' t^{d(1/2+\alpha)}\eps^{s+\alpha/2}.
\end{equation}
Next, we wish to give a lower bound for $P[w_\tau,w_\tau]$. Let us write
\begin{equation}\label{eq:Ptautau}
P[w_\tau, w_\tau] = \sum_{\jbf} P[\chi_\jbf, \chi_\jbf] + \sum_{\jbf \neq \jbf'} \tau_\jbf \tau_{\jbf'} P[\chi_\jbf, \chi_{\jbf'}]
\end{equation}
where $P[u,v]$ is the sesquilinear form associated with the quadratic form $P[u,u]$. Note that for each $\kbf$ one has
$$
(\chi_{\jbf})_{\kbf} = \int_M \chi_\jbf(x) e^{-i \omega_\kbf \cdot x} \dd x = \varepsilon^{s + d} e^{i \omega_\kbf \cdot x_\jbf} \mathcal F \chi(\omega_{\varepsilon \kbf}).
$$
Therefore one obtains
$$
P[\chi_\jbf, \chi_\jbf] = \varepsilon^{2s + 2d} \sum_{\kbf \neq 0} \gamma(\kbf, \kbf_\star) |\mathcal F \chi(\omega_{\varepsilon \kbf})|^2 = \varepsilon^{2s+ d } \varepsilon^{-d} \sum_{\kbf \neq 0} \varrho(\varepsilon \kbf)
$$
where $\varrho : \R^d \to \R$ is defined by $\varrho(\xi) = \zeta(\xi)|\mathcal F \chi(\omega_{\xi})|^2$, with   $\zeta(0) = 0$ and 
$$
\zeta(\xi) = -\frac{(\omega_\xi \cdot \omega_{\kbf_\star})^2}{|\omega_\xi|^2}, \quad \xi \neq 0.
$$
The function $\varrho$ is smooth and rapidly decreasing on $\R^d \setminus \{0\}$ (since $\chi \in C^\infty_c(\R^d)$) and continuous on $\R^d$. Hence it is Riemann integrable and we have
$$
\varepsilon^{-d} \sum_{\kbf\neq 0} \varrho(\varepsilon \kbf)  = \varepsilon^{-d} \sum_{\kbf \in \Z^d} \varrho(\varepsilon \kbf) \underset{\varepsilon \to 0}{\longrightarrow} \int_{\R^d} \varrho(\xi) \dd \xi.
$$
Now it is always possible to choose $\chi$ so that the value $c$ of the latter integral is positive. Indeed, there is an even nonnegative function $\nu_0$ such that
\[
 \int \zeta(\xi)\nu_0(\xi) \dd \xi >0
\]
is non zero (take $\nu_0=-\zeta$ for instance). Such a function can be written as $\xi\mapsto |\cF{\chi_0}(\omega_\xi)|^2$ for some real-valued function $\chi_0$ which does not have compact support. By continuity of the integral, we can approximate $\chi_0$ by a compactly supported function $\chi$ without changing too much the value of the integral, ensuring that $\int \rho(\xi)\dd \xi = \int \zeta(\xi) |\cF{\chi}(\omega_\xi)|^2\dd \xi>0$. This is our final choice of $\chi$.

 Hence for $\varepsilon$ small, one obtains
\begin{equation}\label{eq:direct_term}
\sum_\jbf P[\chi_\jbf, \chi_\jbf] \geq c N\eps^{2s+d} \geq ct^d\varepsilon^{2s }
\end{equation}
for some constant $c>0$. 
Next, for $\jbf \neq \jbf'$, let us write
$$
\begin{aligned}
P[\chi_\jbf, \chi_{\jbf'}] &= \sum_{\kbf \neq 0} \gamma(\kbf, \kbf_\star)(\chi_{\jbf})_\kbf \overline{(\chi_{\jbf'})_\kbf} 
&=  \varepsilon^{2s+2d }\sum_{\kbf \neq 0} \varrho(\varepsilon \kbf) e^{i\omega_\kbf \cdot(x_\jbf - x_\jbf')}.
\end{aligned}
$$
By Poisson's summation formula, we have
$$
P[\chi_\jbf, \chi_{\jbf'}]  = \varepsilon^{2s+d} \sum_{\ell} \varrho\left(\frac{x_\jbf - x_{\jbf'} + \widetilde \omega_\ell}{\varepsilon}\right) \quad \text{where} \quad \widetilde \omega_\ell = (\kappa_1 \ell_1, \dots, \kappa_d \ell_d).
$$
Remark that $x_\jbf-x_{\jbf'}\in \prod_{i=1}^d [-\kappa_i,\kappa_i]$. Thus, if $\ell\not \in \{-1,0,1\}^d$, then $|x_\jbf-x_{\jbf'}-\widetilde \omega_\ell|\geq \delta |\widetilde \omega_\ell|$ for some $\delta>0$ depending on $(\kappa_1,\dots,\kappa_d)$. As the function $\varrho$ is rapidly decreasing, there is $K>0$ such that
\[
|\varrho(\xi)|\leq K(1+|\xi|)^{-d-1}, \quad \xi \in \R^d.
\]
Thus, by a standard integral comparison
\[
\sum_{\ell \notin\{-1,0,1\}^d} \varrho\left(\frac{x_\jbf - x_{\jbf'} + \widetilde \omega_\ell}{\varepsilon}\right)\leq \sum_{\ell \notin\{-1,0,1\}^d}  K\p{1+\frac{\delta |\widetilde \omega_\ell|}\eps}^{-d-1} \leq KC_0\varepsilon^{d}
\]
for some constant $C_0$ depending on $(\kappa_1,\dots,\kappa_d)$. We therefore obtain that for a fixed $\jbf$, it holds that
\begin{align*}
\sum_{\jbf':\ \jbf'\neq \jbf} P[\chi_{\jbf},\chi_{\jbf'}] &\leq KC_0\varepsilon^{2s+2d}N^d + \varepsilon^{2s+d}\sum_{\ell\in \{-1,0,1\}^d}\sum_{\jbf':\ \jbf'\neq \jbf}\varrho\left(\frac{x_\jbf - x_{\jbf'} + \widetilde \omega_\ell}{\varepsilon}\right) \\
&\leq KC_0 \eps^{2s+d}t^d +K \varepsilon^{2s+d}\sum_{\ell\in \{-1,0,1\}^d} \sum_{\jbf':\ \jbf'\neq \jbf} \left(1+\frac{|x_\jbf - x_{\jbf'} + \widetilde \omega_\ell|}{\varepsilon}\right)^{-d-1}.
\end{align*}
By another comparison with an integral, we bound
\[
\begin{aligned}
\sum_{\jbf':\ \jbf'\neq \jbf} \left(1+\frac{|x_\jbf - x_{\jbf'} + \widetilde \omega_\ell|}{\varepsilon}\right)^{-d-1}& \leq C_1N^d \int_{x\in [-1,1]^d} \left(1+\frac{|x + \widetilde \omega_\ell|}{\varepsilon}\right)^{-d-1} \\
&\leq C_2 N^d\eps^d
\end{aligned}
\]
for some constants $C_1,C_2$. Thus, as $N\eps=t$,
\[
\sum_{\jbf':\ \jbf'\neq \jbf} P[\chi_{\jbf},\chi_{\jbf'}] \leq K(C_0 +3^d C_1C_2)\eps^{2s+d}t^d.
\]
From this inequality, \eqref{eq:Ptautau} and \eqref{eq:direct_term}, we get that
\[
P[w_\tau,w_\tau] \geq \frac{ct^d}2 \eps^{2s} - K(C_0  + 3^d C_1C_2)\eps^{2s+d}N^dt^{d}= \p{\frac c2- K(C_0  + 3^d C_1C_2)t^d}\eps^{2s}t^d.
\]
Choose $t$ small enough so that $\p{\frac c2- K(C_0  + 3^d C_1C_2)t^d}\geq \frac c4$. For such a choice of $t$, we have by \eqref{eq:boundquu} and \eqref{eq:omega_H-1/2} that
\[
\mu^{(2)}_{\Upsilon,f}[w_\tau,w_\tau] \geq \frac c4 t^d \eps^{2s} - o(\eps^{2s}).
\]
From the Taylor expansion \eqref{eq:lambda_tau_DL}, we are now in position to apply the minimax lower bound given by \Cref{lem:birge-massart} with $\beta$ of order $\eps^{2s}$, $m$ of order $\eps^{-d}$, and $a$ of order $\eps^{2s+d}$. For $\eps$ of order $n^{-\frac 2{4s+d}}$, the minimax lower bound is of order $n^{-\frac{4s}{4s+d}}$.

It remains to show that the minimax rate is also at least of order $n^{-\frac 12}$. The argument  relies on Le Cam's lemma. Let $f\in \cP_s(\Upsilon, \delta,L)$ for some curve $\Upsilon$, such that  the domain bounded by $\Upsilon$ contains  a simple eigenvalue of $-\Delta_f$. 
According to \Cref{prop:eigenvalue_perturbative}, for $h\in \Cf^2$ with $\int h=0$ and $n$ large enough, it holds that
\begin{equation}
    \mu_{\Upsilon,f+n^{-\frac 12}h}-\mu_{\Upsilon,f} = n^{-\frac 12}\mu^{(1)}_{\Upsilon,f}[h]+ n^{-1}O(\|h\|_{\L^\infty}^2).
\end{equation}
As long as $\mu^{(1)}_{\Upsilon,f}\neq 0$, one can therefore choose $h$ with 
\[ |  \mu_{\Upsilon,f+n^{-\frac 12}h}-\mu_{\Upsilon,f}|\gtrsim n^{-\frac 12}.\] 
Furthermore, the Hellinger distance between $f$ and $f+n^{-\frac 12}h$ is of order $n^{-\frac 12}$. By Le Cam's lemma, this implies that the minimax risk is at least of order $n^{-\frac 12}$. The only nontrivial point in this construction is the existence of a density $f$ such that $\mu^{(1)}_{\Upsilon,f}\neq 0$. Indeed, the differential $\mu^{(1)}_{\Upsilon,f_0}$ vanishes when $f_0$ is the uniform density, as was already remarked in \eqref{eq:first_diff_is_zero}. However, when proving that the minimax rate was at least of order $n^{-\frac{4s}{4s+d}}$, we constructed a family of densities $f_\tau$ with $\mu_{\Upsilon,f_\tau}-\mu_{\Upsilon,f_0}$ that is nonzero. This implies that the functional $\mu_\Upsilon$ is not constant on a neighborhood of $f_0$, and therefore that there is a density $f$ close to $f_0$ with a nonzero differential. This concludes the proof.

\section*{Acknowledgements} We are particularly grateful to Clément Berenfeld for insightful discussions on the statistical properties of graph Laplacians that led to the formulation of the main questions addressed in this work. 
Y. Chaubet thanks Gilles Carron, Dorian Le Peutrec and Gabriel Rivière for interesting discussions and helpful comments about this work. V. Divol would also like to thank Eddie Aamari, Gilles Blanchard and Martin Wahl for thoughtful conversations related to this project.

\appendix

\section{Functional spaces}\label{app:besov}

In this section we recall useful definitions and properties of classical functional spaces. In  the sequel, $M$ is a smooth, closed Riemannian manifold.

\subsection{Multi-resolution analysis} \label{subsec:wavelet}
 
 Let $S \in \N$. Following \cite[Proposition 1.53 and Theorem 1.61(ii)]{triebel2006theory} (see also the introduction to wavelets in \cite[Section 4.2]{gine2015mathematical}), we consider a family of Daubechies wavelets of regularity $S$ on $\R^d$, as follows. There are countable sets
$$
\Psi_j = \left\{\psi_{j, \lambda}~:~\lambda \in \Lambda_j\right\} \subset \mathrm C_c^S(\R^n), \quad j \in \N_{\geqslant 0},
$$
of compactly supported $\mathrm C^S$ functions on $\R^d$ such that the following holds.
\begin{enumerate}[label=(\roman*)]
\item If $\Omega \subset \R^d$ is a bounded set, then there is $C_\Omega > 0$ such that
$$
\sharp\left\{\psi \in \Psi_j~:~\supp \psi_j \cap \Omega \neq \emptyset\right\} \leqslant C_\Omega 2^{jd}, \quad j \geqslant 0.
$$
\item There is $C > 0$ such that
$$
\|\psi_{j, \lambda}\|_{\mathrm C^s(\R^d)} \leqslant C 2^{j(S + 1/2)},  \quad \lambda \in \Lambda_j, \quad j \geqslant 0.
$$
\item For $j \geqslant 0$ and $\kbf \in \Z^d$, we have $\psi \in \Lambda_j$ if and only if $\psi(\cdot - \kbf) \in \Lambda_j$. If moreover $j \geqslant 1$ we also have $\psi \in \Lambda_j$ if and only if
$
2^{d/2}\psi(2 \, \cdot) \in \Lambda_{j+1}.
$
\item The family $\left\{\psi_{j, \lambda}~:~\lambda \in \Lambda_j,~j \in \N_{\geqslant 0}\right\}$ is an orthonormal family of $\L^2(\R^d)$.
\end{enumerate}
We will denote $\Psi = \bigcup_j \Psi_j$. 

\subsection{Besov spaces in $\R^d$} Let $t \in \R$ and $1 \leqslant p,q \leqslant \infty$. Given any sequence $a = (a_{j, \lambda})_{\lambda \in \Lambda_j,j\geqslant 0} \in \C^\Psi$, we define 
$$
\|a\|_{\mathrm b^t_{p,q}} = \left(\sum_{j \geqslant 0} 2^{jq(t + \frac d2 - \frac dp)} \left(\sum_{\lambda \in \Lambda_j}|a_{j,\lambda}|^p\right)^{\frac qp}\right)^{\frac 1q},
$$
with usual modifications if $p$ or $q$ is equal to $\infty$.
Let $\mathcal D'_{S}(\R^d)$ be the space of distributions on $\R^d$ of order $S$, that are, distributions $f \in \mathcal D'(\R^d)$ such that for each compact set $\Omega$, there is $C_\Omega > 0$ such that 
$$
|\langle f, \varphi\rangle|\leqslant C_\Omega\|\varphi\|_{\mathrm C^S(\R^d)}, \quad \varphi \in \mathrm C^\infty_c(\R^d), \quad \operatorname{supp} \varphi \subset \Omega.
$$
Any $f \in \mathcal D'_S(\R^d)$ gives rise to a sequence $a_\Psi(f) = (a_\Psi(f)_{j, \lambda})_{\lambda \in \Lambda_j,j\geqslant 0} \in \C^\Psi$ given by
$$
a_\Psi(f)_{j, \lambda} = \langle f, \psi_{j, \lambda}\rangle, \quad \lambda \in \Lambda_j, \quad j \geqslant 0.
$$
For any $|t| < S$, the space $\B^t_{p,q}(\R^d)$ is defined as 
$$
\B^t_{p,q}(\R^d) = \left\{f \in \mathcal D'_S(\R^d)~:~\|a_\Psi(f)\|_{b^t_{p,q}} < \infty \right\}.
$$
It is endowed with the norm
$$
\|f\|_{\B^t_{p,q}(\R^d)} = \|a_\Psi(f)\|_{\mathrm b^t_{p,q}}, \quad f \in \B^t_{p,q}(\R^d).
$$
By \cite[Theorem 1.64]{triebel2006theory} (see also \cite[Section 4.3]{gine2015mathematical}) the set $\B^t_{p,q}(\R^d)$ does not depend on the chosen family $\Psi$ and coincides with the standard Besov spaces on $\R^d$ (see \cite[\S2.3.1]{triebel1983theory}). Remark that by construction, 
\begin{equation}\label{eq:norm_wavelet}
\|\psi_{j, \lambda}\|_{\B^{t}_{pq}(\R^d)} = 2^{j(t+\frac d2-\frac dp)}, \quad \lambda \in \Lambda_j, \quad j \geqslant 0.
\end{equation}

\subsection{Besov spaces on manifolds}
If $M$ is a compact smooth manifold, we define Besov spaces through a system of charts. Namely, let $U_1,\dots,U_N$ be a finite cover of $M$ by open sets diffeomorphic to the unit ball $\mathbb B^d \subset \R^d$, through diffeomorphisms $\kappa_i : \mathbb B^d \to U_i.$ Consider a smooth partition of unity $(\chi_i)_{i=1,\dots,N}$ subordinated to this cover, such that $\chi_i = \zeta_i^2$ for some $\zeta_i \in \mathrm C^\infty_c(U_i)$, so that
\begin{equation}\label{eq:partunitquad}
1 = \sum_{i=1}^N \chi_i = \sum_{i=1}^N \zeta_i^2.
\end{equation}
We also take $\eta_i \in \mathrm C^\infty_c(U_i)$ such that $\eta_i = 1$ on $\operatorname{supp} \chi_i$. Given a distribution $f \in \mathcal D'_S(M)$ of order $S$, we define 
$$
\|f\|_{\B^{t}_{p,q}(M)} = \sum_{i=1}^N \|\kappa_i^*(\eta_i f)\|_{\B^{t}_{p,q}(\R^d)},
$$
where $\kappa_i^* : f \mapsto f \circ \kappa_i$ is the pull-back operator by $\kappa_i.$ Then the space $\B^{t}_{p,q}(M)$ (or simply $\B^t_{p,q}$) is the space of distributions $f \in \mathcal D'_S(M)$ for which $\|f\|_{\B^{t}_{p,q}(M)} < \infty.$

\subsection{Zygmund and Sobolev spaces}
For $s > 0$, we define the Zygmund space $\mathscr C^s$ on $M$ by
$
\mathscr C^s = \B^s_{\infty, \infty}.
$
Recall that $\mathscr C^s$ coincides with the usual H\"older space $\mathrm C^s$ when $s \notin \N$, and that $\mathrm C^s \hookrightarrow \mathscr C^s$ for $s \in \N$. For $s \in \R$ and $p \in [1, \infty]$, the Sobolev space $\W^{s,p}$ is defined by
$
\W^{s,p} = \Lambda^{-s}\L^p,
$
where $\Lambda = (1 - \Delta)^{1/2}$. Finally recall that $\H^s = \W^{s, 2}$.

For convenience of the reader we record below some useful properties regarding relations between the spaces $\mathscr C^s, \W^{s,p}, \H^s$ and $\B^s_{p,q}$ that we will use (see \cite{triebel1992theory}):
\begin{enumerate}[label=(\roman*)]
    \item for $s\in \R$, $\H^s=\B^s_{2,2}$;
    \item if $s>0$, then $\cC^s\hookrightarrow \L^\infty$;
    \item if $s_1\leq s_2$ and $1\leq p,q\leq \infty$, then $\B^{s_2}_{p,q}\hookrightarrow\B^{s_1}_{p,q}$;
    \item if $s\in \R$, $1\leq p_1\leq p_2\leq \infty$ and $1\leq q_2\leq q_1\leq \infty$, then $\B^s_{p_2,q_2}\hookrightarrow \B^s_{p_1,q_1}$;
    \item if $s\in \R$ and $2\leq p<\infty$, then $\B^s_{p,2}\hookrightarrow \W^{s,p}$;
    \item (Sobolev embedding) if $s_1\geq s_2$ and $p_1, p_2\geq 1$ satisfy $\frac 1{p_1} \leq  \frac 1{p_2} + \frac{s_1-s_2}d$,
 then $\W^{s_1,p_1} \hookrightarrow \W^{s_2,p_2}$;
    \item if $t>s>0$ and $p> d/(t-s)$, then $\W^{t,p}\hookrightarrow \cC^s$;
        \item if $p\in (1,+\infty)$ and $s,t$ are such that $s> |t|$, then the multiplication $\cC^s\times \W^{t,p}\to \W^{t,p}$ is continuous;
        \item if $1< p,q<\infty$ and $s\in \R$, then the dual space of $\B^s_{p,q}$ is $\B^{-s}_{p^\star,q^\star}$ where $p^\star,q^\star$ are the conjugate exponents of $p,q$. If $p$ or $q$ is equal to $\infty$, the dual space of $\B^s_{p,q}$ contains $\B^{-s}_{p^\star,q^\star}$;
        \item if $s >0$, then the multiplication $\cC^s\times \cC^s\to \cC^s$ is continuous.
\end{enumerate}

Finally, let us recall the definition of the Hilbert-Schmidt norm associated to an operator ${G}:\cH_1\to \cH_2$ between two separable Hilbert spaces. It is defined as
$$
\|G\|_{\mathrm{HS}} = \p{\sum_{i_1,i_2\geq 1} |\dotp{e^2_{i_2},G[e^1_{i_1}]}_{\cH_2}|^2}^{1/2},
$$
where $(e^j_i)_{i\geq 1}$ is an orthonormal basis of $\cH_j$ for $j=1,2$.  

\section{Approximations of identity}

This section is dedicated to the construction of admissible kernels in the sense of \Cref{def:admissible}.   These kernels are constructed by using charts to lift the wavelet basis on $\R^d$ introduced in \S\ref{subsec:wavelet} to $M$. In the process, we lose  orthonormality of the basis.   Still, the fact that the wavelet basis on the manifold is ``almost orthogonal'' (in the sense that it forms a Bessel frame, see \cite{christensen2003introduction}) is enough to show that the conditions of \Cref{def:admissible} are satisfied. 

For any $D \in [2^J, 2^{J+1}-1] \cap \N$ with $J \geq 1$, we let $\widetilde\pi_D : \L^2(\R^d) \to \L^2(\R^d)$ be the orthogonal projection on 
$$
V_D = \operatorname{ran}\{\psi_{j, \lambda}~:~\lambda \in \Lambda_j,~0 \leqslant j \leqslant J\}.
$$
From the definition of the Besov norms, it is clear that for any $-S < t < s < S$, $1\leq p' \leq p <\infty$ and $1\leq q, q'\leq\infty$, there exists $C$ such that for all $D\geq 1$,
\begin{align}
    \|1-\widetilde \pi_D\|_{\B^{s}_{p,q}\to \B^{t}_{p',q'}} &\leq C D^{-(s-t)},\label{eq:embedding_tilde} \\
    \|\widetilde \pi_D\|_{\B^s_{p,q}\to \B^s_{p,q}}&\leq C\label{eq:embedding_tilde_bis}.
\end{align} 

Next, we lift the projections $\widetilde \pi_D$ on $M$ through the use of charts. For $i=1,\dots,N$, we define the operators $A_i : \mathcal D'(M) \to \mathcal D'(\mathbb B^d)$ and $B_i : \mathcal D'(\mathbb B^d) \to \mathcal D'(M)$ by
$$
A_if  = \kappa_i^*(\zeta_i f) \quad \text{and} \quad B_i g = (\kappa_i^{-1})^*\bigl(\widetilde \zeta_i g\bigr) \quad \text{for}\quad f \in \mathcal D'(M) \quad \text{and} \quad g \in \mathcal D'(\mathbb B^d),
$$
where $\widetilde \zeta_i = \zeta_i \circ \kappa_i^{-1}$.
We also let $A_i^{\star}$ and $B_i^{\star}$ be the adjoints of $A_i$ and $B_i$ over $\L^2$. By a change of variables, the adjoint operator $A_i^\star$ (resp. $B_i^\star$) can be defined exactly as $B_i$ (resp. $A_i$), replacing $\widetilde \zeta_i$ (resp. $\zeta_i$) by another compactly supported function involving the Jacobian of $\kappa_i$. Because (local) Besov spaces are invariant by composition by smooth diffeomorphisms and multiplication by smooth functions, we have that
\begin{equation}\label{eq:boundedtrivial}
A_i, B_i^\star : \B^{t}_{p,q}(M) \to \B^{t}_{p,q}(\R^d) \quad \text{and} \quad B_i, A_i^\star :\B^{t}_{p,q}(\R^d) \to \B^{t}_{p,q}(M)
\end{equation}
are bounded.
Next, define the self-adjoint operator $\pi_D : \L^2(M) \to \L^2(M)$ by setting
\begin{equation}\label{eq:defpid}
 \pi_D  = \frac 12\sum_{i=1}^N (B_i \widetilde \pi_D A_i+ A_i^\star \widetilde \pi_D B_i^{\star}) .
 \end{equation}
 \begin{proposition}\label{prop:pidadmissible}
 The family $(\pi_D)_{D \geqslant 0}$ given by \eqref{eq:defpid} is an admissible family of symmetric operators of class $\mathrm C^{S}$ in the sense of \Cref{def:admissible} for any $0\leq s_\star \leq S - d/2$: for all $\ell,s\in \left[-s_\star,s_\star\right]$ with $\ell< s$ and $1\leq p,p'\leq \infty$ and $1\leq q,q' \leq \infty$, there exists $C_{\star}$ such that for all $D\geq 1$:
    \begin{enumerate}[start=1, label=\emph{(K\arabic*)}]
    \item  if $p\geq p'$, then $\|1-\pi_D\|_{\B^{s}_{p,q}\to \B^{\ell}_{p'q'}}\leq C_{\star} D^{-(s-\ell)}$;
    \item   $\|\pi_D\|_{\B^s_{p,q}\to \B^{s}_{p,q}}\leq C_{\star}$;
    \item  for any integer $J\geq 1$,  $\ell_1,\dots,\ell_J\in [0,s_\star]$, integers $D_1,\dots,D_J\geq 1$, and any multilinear form $G:\H^{\ell_1}\times \cdots\times \H^{\ell_J}\to \R$ with Hilbert-Schmidt norm smaller than $L$, $$\int |G[K_{D_1}(x_1,\cdot), \dots,K_{D_J}(x_J,\cdot)]|^2 \dd x_1\dots \dd x_J \leq C_{\star} L^2 \prod_{j=1}^J D_j^{2\ell_j}. $$
    \item $\|y\mapsto \|K_D(\cdot,y)\|_{\B^{\ell}_{p,p}}\|_{\L^p} \leq C_{\star} \kappa_{p,d,\ell}(D)$, where $\kappa_{p,d,\ell}(D)$ is equal to $D^{\ell+d(1-1/p)}$ if the exponent $\ell+d(1-1/p)$ is positive, $\log(D)$ if it is equal to zero, and $1$ if it is negative. 
\end{enumerate}
 \end{proposition}
 The rest of this subsection is dedicated to the proof of the above proposition.
\begin{proof}[Proof of Proposition \ref{prop:pidadmissible}]
By \eqref{eq:partunitquad}, there holds $\sum_{i=1}^N B_i A_i=\sum_{i=1}^N  A_i^{\star}B_i^{\star}=1$, which yields 
$$
 1-\pi_D = \frac 12\sum_{i=1}^N (B_i (1-\widetilde \pi_D)A_i+ A_i^{\star}(1-\widetilde \pi_D)B_i^{\star}).
$$
Due to \eqref{eq:boundedtrivial}, \eqref{eq:embedding_tilde} and \eqref{eq:embedding_tilde_bis},  $( \pi_D)_{D\geq 1}$ automatically satisfies \emph{\ref{K1}} and \emph{\ref{K2}}. Moreover, $\pi_D$ is symmetric by construction. 

Next, we prove that $\pi_D$ satisfies the property \emph{\ref{K4}}. Let $|\ell| < S.$ The kernel $K_D$ of $\pi_D$ is defined for $x,y\in M$ by
\begin{equation}
  K_D(x,y)= \frac 12\sum_{i=1}^N \sum_{j=0}^J \sum_{\lambda \in \Lambda_j} (A_i^\star \psi_{j, \lambda}(x) B_i\psi_{j, \lambda}(y)+ B_i\psi_{j, \lambda}(x)A_i^{\star}\psi_{j, \lambda}(y)).
\end{equation}
For $i\in [N]$ and $j\geq 0$, define $\Psi^{(i)}_j$ as the set of functions either of the form $\psi^{(i)}=B_i\psi$ or $\psi^{(i)}=A_i^{\star}\psi$ for $\psi \in \Psi_j$, with $\Psi^{(i)}=\bigcup_{j\geq 0}\Psi^{(i)}_j$. Remark that as the set of functions in $\Psi_j$ whose support intersect the unit ball is of order $2^{jd}$, the set of nonzero functions in $\Psi^{(i)}_j$ is also of order $2^{jd}$.

Let us rewrite $K_D(x,y)$ as 
\begin{equation}
      K_D(x,y)= \sum_{i=1}^N \sum_{j=0}^J \sum_{\psi^{(i)} \in \Psi_j^{(i)}} \check{\psi}^{(i)}(x) \psi^{(i)}(y),
\end{equation}
where $\check{\psi}^{(i)}= A_i^{\star}\psi/2$ if $\psi^{(i)} = B_i\psi$ and $\check{\psi}^{(i)}= B_i\psi/2$ if $\psi^{(i)} = A_i^{\star}\psi$. 
We will use the following lemma to bound $\|K_D(x,\cdot)\|_{\B^{\ell}_{p,p}}$.
\begin{lemma}\label{lem:eqnorm}
There is $C > 0$ such that for any sequence $(a^{(i)}_{\psi^{(i)}})$ indexed by $\psi^{(i)} \in \Psi_j^{(i)}$ for $j \geqslant 0$ and $i = 1, \dots, N$, one has
$$
\left\|\sum_{i=1}^N\sum_{j\geq 0}\sum_{\psi^{(i)} \in \Psi_j^{(i)}} a^{(i)}_{\psi^{(i)}} \psi^{(i)}\right\|_{\B^\ell_{p,p}} \leqslant C\|a^{(i)}\|_{\mathrm b^\ell_{p,p}}.
$$
\end{lemma}
\begin{proof}
By definition of the $\B^s_{p,p}$ norm on $M$, the left-hand side of the above inequality is equal to
\begin{equation}\label{eq:bigsum}
\sum_{i_1=1}^N \left\|\sum_{i=1}^N\sum_{j=0}^D\sum_{\psi^{(i)} \in \Psi_j^{(i)}} a^{(i)}_{\psi^{(i)}} \kappa_{i_1}^*\left(\eta_{i_1}\psi^{(i)}\right) \right\|_{\B^\ell_{p,p}(\R^d)}.
\end{equation}
However note that for each $i_1, i \in [N]$ and $\psi^{(i)} \in \Psi_j^{(i)}$, one has
$$
\kappa_{i_1}^*\eta_{i_1} \psi^{(i)} = (\kappa_i^{-1} \circ \kappa_{i_1})^*(\omega_{i_1, i} \psi)
$$
for some $\omega_{i_1, i} \in \mathrm C^\infty_c(\mathbb B^d)$ supported in $\kappa_i^{-1}(U_i) \cap \kappa_{i_1}^{-1}(U_{i_1})$. Now notice that the operator $f \mapsto (\kappa_i^{-1} \circ \kappa_{i_1})^*(\omega_{i_1, i} f)$ is continuous on Besov spaces. Hence one gets that \eqref{eq:bigsum} is controlled by
$$
\left\|\sum_{j=0}^D\sum_{\psi^{(i)} \in \Psi_j^{(i)}} a^{(i)}_{\psi^{(i)}} \psi \right\|_{\B^\ell_{p,p}(\R^d)} \lesssim \left(\sum_{j=0}^D 2^{jp(s + \frac d2-\frac dp)}\sum_{\psi^{(i)} \in \Psi_j^{(i)}} |a^{(i)}_{\psi^{(i)}}|^p \right)^{1/p},
$$
which gives the sought result.
\end{proof}
The above lemma yields
$$
\begin{aligned}
\|K_D(x,\cdot)\|_{\B^{\ell}_{p,p}}^p  &\lesssim \sum_{j=0}^J 2^{jp(\ell + \frac d2-\frac dp)}\sum_{\psi^{(i)} \in \Psi_j^{(i)}} |\check \psi^{(i)}(x)|^p.
\end{aligned}
$$
Integrating this inequality, using that $\Psi_j^{(i)}$ contains a number of nonzero functions of order $2^{jd}$ and that $B_i$ and $A_i^{\star}$ are continuous $\L^p \to \L^p$ together with the estimate $\|\psi\|_{\L^p}^p \lesssim 2^{jd(\frac p2-1)}$, we conclude  that \emph{\ref{K4}} holds. A similar proof holds for $p=\infty$. 

It remains to prove \emph{\ref{K3}}. To do so, we need to establish more precise properties of the families of wavelets on manifolds. 
Let $\cH$ be a separable Hilbert space and let $(\phi_k)_{k\geq 0}$ be a family of vectors in $\cH$. We say that $(\phi_k)_{k\geq 0}$ is a Bessel sequence if  there exists $\alpha>0$ such that for all $f\in \cH$,
\begin{equation}
    \sum_{k\geq 0} |\dotp{f,\phi_k}_{\cH}|^2 \leq \alpha\left\|f\right\|_{\cH}^2.
\end{equation}
\begin{lemma}\label{lem:bessel}
    Let $(\phi_k)_{k\geq 0}$ be a Bessel sequence with constant $\alpha >0$. Then, for all linear forms $G:\cH\to \R$ and sequences of numbers $a = (a_k)_{k\geq 0}$,
    \begin{align*}
        &\sum_{k\geq 0} |G[\phi_k]|^2 \leq \alpha \|G\|_{\mathrm{HS}}^2\quad \text{ and }\quad \left\| \sum_{k\geq 0} a_k\phi_k \right\|_{\cH}^2 \leq \alpha \sum_{k\geq 0}a_k^2.
    \end{align*}
    Furthermore, if $(\phi_k)_{k\geq 0}$ is a family of vectors such that there exists $\alpha>0$ with $$\sup_{j\geq 0}\sum_{k\geq 0}|\dotp{\phi_j,\phi_k}_{\cH}|\leq \alpha,$$ then $(\phi_k)_{k\geq 0}$ is a Bessel sequence with constant $\alpha >0$.
\end{lemma}
\begin{proof}
    The second inequality is stated in \cite[Theorem 3.2.3]{christensen2003introduction}. Let 
    $$
    T:a\in \ell^2(\N_{\geqslant 0}) \mapsto \sum_{k\geq 0}a_k\phi_k\in \cH.
    $$
    Both $\mathrm{T}$ and its adjoint are continuous with norm smaller than $\sqrt{\alpha}$, and the adjoint $T^*$ is given by $T^*\psi = (\dotp{\psi,\phi_k}_{\cH})_{k\geq 0}$ for $\psi\in \cH$, see \cite[Lemma 3.2.1]{christensen2003introduction}. Let $(e_i)_{i\geq 0}$ be an orthonormal basis of $\cH$. Then, $\sum_{i\geq 0} G[e_i]T^*e_i$ is the sequence whose $k$-th entry is equal to $\sum_{i\geq 0} G[e_i]\dotp{e_i,\phi_k}=G[\phi_k]$. Thus,
    \begin{equation*}
        \sum_{k\geq 0} |G[\phi_k]|^2 = \|\sum_{i\geq 0} G[e_i]T^*e_i\|^2_{\ell^2} \leq \alpha \|\sum_{i\geq 0} G[e_i]e_i\|^2_{\cH}= \alpha \|G\|_{\mathrm{HS}}^2.
    \end{equation*}
    The last property is given in \cite[Proposition 3.5.4]{christensen2003introduction}.
\end{proof}
 For $|t| < S$, $j\geq 0$ and $\psi^{(i)}\in \Psi_j^{(i)}$, we let $\psi_t^{(i)} =2^{-jt}\psi^{(i)}$. Define $\Psi_{j,t}^{(i)}$ and $\Psi^{(i)}_t$ in a obvious way.

\begin{lemma}
Assume that $S > d/2$.    For $|t| < S - d / 2$, the family $\bigcup_{i\in [N]} \Psi^{(i)}_t$ is a Bessel sequence in $\H^t$.
\end{lemma}
\begin{proof}
    We check that the sufficient condition given in \Cref{lem:bessel} holds. Fix $i_0\in [N]$ and $j_0 \geqslant 0$. Let $\psi_t^{(i_0)}=2^{-j_0t}\psi_0^{(i_0)}\in \Psi_t^{(i_0)}$, where $\psi_0^{(i_0)}=B_{i_0}\psi_0$ for some $\psi_0\in \Psi_{j_0}$. Then, 
    \begin{align*}
    \sum_{i=1}^N \sum_{\psi^{(i)} \in \Psi_t^{(i)}}& |\dotp{ \psi^{(i)}_t,\psi_t^{(i_0)}}_{\H^t}|    =\sum_{i=1}^N \sum_{j\geq 0} 2^{-(j+j_0)t}\sum_{\psi^{(i)} \in \Psi_j^{(i)}} |\dotp{ \psi^{(i)},\Lambda^{2t}\psi_0^{(i_0)}}| \\
        &=\frac 12\sum_{i=1}^N \sum_{j\geq 0} 2^{-(j+j_0)t}\sum_{\psi \in \Psi_j} (|\dotp{ B_i\psi,\Lambda^{2t}B_{i_0}\psi_0}|+|\dotp{ A_i^{\star}\psi,\Lambda^{2t}B_{i_0}\psi_0}|).\end{align*}
        By duality and using the definition of the Besov norm $\B_{1,1}^{-t+d/2}(\R^d)$, one sees that this quantity is equal to
\begin{align*}
       &\frac 12\sum_{i=1}^N \sum_{j\geq 0} 2^{-(j+j_0)t}\sum_{\psi \in \Psi_j} \left(|\dotp{ \psi,B_i^{\star}\Lambda^{2t}B_{i_0}\psi_0}|+|\dotp{ \psi,A_i\Lambda^{2t}B_{i_0}\psi_0}|\right)\\
        &\qquad \qquad \qquad =  \frac{2^{-j_0t}}{2}\sum_{i=1}^N  \left(\| B_i^\star \Lambda^{2t}B_{i_0}\psi_0\|_{\B^{-t+d/2}_{1,1}(\R^d)}+\| A_i \Lambda^{2t}B_{i_0}\psi_0\|_{\B^{-t+d/2}_{1,1}(\R^d)}\right).    \end{align*}
        By \eqref{eq:boundedtrivial} and the fact that $\Lambda^{2t}:\B^{t+d/2}_{1,1}\to \B^{-t+d/2}_{1,1}$ (see \cite[Section 13.9, Exercise 5]{taylor1996partial}), it holds that both $\| B_i^\star \Lambda^{2t}B_{i_0}\psi_0\|_{\B^{-t+d/2}_{1,1}(\R^d)}$ and $\| A_i \Lambda^{2t}B_{i_0}\psi_0\|_{\B^{-t+d/2}_{1,1}(\R^d)}$ are at most of order $ \|\psi_0\|_{\B^{t+d/2}_{1,1}(\R^d)}=2^{j_0t}$. This shows that 
        \begin{equation*}\sum_{i=1}^N \sum_{\psi^{(i)} \in \Psi_t^{(i)}} |\dotp{ \psi^{(i)}_t,\psi_t^{(i_0)}}_{\H^t}|\lesssim 1.
        \end{equation*}  The same condition holds if $\psi_0^{(i_0)}=A_{i_0}^*\psi_0$ for some $\psi_0\in \Psi_{j_0}$. By \Cref{lem:bessel}, the family is a Bessel sequence.
\end{proof}
One claims that \Cref{lem:bessel} implies that $\pi_D$ satisfies \emph{\ref{K3}}. Indeed, let $s_\star \in [0,S -d/2]$ and $\ell\in [0,s_\star]$. For a linear form $G:\H^\ell\to \R$, letting $D\geq 1$ be of order $2^{J}$, it holds
\begin{align*}
    G[K_D(x,\cdot)]= \sum_{i\in [N]}\sum_{j= 0}^J 2^{j\ell}\sum_{\psi\in \Psi_{j}^{(i)}} G\left[\check \psi^{(i)}\right]\psi^{(i)}(x).
\end{align*}
Thus, by properties of a Bessel frame, letting $\alpha$ be the constant of the frame,
\begin{align*}
    \int |G[K_D(x,\cdot)]|^2 \dd x &= \left\|\sum_{i\in [N]}\sum_{j= 0}^J 2^{j\ell}\sum_{\psi^{(i)}\in \Psi_{j}^{(i)}} G\left[\check \psi^{(i)}\right]\psi^{(i)} \right\|^2_{\L^2}\\
    & \leq \alpha  \sum_{i\in [N]} \sum_{j=0}^J\sum_{\psi^{(i)}\in \Psi_{j}^{(i)}} 2^{2j\ell} \left|G\left[\check \psi^{(i)}\right]\right|^2\\
    &\lesssim \alpha D^{2\ell} \sum_{i\in [N]} \sum_{j=0}^J \sum_{\psi^{(i)}\in \Psi_{j}^{(i)}} \left|G\left[\check \psi^{(i)}\right]\right|^2.
\end{align*}
The last term can be bounded by $\alpha^2 D^{2\ell} \|G\|_{\mathrm{HS}}^2$ by Lemma \ref{lem:bessel}, which proves \emph{\ref{K3}} for linear forms. For multilinear forms, the proof is similar and we leave it to the reader.
\end{proof}

\section{Density estimation with wavelets}\label{app:density}
We prove in this section that the kernels $(K_D)_{D\geq 1}$ can  be used to create minimax density estimators.

\begin{proposition}\label{prop:minimax_standard}
Let $m\geq 1$ be  integer and let $p\geq 2$, $R>0$, $s\in [0, s_\star]$ and $\ell\in \left[-s_\star,s\right[$. There exists $C>0$ such that for all $f\in \cC^s$ with $\|f\|_{\cC^s}\leq R$, if $\mu_n$ is the empirical measure associated  with $n$ i.i.d.~observations with density $f$ then, for $D$ of order $n^{\frac 1{2s+d}}$, the estimator $\hat f=\pi_D\mu_n$ satisfies
\begin{equation}\label{eq:minimax_wass}
\E[\|\hat f - f\|^2_{\W^{\ell}_{p}}]^{\frac 12} \leq C\begin{cases}
n^{-\frac{s-\ell}{2s+d}} &\text{if } \ell>-d/2,\\
n^{-\frac 12}\sqrt{\log n} &\text{if } \ell=-d/2,\\
n^{-\frac 12}&\text{if } \ell<-d/2.
\end{cases}
\end{equation}
Furthermore, if $0\leq \ell\leq s$, then 
\begin{equation}\label{eq:minimax_holder}
\E[\|\hat f - f\|^m_{\cC^{\ell}}]^{\frac 1m} \leq C\p{\frac{\log n}n}^{\frac{s-\ell}{2s+d}}.
\end{equation}
\end{proposition}

\begin{proof}
We show the first inequality. 
Recall that $\B^\ell_{p,2}\hookrightarrow \W^\ell_p$. Thus, it suffices to bound the $\B^\ell_{p,2}$-norm. 
We use the triangle inequality to decompose the risk into a bias term $\|(\pi_D-1) f\|_{\B^{\ell}_{p,2}}$ and a variance term  $ \E[ \|\pi_D(\mu_n-f)\|_{\B^\ell_{p,2}}^2]$. By \emph{\ref{K1}}, 
the  bias term is controlled, up to a multiplicative constant, by $RD^{-(s-\ell)}$. 
We bound the variance with the help of Rosenthal inequality.
\begin{lemma}[Rosenthal inequality]\label{lem:rosenthal}
Let $m\geq 2$. 
 There exists a constant $C_m$ such that for all sequences $Y_1,\dots,Y_n$ of i.i.d.~random variables with $\E[|Y_1|^m]<+\infty$ and $\E[Y_1]=0$, it holds that
\begin{equation}
\E\left[ \left|\frac 1n \sum_{i=1}^n Y_i\right|^m \right] \leq C_m \p{ n^{-\frac m2}\E[|Y_1|^2]^{\frac m2} + n^{1-m} \E[|Y_1|^m]}.
\end{equation}
\end{lemma}
    We have $\pi_D(\mu_n-f)= \sum_{i=1}^N \sum_{j=0}^J \sum_{\psi\in \Psi^{(i)}_j} \dotp{\mu_n-f,\check{\psi}^{(i)}}\psi^{(i)}$ where $2^{Jd}$ is of order $D$. Then by Lemma \ref{lem:eqnorm} and Jensen's inequality, $ \E[ \|\pi_D(\mu_n-f)\|_{\B^\ell_{p,2}}^2]$ is controlled up to a multiplicative constant by
    \begin{equation*}
           \sum_{i=1}^N\sum_{j=0}^J 2^{2j\p{\ell+\frac d2-\frac dp}}\p{\sum_{\psi^{(i)}\in \Psi^{(i)}_j} \E[|\dotp{\mu_n-f,\check\psi^{(i)}}|^p]}^{\frac 2p}.
    \end{equation*}
    We apply Rosenthal inequality to each expectation to obtain the bound
    \begin{align*}
    \E[|\dotp{\mu_n-f,\check\psi^{(i)}}|^p]&\lesssim n^{-\frac p2}\|\check\psi^{(i)}\|_{\L^2}^{\frac p2}+ n^{1-p}\|\check\psi^{(i)}\|_{\L^p}^{p}\\
    &\lesssim n^{-\frac p2}+ n^{1-p}2^{jp\p{\frac d2-\frac dp}}.
    \end{align*}
Assume that $2\ell >-d$. Then, as the number of nonzero functions in $ \Psi^{(i)}_j$ is of order $2^{jd}$, 
    \begin{align*}
   \E[ \|\pi_D(\mu_n-f)\|_{\B^\ell_{p,2}}^2] &\lesssim n^{-1}\sum_{j=0}^J 2^{j\p{2\ell+d}} + n^{\frac 2p-2} \sum_{j=0}^J 2^{2j(\ell+d\p{1-\frac 1p})}\\
   & \lesssim n^{-1}D^{2\ell+d} + n^{\frac 2p-2}D^{2\ell+2d\p{1-\frac 1p}}.
    \end{align*}
    As $D^d/n\leq 1$ and $p\geq 2$, the second term is negligible. This shows that, for $D$ of order $n^{\frac 1{2s+d}}$,  $\E[\|\pi_D\mu_n - \pi_Df\|^2_{\B^{\ell}_{p,2}}]^{1/2}\lesssim n^{-\frac{s-\ell}{2s+d}}$. The bias term is of the same order for this value of $D$, concluding the proof in the case  $2\ell >-d$. Similar computations can be made in the two other cases.

        It remains to prove the second inequality.
        We first treat the case $m=1$. By Lemma \ref{lem:eqnorm}, the fluctuations are  of order
        \begin{equation}
             \E[ \|\pi_D(\mu_n-f)\|_{\cC^\ell}] \lesssim \sum_{i=1}^N\E[\max_{0\leq j\leq J, \psi\in \Psi^{(i)}_j} 2^{j\p{\ell+\frac d2}}|\dotp{\mu_n-f,\check\psi^{(i)}}|].
        \end{equation}
        This is the expectation of the supremum of an empirical process. Remark that for $0\leq j\leq J$ and $ \psi^{(i)}\in \Psi^{(i)}_j$, $2^{j\p{\ell+\frac d2}}\|\check\psi^{(i)}\|_{\L^\infty}\lesssim D^{\ell+d}$ and that $2^{j\p{\ell+\frac d2}}\|\check\psi^{(i)}\|_{\L^2}\lesssim D^{\ell+d/2}$. Thus, as the supremum is taken over a set of size of order $D^d$, a direct application of \cite[Lemma 3.5.12]{gine2015mathematical} yields
           \begin{equation}
             \E[ \|\pi_D(\mu_n-f)\|_{\cC^\ell}] \lesssim n^{-1/2}D^{\ell+d/2}\sqrt{\log D}+n^{-1} D^{\ell+d}\log D.
        \end{equation}
           Thus, for $m=1$, we obtain the conclusion  by letting $D$ be of order $(n/\log n)^{\frac 1{2s+d}}$.
       According to \cite[Theorem 2.14.23]{vaart2023empirical},  the exponential Orlicz norm of the supremum $\max_{ 0\leq j\leq J, \psi\in \Psi^{(i)}_j}|\dotp{\mu_n-f,\check\psi^{(i)}}|$ is at most of order
        \begin{equation*}
         \E[ \max_{ \psi\in \Psi^{(i)}_j}|\dotp{\mu_n-f,\check\psi^{(i)}}|] + n^{-1}(1+\log n) D^{\ell+d}. 
        \end{equation*}
As moments of a random variable are bounded by its exponential Orlicz norm, we find that  for $m\geq 1$, $    \E[ \|\pi_D(\mu_n-f)\|_{\cC^\ell}^m]^{\frac 1m} \lesssim \p{(\log n)/n}^{\frac{s-\ell}{2s+d}}$.
\end{proof}

The estimator $\pi_D \mu_n$ can be easily modified so that it belongs almost surely to a prescribed statistical model $\Theta$. This simple modification will turn out to be more convenient to use.

\begin{corollary}\label{cor:minimax_standard}
Let $m\geq 1$ be  integer and let $p\geq 2$, $R,\eta>0$, $0<s'<s\leq s_\star$. Let $\Theta_0\subset \Theta$ be  subsets of $\cC^s$ such that $\Theta$ is included in a $\eta$-neighborhood of $\Theta_0$ for the $\cC^{s'}$-norm. There exists  an estimator $\hat f\in \Theta$ based on $n$ i.i.d.~observations, such that if the density $f$ of the observations is in $f\in \Theta_0$ and satisfies $\|f\|_{\cC^s}\leq R$, then $\hat f$ satisfies the same rates of convergence as in \Cref{prop:minimax_standard}.
\end{corollary}
\begin{proof}
We define $\hat f=\pi_D\mu_n$ if $\pi_D\mu_n\in \Theta$, and we let $\hat f$ be a fixed, deterministic function in $\Theta$ should this condition be not verified. Remark that $\hat f\in \Theta$ and that  $\hat f=\pi_D\mu_n$ as long as $\|\pi_D\mu_n-f\|_{\cC^{s'}}\leq \eta$. Let us call $E$ this event. By Markov's inequality and \Cref{prop:minimax_standard}, $\P(E^c)\leq C_k(\log n/n)^{k\frac{s-s'}{2s+d}}$ for all $k\geq 1$ and some constant $C_k$. Furthermore, we use \emph{\ref{K4}} to show that $\|\pi_D\mu_n\|_{\cC^{\ell}}$  can be bounded almost surely by $\sup_{x}  \|\pi_{D}\delta_x\|_{\cC^{\ell}}\leq C_{\star}\kappa_{\infty,d,\ell}(D)$, a quantity which grows at most polynomially with $D$ (and thus with $n$). These two ingredients are enough to show that the residual term $\E[\|\hat f-f\|_{\cC^\ell}^m\ones\{E^c\}]$ is negligible in front of $\E[\|\hat f-f\|_{\cC^\ell}^m\ones\{E\}]=\E[\|\pi_D\mu_n -f\|_{\cC^\ell}^m\ones\{E\}] $. We then conclude with \Cref{prop:minimax_standard}. The risk of $\E[\|\hat f-f\|_{\W^\ell_p}^2]$ is controlled in a similar fashion.
\end{proof}

\begin{remark}\label{remark:moment_control_density}
If $s$ is an integer and $f\in \Cf^s$ with $\|f\|_{\Cf^s}\leq L$, one can easily modify the previous proofs to see that one can also impose that for all $m\geq 1$, there is a constant $C>0$ such that $\E[\|\hat f\|_{\Cf^s}^m]^{1/m}\leq C$. This follows from the  injection $\B^s_{\infty,1}\hookrightarrow \Cf^s$.
\end{remark}

\section{Kernel estimates}\label{app:kernel_estimates}

In this section we give some estimates on the explosion of kernels of resolvent type operators.

\subsection{Resolvent-like kernels}

We say that a kernel 
$K \in \Cf^\infty(M \times M \setminus \Delta)$ is \emph{resolvent-like} if  for each $\alpha \in \N\times \N$ and $\beta \in \N^d\times \N^d$, with $m=d-2-\alpha_1-\alpha_2+|\beta_1|+|\beta_2|$, there exists $C_{\alpha,\beta}$ such that for all $x,y\in M,\ x\neq y$,  
\begin{equation}\label{eq:def_resolvent_kernel}
|\Lambda_x^{-\alpha_1}\Lambda_y^{-\alpha_2} \partial^{\beta_1}_x \partial^{\beta_2}_y K(x,y)| \leq C_{\alpha,\beta}\begin{cases}
|x-y|^{-m} &\text{ if } m>0 \\
1+|\log |x-y|| &\text{ if } m=0 \\
1&\text{ if } m<0.
\end{cases}
\end{equation}

\begin{lemma}\label{lem:kernel_estimates}
Let $u \in \Cf^1$ and let $K$ be a resolvent-like kernel.
 Consider $K_{1,u} = K$ and set for $n \in \N$,
$$
K_{n + 1,u}(x,y) = \int_{M} K(x,z) \nabla u(z) \cdot \nabla_z K_{n,u}(z, y) \dd z ,\quad  x,y\in M,\ x\neq y.
$$
Then $K_{n, u} \in \Cf^\infty(M\times M\setminus \Delta)$ is resolvent-like for each $n \in \N$, with constants in \eqref{eq:def_resolvent_kernel} depending only on the corresponding constants for $K$, on an upper bound on $\|u\|_{\Cf^1}$ and on $n$.
\end{lemma}
\begin{proof}
Using a partition of unity we may assume that $M = \R^d$ and $u$ is compactly supported. We proceed by induction on $n$. Assume that for some $n \in \N$, $K_{n,u}$ is resolvent-like. We have
$$
K_{n + 1,u}(x,y) = \sum_{j=1}^d \int_{\R^d} K(x,z) \partial_{z_j}u(z) \partial_{z_j} K_{n,u}(z, y) \dd z.
$$
We see that $K_{n + 1,u}$ is resolvent-like by the induction hypothesis and using standard estimates for Riesz kernels (see \cite[§V.1]{stein1970} or \cite[§1.1]{landkof1973}).
\end{proof}

\subsection{Resolvent of $\Delta_f$}
The resolvent of $\Delta$ is the Bessel potential, which is known to be resolvent-like \cite[Proposition 2.2]{taylor2013partial}. Likewise, we show that the resolvent of $\Delta_f$ is decomposed into the sum of an operator with a resolvent-like kernel and a regularizing operator.

\begin{lemma}\label{lem:decomposition_resolvent}
    Let $\delta,L>0$. 
There exists $C>0$ such that the following holds. Let $f\in \Cf^2$ be a function with $\inf f>\delta$ and $\|f\|_{\Cf^2}\leq L$ and let $z\in \C$ with $|z|\leq L$ and $\dist(z,\sigma(-\Delta_f))\geq \delta$. Then, there is a decomposition
    \[
R_f(z) = W_z + Q_z
    \]
    where $W_z$ has a resolvent-like kernel $K_z$ with constants in \eqref{eq:def_resolvent_kernel} depending only on $C$, 
    and $Q_z$ has an operator norm smaller than $C$ as an operator $\H^{-1}\to \cC^{1}$ and $\W^{-1,1}\to \H^1$.
\end{lemma}
\begin{proof}
Write $f = e^U$. 
For $n\in \N$, we use the resolvent identity (see \eqref{eq:devresolve})
\[
R_f(z) = \sum_{i=0}^n (-\alpha)^i (z+\Delta)^{-1} (A_U (z+\Delta)^{-1} )^i + (-\alpha)^{n+1}R_f(z) (A_U (z+\Delta)^{-1} )^{n+1}.
\]
\Cref{lem:kernel_estimates} indicates that for any $n\geq 1$, the operator $W_z = \sum_{i=0}^n (-1)^i (z+\Delta)^{-1} (A_U (z+\Delta)^{-1} )^i$ has a resolvent-like kernel. 
Furthermore, applying iteratively \Cref{lem:operator_norm_Ah} and \Cref{prop:resolvbounded} together with Sobolev embeddings show that the remainder term $Q_z= (-1)^{n+1}(z+\Delta_f)^{-1} (A_U (z+\Delta)^{-1} )^{n+1}$ maps $\H^{-1}\to \cC^{1}$ and $\W^{-1,1}\to \H^1$ for $n$ large enough (depending on $d$).
\end{proof}

%%%
\section{Control of Hilbert-Schmidt norms}\label{app:hilbert-schmidt}

In this section, we give some bounds on Hilbert-Schmidt norms of operators appearing in  the article.

Let $(\phi_k)_{k\geq 0}$ be an orthonormal basis of eigenfunctions of the Laplace operator $-\Delta$, with $(\lambda_k)_{k\geq 0}$ the associated sequence of eigenvalues. For $t\in \R$, let $\phi_{k,t} = (1+\lambda_k)^{-t/2}\phi_k$. By definition of the Sobolev space $\H^t$, $(\phi_{k,t})_{k\geq 0}$ is an orthonormal basis of $\H^t$. 

\subsection{Hilbert-Schmidt bounds for tri-bounded operators}

We first show that the  condition given in \Cref{remark:sufficient_G3} implies that the Hilbert-Schmidt condition in \emph{\ref{G3}} is satisfied.

\begin{lemma}\label{lem:sufficient_G3}
    If $G:(\L^2)^3\to \R$ is continuous, then the Hilbert-Schmidt condition in \emph{\emph{\ref{G3}}} is satisfied.
\end{lemma}
\begin{proof}
    Let $t>d/2$. We show that $G:\H^{t}\times  \H^{t}\times \L^2$ has a finite Hilbert-Schmidt norm. It holds that
    \begin{align*}
        \sum_{k_1,k_2,k_3\geq 0}& |G[\phi_{k_1,t},\phi_{k_2,t},\phi_{k_3}]|^2 =    \sum_{k_1,k_2,k_3\geq 0} (1+\lambda_{k_1})^{-t}(1+\lambda_{k_2})^{-t} |G[\phi_{k_1},\phi_{k_2},\phi_{k_3}]|^2.
    \end{align*}
    As $G: (\L^2)^3 \to \R$ is continuous, there exists $\widetilde G:(\L^2)^2 \to \L^2$ such that $G[u,v,w]=\dotp{\widetilde G[u,v],w}$. But then,
    \begin{align*}
        \sum_{k_3\geq 0}|G[\phi_{k_1},\phi_{k_2},\phi_{k_3}]|^2 = \|\widetilde G[\phi_{k_1},\phi_{k_2}]\|^2_{\L^2}\lesssim \|\phi_{k_1}\|^2_{\L^2}\|\phi_{k_2}\|^2_{\L^2}\lesssim 1.
    \end{align*}
    Thus, the sum $\sum_{k_1,k_2,k_3\geq 0} |G[\phi_{k_1,t},\phi_{k_2,t},\phi_{k_3}]|^2$ is of order at most 
    $$
    \sum_{k_1,k_2\geq 0} (1+\lambda_{k_1})^{-t}(1+\lambda_{k_2})^{-t}.
    $$
    However Weyl's law yields $\lambda_k\sim Ck^{2/d}$, hence this sum is finite.
\end{proof}

\subsection{Hilbert-Schmidt bounds for resolvent-like kernels}
We write $|x-y|$ for the geodesic distance between two points $x,y\in M$. We let $\Delta = \{(x,x):\  x\in M\}$ be the diagonal of $M\times M$. If $K$ is a distribution on $M\times M$ and $\Gamma$ is an operator, we let $\Gamma_y K(x,y)$ be the distribution obtained by applying $\Gamma$ to the variable $y$ of $K$, evaluated at $(x,y)$.
\begin{lemma}\label{lem:G3_from_kernel}
Let $t>d/2$ and $C>0$. 
Let $K_1,K_2,K_3$ be  distributions on $M \times M$ such that for $i=1,2,3$, $(x,y)\mapsto \Lambda_y^{-t}K_i(x,y)$ is continuous on $(M \times M) \setminus \Delta$ and such that
$$
|\Lambda_y^{-t} K_i(x,y)|\leq C(1+|x-y|^{-d+t}).
$$
Let $B_i$ be the operators with kernel $K_i$, that we assume have operator norms $\L^2\to \L^2$ smaller than $C$. Then, the symmetrization of the operator $ \H^t\times \H^t\times \L^2 \to \R$ given by $(u_1,u_2,u_3)\mapsto \int (B_1u_1)\cdot (B_2u_2)\cdot(B_3u_3)$ has finite Hilbert-Schmidt norm, bounded by $C_1C^3$ for some constant $C_1$ depending on $t$ and $d$.
\end{lemma}
\begin{proof}
 The Hilbert-Schmidt norm of the (nonsymmetric) operator is equal to the sum
     \begin{align*}
        \sum_{k_1,k_2,k_3\geq 0} |\dotp{\phi_{k_1}, B_1^*\left[B_2\phi_{k_2,t}B_3\phi_{k_3,t}\right]}|^2 &=    \sum_{k_2,k_3\geq 0} \| B_1^*\left[B_2\phi_{k_2,t}B_3\phi_{k_3,t}\right]\|_{\L^2}^2 \\
        &\leq C^2 \sum_{k_2,k_3\geq 0} \int |B_2\phi_{k_2,t}(x) B_3\phi_{k_3,t}(x)|^2 \dd x \\
        &= C^2 \int \rho_2(x)\rho_3(x) \dd x
    \end{align*}
    where  $\rho_i(x) = \sum_{k\geq 0}| B_i\phi_{k,t}(x)|^2$ for $i=2,3$. 
    We have
    \[
    \rho_i(x) = \|\Lambda_y^{-t}K(x,\cdot)\|^2_{\L^2}\leq C^2\int (1+|x-y|^{-d+t})^2 \dd y \leq C^2 C_{d,\eps}
    \]
    for some constant $C_{d,t}$ depending on $d$ and $t$. Thus, $\int \rho_2(x)\rho_3(x) \dd x$ is of order $C^4$. This proves  that the Hilbert-Schmidt norm of $(u_1,u_2,u_3)\in \L^2\times \H^t\times \H^t\mapsto \int (B_1u_1)\cdot (B_2u_2)\cdot(B_3u_3)$ is finite. The other symmetrizations are treated similarly.
\end{proof}

We can in particular apply this lemma with $B_1=B_2=B_3=1$. Indeed, for $t>d/2$, $\Lambda_y^{-t}\delta(x-y)$ is the Bessel potential of order $t$, which verifies
\[
|\Lambda_y^{-t}\delta(x-y)|\lesssim 1+|x-y|^{-d+t},
\]
see \cite{bessel}. 
As $-d+t>-d/2$, we may apply \Cref{lem:G3_from_kernel}, which proves that the trilinear form $(u_1,u_2,u_3)\mapsto \int u_1u_2u_3$ satisfies the Hilbert-Schmidt condition in \emph{\ref{G3}}.

\subsection{Control of Hilbert-Schmidt norms for differentials of eigenvalues}\label{subsec:endproofeiganvalue}
In this subsection we give the missing details in the proof of Proposition \ref{prop:eigenvalue_perturbative}. More precisely, we explain here how to bound the Hilbert-Schmidt norms of 
\[
 G_1[u_1u_2u_3],\ G_2[u_1u_2,u_3],\ G_2[u_1u_3,u_2],\ G_2[u_2u_3,u_1],\ G_3[u_1,u_2,u_3]
\]
over $\H^t\times \H^t\times \L^2$ for $t>d/2$ (see \Cref{sec:asymptotic_expansion}).
By definition of the $G_j$'s, it suffices to bound the Hilbert-Schmidt norm of the symmetrization of the forms
\begin{align}
 \dotp{B_z[u_1u_2u_3]\phi,\phi}_f,
\quad \dotp{B_z[u_1u_2]B_z[u_3]\phi,\phi}_f \quad 
 \text{and } \quad \dotp{B_z[u_1]B_z[u_2]B_z[u_3]\phi,\phi}
\end{align}
for $z\in \Upsilon$ and $\phi$ an eigenfunction of $\Delta_f$. We denote the symmetrizations of these three terms by $H_1(u_1, u_2, u_3)$, $H_{2}(u_1, u_2, u_3)$ and $H_3(u_1, u_2, u_3)$ respectively.\\

\noindent \textbf{Control of $H_1$}.
First, an integration by parts show that the trilinear form $H_1$ can be written as 
$$
H_1(u_1, u_2, u_3) = \langle u_1 u_2 u_2, \psi_1 \rangle
$$
for some $\psi_1 \in \cC^s$, so that the control of the Hilbert-Schmidt norm of this form  follows from \Cref{lem:example_form}. \\

\noindent \textbf{Control of $H_{2}$}.
 Define $\Gamma=A_{\overline \phi}-\overline \phi z$, which is a differential operator of order $1$ with $\mathscr C^s$ coefficients.
The computations in the proof of \Cref{lem:sym_Sj_map} show that the symmetrization of $\dotp{ B_z[u]B_z[v]\phi,\phi}$ can be written as a linear combination of the symmetrization of
\begin{equation}\label{eq:decompc2}
\dotp{u v,\psi_2}_f \ \text{ and }\ \dotp{\Gamma R_f(z) \nabla v \nabla \phi,u}_f
\end{equation}
for some function $\psi_2\in \cC^s$. Hence we can write $H_{2}(u_1, u_2, u_3)$ as a linear combination of the symmetrization of
$$
\dotp{u_1u_2 u_3,\psi_2}_f  \quad \text{and} \quad \dotp{\Gamma R_f(z) \nabla u_3 \nabla \phi,u_1u_2}_f
$$
Applying again \Cref{lem:example_form}, we see that the trilinear form $\dotp{u_1u_2u_3,\psi_2}_f $ has finite Hilbert-Schmidt norm. One decomposes the second term of \eqref{eq:decompc2} as
\[
\dotp{\Gamma R_f(z) \nabla v \nabla \phi,u}_f=\dotp{\Gamma W_z \nabla v \nabla \phi,u}+\dotp{\Gamma Q_z \nabla v \nabla \phi,u}_f.
\]
The mapping properties of $Q_z$ stated in \Cref{lem:decomposition_resolvent} imply that the symmetrization of the trilinear form $\dotp{\Gamma Q_z \nabla u_3 \nabla \phi,u_1u_2}_f$ is continuous over $(\L^2)^3$ so its Hilbert-Schmidt norms over $\H^t\times \H^t \times \L^2$ is bounded according to \Cref{lem:sufficient_G3}. On the other hand, the kernel of the operator $P:v\mapsto \Gamma W_z \nabla v \nabla \phi$ is equal to
\begin{align*}
   - \Gamma_x \div_y(K_z(x,y) \nabla\phi(y)),\quad (x,y) \in (M\times M) \setminus \Delta.
\end{align*}
As the kernel $K_z$ is resolvent-like and $\Gamma$ is a differential operator of order $1$, we have  
$$
|\Lambda_y^{-t} \Gamma_x \div_y(K_z(x,y) \nabla\phi(y))| \leqslant C|x-y|^{-d+t}
$$
As $P$ is also bounded $\L^2\to \L^2$ according to \Cref{prop:resolvbounded}, we may apply \Cref{lem:G3_from_kernel}, which implies that the symmetrization of the map 
\begin{equation}\label{eq:map}
(u_1,u_2,u_3)\mapsto \dotp{Pu_3,u_1u_2}_f 
\end{equation}
has a finite Hilbert-Schmidt norm over $\H^t\times \H^t\times \L^2$. 

\noindent \textbf{Control of $H_3$}. At last, we must bound the Hilbert-Schmidt norm of $H_3$. By applying two integration by parts as in the proof of \Cref{lem:sym_Sj_map}, we find that the symmetrization $H_3(u_1, u_2, u_3)$ of $\dotp{B_z[u_1]B_z[u_2]B_z[u_3]\phi,\phi}$ can be written as a linear combination of 
\begin{equation}\label{eq:first_term}
\dotp{u_1u_2u_3, \psi_3}_f
\end{equation}
for some $\psi_3\in\cC^s$ 
and of the symmetrization of 
\[
\dotp{\Gamma R_f(z) \nabla u_3\cdot \nabla\phi, u_1u_2}_f\quad \text{ and } \quad \dotp{\Gamma R_f(z) \nabla u_2\cdot \nabla R_f(z) \nabla u_3\cdot\nabla \phi ,u_1}_f.
\]
The form defined in \eqref{eq:first_term} satisfies the second condition in \emph{\ref{G2}} according to \Cref{lem:example_form}. The second term was already studied when we bounded the Hilbert-Schmidt norm of $H_2$.  Two additional integration by parts are enough to write the symmetrization of the last term as a linear combination of the symmetrizations of
\begin{equation}
    \dotp{u_2, \nabla (R_f(z)\Gamma^* u_1)\cdot \nabla R_f(z) \nabla u_3\cdot\nabla \phi}_f,\quad \dotp{u_2, (R_f(z)\Gamma^* u_1)\cdot R_f(z) \nabla u_3\cdot\nabla \phi}_f
\end{equation}
and 
\[
\dotp{u_1, \Gamma R_f(z) \nabla (u_2u_3)\cdot\nabla \phi}_f,
\]
where $\Gamma^*$ is the adjoint of $\Gamma$ in $\L^2(f^\alpha)$.
For each of these terms, we decompose  each instance of $R_f(z)$   into $W_z+Q_z$ and expand the product. The terms containing at least one $Q_z$ are shown to be bounded over $(\L^2)^3$ thanks to the regularizing properties of $Q_z$ stated in \Cref{lem:decomposition_resolvent}, together with \Cref{lem:operator_norm_Ah} and \Cref{prop:resolvbounded}. We may therefore replace $R_f(z)$ by $W_z$ in each of these terms. But then, we show as we did for $H_2$ that the corresponding terms satisfy the assumptions of \Cref{lem:G3_from_kernel}, hence they have finite Hilbert-Schmidt norm over $\H^t \times \H^t \times \L^2$.

This completes the proof of Proposition \ref{prop:eigenvalue_perturbative}.

\small
\bibliography{biblio}
\bibliographystyle{alpha}

\end{document}